\documentclass[11pt, letterpaper]{article}

\usepackage[utf8]{inputenc}
\usepackage[T1]{fontenc}
\usepackage[english]{babel}
\usepackage{csquotes}

\usepackage{amsmath, amsthm, amssymb}
\usepackage{ytableau}      
\usepackage{ragged2e}
\usepackage{multicol}
\usepackage[rgb]{xcolor}

\usepackage{graphicx, rotating, float}
\usepackage{tikz}
\usetikzlibrary{babel}
\usetikzlibrary{arrows, bending, positioning, chains}
\usetikzlibrary{automata, shapes.geometric}
\usetikzlibrary{calc, arrows.meta}

\tikzset{
    myedge/.style={->, -{Latex[#1]}}
}

\usepackage[margin=1in]{geometry}
\usepackage{authblk}

\newcommand{\R}{\mathbb{R}}
\newcommand{\Z}{\mathbb{Z}}

\theoremstyle{plain}
\newtheorem{theorem}{Theorem}[section]
\newtheorem{lemma}[theorem]{Lemma}
\newtheorem{proposition}[theorem]{Proposition}
\newtheorem{corollary}[theorem]{Corollary}

\theoremstyle{definition}
\newtheorem{definition}[theorem]{Definition}
\newtheorem{example}[theorem]{Example}

\theoremstyle{remark}
\newtheorem{remark}[theorem]{Remark}

\usepackage[hidelinks]{hyperref}
\hypersetup{
    colorlinks=true,
    linkcolor=blue,
    citecolor=red,
    urlcolor=cyan,
    pdftitle={Weyl groups and the Kostant game},
    pdfauthor={Juan Sebastián Cortés Cruz}
}

\title{Weyl groups and the Kostant game}
\author{Juan Sebastián Cortés-Cruz\thanks{Email: \texttt{juascortescru@unal.edu.co}}}
\affil{Department of Mathematics, Universidad Nacional de Colombia}
\date{January 21, 2026}

\begin{document}

\maketitle

\begin{abstract}
    This paper establishes a novel combinatorial framework at the intersection of Lie theory and algebraic combinatorics, based on a generalization of the Kostant game. We begin by reviewing the foundations of root systems, the classification of Dynkin diagrams, and the structure of Weyl groups. Subsequently, we analyze the original Kostant game as a tool for generating positive roots, demonstrating its unique termination on simply-laced diagrams and its role in an alternative classification thereof.

    The main contribution of this work—which, to our knowledge, has not been studied before—is a \textbf{multi-vertex generalization} of the game that allows for the simultaneous modification of multiple vertices of a Dynkin diagram. We prove that the resulting configurations of this new game establish a natural bijection with the elements of the quotient $W/W_J$ of Weyl groups by parabolic subgroups. This formalism is applied to problems in algebraic geometry, specifically addressing cases of the Mukai conjecture via Hilbert polynomials, and is accompanied by a computational implementation in Java. These results offer new combinatorial perspectives for studying root counting problems, the regularity of reduced word languages, and the construction of Young Tableaux.

    \vspace{0.5cm}
    \noindent \textbf{Keywords:} Kostant game, Weyl groups, Root systems, Dynkin diagrams, Parabolic subgroups.
\end{abstract}

\section{Introduction}
\label{sec:intro}

Since the late 19th century, the study of symmetry groups has played a fundamental role in the advancement of mathematics and theoretical physics. The notion of symmetry is not only essential in areas such as geometry and number theory, but it is also crucial for understanding physical phenomena at quantum and cosmological levels. The emergence of Lie groups marked a milestone in this exploration, providing a language to describe the continuous symmetries that govern our universe. It was Sophus Lie who, in the 1870s, initiated the systematic study of these groups to address the integration of differential equations \cite{lie1880}. Although Lie laid the foundations, he only captured part of the complete picture.

The challenge of classifying the Lie algebras associated with these groups was addressed by Wilhelm Killing in 1888, who, in his seminal work \cite{killing1888}, successfully identified the infinite families of type $A_n$, $B_n$, $C_n$, and $D_n$. However, the exceptional cases, due to their elusive nature, presented insurmountable difficulties for him. The task was masterfully completed in 1894 by Élie Cartan. In his doctoral thesis \cite{cartan1894}, Cartan not only corrected and completed Killing's work, proving the existence of the five exceptional algebras ($G_2$, $F_4$, $E_6$, $E_7$, and $E_8$), but also introduced revolutionary techniques, such as the decomposition of compact groups via maximal tori, which allowed us to peer into the internal structure of these objects.

Progress continued with Hermann Weyl, who in 1939 \cite{weyl1939} made two key contributions: he proved that every semisimple compact group has a universal cover (like the iconic projection of $\text{SU}(2)$ onto $\text{SO}(3)$) and established the indelible connection between these groups and root systems. This connection was a fundamental result, revealing that the complex analytic structure of a Lie group could be completely determined by a discrete combinatorial object. Finally, in 1947, Eugene Dynkin \cite{dynkin1947} culminated this line of research by introducing his famous diagrams. With astounding simplicity and elegance, these graphs managed to unify and classify all root systems, providing an invaluable visual tool that remains central to the study of Lie algebras today.

This historical journey leaves us at a point of profound theoretical beauty, but it also poses new challenges. The classification is complete, but how do we work with it effectively? Is it possible to construct root systems and their structure in an algorithmic and combinatorial manner, without resorting to an exhaustive case-by-case enumeration?

The Kostant game, introduced by Bertram Kostant and developed in studies of algebraic combinatorics \cite{chen2017}, emerges as a promising answer to these questions. This game offers a dynamic and constructive perspective for generating the positive roots of a system, dynamically manifesting the action of the Weyl group. However, the classical game, in its elegance, has an inherent limitation: it is designed to explore the root system in its entirety. But what happens when our interest focuses on finer substructures?

This is precisely the problem that motivates the present work. Our research focuses on a \textbf{modified version of the Kostant game}, a tool that emerged unexpectedly in the study of the Mukai conjecture in algebraic geometry \cite{CaviedesCastro2022}. It soon became evident that this tool, born as a means to an end, possessed intrinsic mathematical richness and deserved to be studied in its own right.

The central modification proposed here consists of generalizing the game to allow the simultaneous alteration of multiple vertices in a Dynkin diagram. This generalization is non-trivial; it reveals deep and unexplored connections with the quotients of Weyl groups by parabolic subgroups, denoted as $W/W_J$. These quotients are not merely algebraic curiosities; they are fundamental objects that parameterize geometric structures of great importance, such as flag varieties.

Thus, this work aims to establish a complete characterization of the configurations arising from this generalized game. The most significant result is the proof of a \textbf{bijective correspondence} between the allowed sequences of moves and the elements of the quotient $W/W_J$. This bridge between the dynamic combinatorics of the game and the algebra of Weyl groups offers a novel perspective and an algorithmic method to understand, visualize, and work with these structures, simplifying fundamental aspects of the classification theorem.

To achieve this goal, the paper is structured as follows. Section \ref{Chapter1} establishes the necessary theoretical scaffolding: root systems, Dynkin diagrams, and Weyl groups. Section \ref{Chapter2} analyzes the original Kostant game in detail, its dynamic properties, and its use in the classification of Dynkin diagrams.

Section \ref{Chapter3} presents the central contribution: the modified version of the game. We establish its rules, explore its properties, and rigorously prove the bijection with Weyl quotients. Finally, we explore the powerful consequences of this connection, demonstrating its application to the Mukai conjecture, the regularity of formal languages, and the construction of Young Tableaux. Finally, the section \ref{java} complements this work with a description of the computational implementation in Java that accompanies this research.

\section{Root systems}
\label{Chapter1}

We begin this work with a motivation born from Lie theory, where root systems are fundamental combinatorial structures in the study of semisimple Lie algebras and Lie groups. These systems allow for the classification of Lie algebras (see \cite{Kirillov2005CompactGroups} to explore this classification) and are intimately related to Weyl groups, which are symmetry groups generated by reflections of these very roots. In this section, we will develop the theory of root systems in detail, including examples and key properties.

\subsection{Abstract root systems}
\label{sec1.1}

The study of continuous symmetries, encapsulated in Lie group theory, culminates in a remarkable reduction principle: the complex analytic structure of a compact and connected Lie group is highly determined by a discrete combinatorial object—its root system. This section is dedicated to developing the theory of root systems and their symmetry groups, the Weyl groups, from their foundations, following the rigorous tradition of Humphreys \cite{Humphreys1972} and the modern perspective of Kirillov Jr. \cite{Kirillov2008}. Our goal is not merely to catalog definitions, but to build a conceptual framework in which each element serves as a fundamental piece for understanding the dynamics of the Kostant game, which will be presented in later sections as a combinatorial manifestation of this deep algebraic structure.

\begin{definition}[Root system]
\label{def:sistema_raices}
Let $E$ be a finite-dimensional Euclidean vector space with an inner product $( \cdot, \cdot )$. A finite subset $\Phi \subset E \setminus \{0\}$ is a \textbf{reduced root system} in $E$ if it satisfies the following axioms:
\begin{enumerate}
    \item $\Phi$ spans $E$ linearly.
    \item If $\alpha \in \Phi$, the only multiples of $\alpha$ in $\Phi$ are $\alpha$ and $-\alpha$.
    \item For each $\alpha \in \Phi$, the reflection $s_\alpha$ with respect to the hyperplane $H_\alpha$ orthogonal to $\alpha$ maps $\Phi$ to itself. This reflection is given by the formula:
        $$ s_\alpha(v) := v - 2 \dfrac{( \alpha, v )}{( \alpha, \alpha )} \alpha, \quad \text{for } v \in E. $$
    \item If $\alpha, \beta \in \Phi$, then the Cartan integer defined as $n_{\alpha\beta} := 2 \dfrac{( \alpha, \beta )}{( \beta, \beta )}$ is an integer.
\end{enumerate}
The elements of $\Phi$ are called \textbf{roots}, and if $\dim{E}=n$, we say that $\Phi$ has rank $n$.
\end{definition}

\begin{remark}[On the axioms]
\leavevmode
\begin{itemize}
    \item The second axiom is a normalization condition; when $\Phi$ does not satisfy this axiom, it is simply called a root system (non-reduced).
    \item The third axiom introduces the set of fundamental symmetries of the system that will later form the Weyl group. To visualize these symmetries, one can consider the graphical representation shown in Figure \ref{fig:refle}.
    
    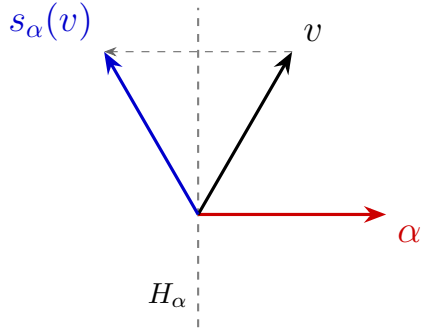
\begin{figure}[H]
    \centering
    \begin{tikzpicture}[
    scale=2.5,
    vector/.style={-Stealth, very thick},
    reflected/.style={-Stealth, very thick, blue!80!black},
    original/.style={-Stealth, very thick, red!80!black},
    hyperplane/.style={dashed, gray, thick},
    construction/.style={-Stealth, dashed, black!60},
    label/.style={font=\Large}
    ]
    
    \coordinate (O) at (0,0);
    \coordinate (alpha) at (1,0);
    \coordinate (v) at (0.5, {sqrt(3)/2}); 
    \coordinate (reflected_v) at (-0.5, {sqrt(3)/2});

    \draw[hyperplane] (0,-0.6) -- (0,1.1) node[pos=0.1, left, black] {$H_\alpha$};

    \draw[original] (O) -- (alpha) node[label, below right] {$\alpha$};
    \draw[vector] (O) -- (v) node[label, above right] {$v$}; 
    \draw[reflected] (O) -- (reflected_v) node[label, above left] {$s_\alpha(v)$};

    \draw[construction] (v) -- (reflected_v) node[midway, above] {};

    \end{tikzpicture}
    \caption{Visualization of a reflection $s_\alpha(v)$.}
    \label{fig:refle}
    \end{figure}
    
    \item The fourth axiom is the most restrictive and profound condition; it imposes a strong constraint on the possible angles between two roots. If $\theta$ is the angle between $\alpha$ and $\beta$, then $4 \cos^2\theta = \left(2 \dfrac{( \beta, \alpha )}{( \alpha, \alpha )}\right) \left(2 \dfrac{( \alpha, \beta )}{( \beta, \beta )}\right)$ must be an integer, which severely limits the possible geometry. Root systems satisfying this property are called crystallographic.
\end{itemize}
\end{remark}

For simplicity, we will work with root systems satisfying the second and fourth axioms, and unless stated otherwise, all root systems will be reduced and crystallographic. Furthermore, to simplify notation, we introduce the concept of co-root, which will be central to our subsequent work.

\begin{definition}[Co-root]
For each root $\alpha \in \Phi$, its associated \textbf{co-root} is the vector $\alpha^\vee \in E$ defined by:
$$ \alpha^\vee := \dfrac{2\alpha}{( \alpha, \alpha )}. $$
With this notation, the Cartan integer of two roots $\alpha$ and $\beta$ is written compactly as $n_{\alpha \beta} =( \alpha, \beta^\vee )$, and the reflection associated with $\alpha$ as $s_\alpha(v) = v - ( v, \alpha^\vee ) \alpha$ for $v \in E$.
\end{definition}

The set of all co-roots, $\Phi^\vee = \{\alpha^\vee \mid \alpha \in \Phi\}$, forms a root system in $E$ in its own right, called the \textbf{dual root system}. The distinction between a root system and its dual will be crucial when analyzing non-symmetric Dynkin diagrams later, such as types $B_n$ and $C_n$.

From the above, we note that the most important information about the root system is contained in the numbers $ n_{\alpha \beta} $. This motivates the following definition.

\begin{definition}[Root system isomorphism]
\label{def:isoraices}
Let $ \Phi_1 \subset E_1 $ and $\Phi_2 \subset E_2 $ be two root systems. A \textbf{root system isomorphism} $ \varphi : \Phi_1 \to \Phi_2 $ is a vector space isomorphism $\varphi : E_1 \to E_2 $ such that $\varphi(\Phi_1) = \Phi_2$ and $ n_{\varphi(\alpha)\varphi(\beta)} = n_{\alpha \beta}$ for any $\alpha, \beta \in \Phi_1. $

In particular, when $\varphi: \Phi \to \Phi$, it is called an \textbf{automorphism} of $\Phi$, and the set of all such automorphisms is denoted by Aut$(\Phi)$.
\end{definition}

\begin{remark}
Note that the condition $n_{\varphi(\alpha)\varphi(\beta)} = n_{\alpha \beta}$ is automatically satisfied if $\varphi$ preserves the inner product. However, not every root system isomorphism preserves the inner product. For example, for any $ c \in \mathbb{R}^+ $, the root systems $ \Phi $ and $ c\Phi = \{c\alpha \mid \alpha \in \Phi\} $ are isomorphic. The isomorphism is given by $ v \mapsto cv $, which does not preserve the inner product.
\end{remark}

To begin visualizing root systems, let us work on a simple-to-construct example: the case of rank 2 root systems. Let $\Phi$ be a reduced root system of rank 2. The following theorem shows that the conditions of axioms 2 and 3 impose very strong restrictions on the relative position of two roots.

\begin{theorem}
\label{theorem:angulosposibles}
Let $\alpha, \beta \in \Phi$ be roots such that neither is a multiple of the other, with $|\alpha| \geq |\beta|$, and let $\theta$ be the angle between them. Then one of the following possibilities must hold:
\begin{itemize}
  \item[(1)] $\theta = \pi/2$ (i.e., $\alpha$ and $\beta$ are orthogonal), $n_{\alpha\beta} = n_{\beta\alpha} = 0$.
  \item[(2a)] $\theta = 2\pi/3$, $|\alpha| = |\beta|$, $n_{\alpha\beta} = n_{\beta\alpha} = -1$.
  \item[(2b)] $\theta = \pi/3$, $|\alpha| = |\beta|$, $n_{\alpha\beta} = n_{\beta\alpha} = 1$.
  \item[(3a)] $\theta = 3\pi/4$, $|\alpha| = \sqrt{2}|\beta|$, $n_{\alpha\beta} = -2$, $n_{\beta\alpha} = -1$.
  \item[(3b)] $\theta = \pi/4$, $|\alpha| = \sqrt{2}|\beta|$, $n_{\alpha\beta} = 2$, $n_{\beta\alpha} = 1$.
  \item[(4a)] $\theta = 5\pi/6$, $|\alpha| = \sqrt{3}|\beta|$, $n_{\alpha\beta} = -3$, $n_{\beta\alpha} = -1$.
  \item[(4b)] $\theta = \pi/6$, $|\alpha| = \sqrt{3}|\beta|$, $n_{\alpha\beta} = 3$, $n_{\beta\alpha} = 1$.
\end{itemize}
\end{theorem}

\begin{proof}
Recalling that $n_{\alpha\beta} = 2 \dfrac{(\alpha, \beta)}{(\beta, \beta)}$ and considering that $(\alpha, \beta) = |\alpha||\beta|\cos\theta$, where $\theta$ is the angle between $\alpha$ and $\beta$, we have $
n_{\alpha\beta} = 2\dfrac{|\alpha|}{|\beta|} \cos\theta.$
Thus, $
n_{\alpha\beta}n_{\beta\alpha} = 4\cos^2\theta.$

Since $n_{\alpha\beta}n_{\beta\alpha} \in \mathbb{Z}$, this implies that this product must be one of the numbers $0, 1, 2$, or $3$. Analyzing these possibilities and using the fact that $
\dfrac{n_{\beta\alpha}}{n_{\alpha\beta}} = \dfrac{|\alpha|^2}{|\beta|^2}$ when $\cos\theta \neq 0$, the statement of the theorem follows.
\end{proof}

Using this theorem, the following corollary can be verified directly.

\begin{corollary}[Rank 2 root systems]
\leavevmode
\begin{itemize}
  \item[(1)] Let $A_1 \cup A_1, A_2, B_2, G_2$ be the sets of vectors in $\mathbb{R}^2$ shown in Figure \ref{fig:1-1}. Then, each of them is a rank 2 root system.
  \item[(2)] Every reduced root system of rank 2 is isomorphic to one of the systems $A_1 \cup A_1, A_2, B_2, G_2$.
\end{itemize}
\end{corollary}

\begin{figure}[H]
\centering

\begin{tikzpicture}[scale=1]
\draw[->] (0,0) -- (1,0);
\draw[->] (0,0) -- (-1,0);
\draw[->] (0,0) -- (0,1);
\draw[->] (0,0) -- (0,-1);
\node at (0,-1.4) {$A_1 \cup A_1$. All angles are $\pi/2$, equal lengths.};
\end{tikzpicture}
\hspace{1cm}
\begin{tikzpicture}[scale=1]
\foreach \angle in {0,60,...,300}
  \draw[->] (0,0) -- ({cos(\angle)}, {sin(\angle)});
\node at (0,-1.4) {$A_2$. All angles are $\pi/3$, equal lengths.};
\end{tikzpicture}

\vspace{1cm}

\begin{tikzpicture}[scale=1]
    \foreach \ang in {0, 90, 180, 270}
    \draw[->] (0,0) -- ({cos(\ang)}, {sin(\ang)});
    
    \foreach \ang in {45, 135, 225, 315}
    \draw[->] (0,0) -- ({sqrt(2)*cos(\ang)}, {sqrt(2)*sin(\ang)});
    
    \node at (0,-1.8) {$B_2$. All angles $\pi/4$, lengths $1$ and $\sqrt{2}$.};
\end{tikzpicture}
\hspace{1cm}
\begin{tikzpicture}[scale=1]
    \foreach \ang in {0, 60, 120, 180, 240, 300}
    \draw[->] (0,0) -- ({cos(\ang)}, {sin(\ang)});

    \foreach \ang in {30, 90, 150, 210, 270, 330}
    \draw[->] (0,0) -- ({sqrt(3)*cos(\ang)}, {sqrt(3)*sin(\ang)});
    
    \node at (0,-2) {$G_2$. All angles $\pi/6$, lengths $1$ and $\sqrt{3}$.};
\end{tikzpicture}

\caption{Rank 2 root systems.}
\label{fig:1-1}
\end{figure}

\begin{example}[A more general example]
\label{ex:An}
Let $ \{e_i\}_{1 \leq i \leq n} $ be the standard basis of $ \mathbb{R}^n $, with the usual inner product given by $ (e_i, e_j) = \delta_{ij} $. We define the subspace:
$$
E = \left\{ (\lambda_1, \ldots, \lambda_n) \in \mathbb{R}^n \mid \sum_{i=1}^n \lambda_i = 0 \right\},
$$
and the set of roots:
$$
\Phi = \{ e_i - e_j \mid 1 \leq i, j \leq n, \ i \neq j \} \subset E.
$$
Then $ \Phi $ is a reduced root system in $E$.

Indeed, for $ \alpha = e_i - e_j $, the corresponding reflection $ s_\alpha : E \to E $ is given by the transposition of the entries $ i $ and $ j $:
$$
s_{e_i - e_j}(\ldots, \lambda_i, \ldots, \lambda_j, \ldots) = (\ldots, \lambda_j, \ldots, \lambda_i, \ldots).
$$

Clearly, $ \Phi $ is stable under such transpositions (and, more generally, under all permutations). Therefore, condition 3 of the root system axioms is satisfied.

Furthermore, since $ (\alpha, \alpha) = 2 $ for all $ \alpha \in \Phi $, condition 4 reduces to $ (\alpha, \beta) \in \mathbb{Z} $ for any $ \alpha, \beta \in \Phi $, which is immediate.

Finally, condition 1 is also evident, since $ \Phi \subset E \setminus \{0\} $ and $ \Phi $ is finite.

We conclude that $ \Phi $ is a root system of rank $ n - 1 $. For historical reasons, this root system is known as the \emph{root system of type $ A_{n-1} $} (the subscript corresponds to the rank of the system).
\end{example}

\begin{definition}[Reducible root systems]
Let $\Phi$ be a root system in a Euclidean space $E$. We say that $\Phi$ is a \textbf{reducible} root system if there exist non-empty subsets $\Phi_1$ and $\Phi_2$ of $\Phi$ such that:
    \begin{enumerate}
        \item $\Phi = \Phi_1 \cup \Phi_2$,
        \item $( \alpha, \beta ) = 0$ for all $\alpha \in \Phi_1$ and $\beta \in \Phi_2$, where $( \cdot, \cdot )$ denotes the inner product in $E$,
        \item $\Phi_1$ and $\Phi_2$ are root systems (in their respective spanned spaces).
    \end{enumerate}
\end{definition}

We say that $\Phi$ is an \textbf{irreducible} root system if it is not reducible, that is, if it cannot be decomposed into a disjoint union of two mutually orthogonal and non-empty root subsystems.

As an example of a reducible root system, we saw earlier the case of the rank 2 system $A_1 \cup A_1$. It can be verified that all other rank 2 root systems are irreducible.

\subsection{Simple and positive roots}
\label{sec1.2}

The complex structure of a root system can be efficiently reconstructed from a minimal subset of generators, known as simple roots. The construction of this subset, following Humphreys \cite{Humphreys1972}, is based on the choice of a polarization.

To introduce a polarization, we select a vector $t \in E$ that is not orthogonal to any root. This vector defines a partition of $\Phi$:
$$ \Phi^+ = \{\alpha \in \Phi \mid ( \alpha, t ) > 0\}, \quad \Phi^- = \{\alpha \in \Phi \mid ( \alpha, t ) < 0\}. $$
The set $\Phi^+$ is called the set of \textbf{positive roots} (with respect to $t$). Clearly, $\Phi = \Phi^+ \sqcup \Phi^-$ and $\Phi^- = -\Phi^+$. This partition of $\Phi$ is called a polarization of $\Phi$ with respect to the vector $t$.

\begin{definition}[Simple roots]
A root $\alpha \in \Phi^+$ is said to be \textbf{simple} with respect to a given polarization if it cannot be written as the sum of two roots in $\Phi^+$. The set of all simple roots is denoted by $\Delta$ and is commonly referred to as the base of $\Phi$.
\end{definition}

The following theorem, fundamental to the theory, establishes that this set of ``indivisible'' roots behaves as a basis, in the sense of vector space bases.

\begin{theorem}[Properties of the base]
\label{thm:propiedades_base}
Let $\Delta$ be the set of simple roots defined from a choice of $\Phi^+$.
\begin{enumerate}
    \item Every positive root $\beta \in \Phi^+$ can be written as a linear combination of simple roots with non-negative integer coefficients:
          $$ \beta = \sum_{\alpha_i \in \Delta} k_i \alpha_i, \quad k_i \in \Z_{\ge 0}. $$
    \item If $\alpha, \beta$ are two distinct simple roots in $\Delta$, then their inner product is non-positive: $( \alpha, \beta ) \le 0$.
    \item The set $\Delta$ is linearly independent and, therefore, forms a basis of the vector space $E$.
\end{enumerate}
\end{theorem}

\begin{proof}
(1) Suppose there exists at least one positive root that cannot be expressed as a non-negative integer linear combination of simple roots. Among all such roots, let $\beta$ be one for which the value $(\beta, t)$ is minimal (where $t$ is the vector defining the polarization for $\Phi^+$).

By assumption, $\beta$ cannot be a simple root ($\beta \notin \Delta$), since if it were, it would be trivially expressed as $1 \cdot \beta$, satisfying the condition.

Since $\beta$ is a non-simple positive root, by definition it must be the sum of two positive roots: $\beta = \beta_1 + \beta_2$ for some $\beta_1, \beta_2 \in \Phi^+$.

Since $\beta_1, \beta_2 \in \Phi^+$, we have $(\beta_1, t) > 0$ and $(\beta_2, t) > 0$. This implies that
$$
 (\beta_1, t) < (\beta_1, t) + (\beta_2, t) = (\beta, t) \quad \text{and} \quad (\beta_2, t) < (\beta, t). 
$$
Due to the minimality of $\beta$, both $\beta_1$ and $\beta_2$ must be writable as the required linear combination:
$$
 \beta_1 = \sum_{\alpha_i \in \Delta} c_i \alpha_i, \quad \text{with } c_i \in \mathbb{Z}_{\ge 0} 
$$
$$
 \beta_2 = \sum_{\alpha_i \in \Delta} d_i \alpha_i, \quad \text{with } d_i \in \mathbb{Z}_{\ge 0} 
$$
Adding these expressions, we obtain:
$$
 \beta = \beta_1 + \beta_2 = \sum_{\alpha_i \in \Delta} (c_i + d_i) \alpha_i. 
$$
Since $c_i$ and $d_i$ are non-negative integers, their sum $k_i = c_i + d_i$ is also a non-negative integer. This shows that $\beta$ can be written in the desired form, which contradicts our initial assumption.
Therefore, every positive root can be written as a linear combination of simple roots with non-negative integer coefficients.

(2) Let $\alpha, \beta \in \Delta$ be two distinct simple roots. Suppose, for the sake of contradiction, that $(\alpha, \beta) > 0$.

A fundamental result of root system theory establishes that if $\alpha$ and $\beta$ are roots and $(\alpha, \beta) > 0$, then $\alpha - \beta$ is also a root (see [\cite{Humphreys1972}, \S 9.4]). Thus, under our assumption, $\alpha - \beta \in \Phi$.

We have two possibilities for the root $\alpha - \beta$:
\begin{enumerate}
    \item If $\alpha - \beta \in \Phi^+$, then we can write $\alpha = (\alpha - \beta) + \beta$. This expresses the simple root $\alpha$ as the sum of two positive roots, which contradicts the definition of a simple root.
    \item If $\alpha - \beta \notin \Phi^+$, then its opposite, $\beta - \alpha$, must be in $\Phi^+$. In this case, we can write $\beta = (\beta - \alpha) + \alpha$. This expresses the simple root $\beta$ as the sum of two positive roots, which is also a contradiction.
\end{enumerate}
In both cases, we arrive at a contradiction. Therefore, the initial assumption $(\alpha, \beta) > 0$ is false. We conclude that $(\alpha, \beta) \le 0$.

(3) To prove linear independence, we consider a linear combination of the simple roots equal to zero with real coefficients:
$$
 \sum_{\alpha_i \in \Delta} c_i \alpha_i = 0, \quad c_i \in \mathbb{R}. 
$$
We separate the terms with positive coefficients from those with negative coefficients. Let $C_+ = \{i \mid c_i > 0\}$ and $C_- = \{i \mid c_i < 0\}$. The equation can be rewritten as:
$$
 \sum_{i \in C_+} c_i \alpha_i = \sum_{j \in C_-} (-c_j) \alpha_j. 
$$
Let us call this vector $\nu$. Note that in the expression on the right, the coefficients $-c_j$ are positive.
$$
 \nu = \sum_{i \in C_+} c_i \alpha_i = \sum_{j \in C_-} (-c_j) \alpha_j. 
$$
We calculate the squared norm of $\nu$:
$$
 (\nu, \nu) = \left( \sum_{i \in C_+} c_i \alpha_i, \sum_{j \in C_-} (-c_j) \alpha_j \right) = \sum_{i \in C_+, j \in C_-} c_i (-c_j) (\alpha_i, \alpha_j). 
$$
The index sets $C_+$ and $C_-$ are disjoint, so $\alpha_i \neq \alpha_j$. By property (2), we know that $(\alpha_i, \alpha_j) \le 0$. Furthermore, $c_i > 0$ and $-c_j > 0$. Therefore, each term in the sum is non-positive: $c_i  (-c_j)  (\alpha_i, \alpha_j) \le 0$.

The sum of non-positive terms is non-positive, i.e., $(\nu, \nu) \le 0$. Since the inner product is positive definite, the squared norm of a vector is always $\ge 0$. The only possibility is that $(\nu, \nu) = 0$, which implies $\nu=0$.

Now we have $\nu = \sum_{i \in C_+} c_i \alpha_i = 0$, with $c_i > 0$.
Take the inner product of this zero vector with the regular element $t$ that defines $\Phi^+$:
$$
 0 = (\nu, t) = \left( \sum_{i \in C_+} c_i \alpha_i, t \right) = \sum_{i \in C_+} c_i (\alpha_i, t). 
$$
By the definition of $\Phi^+$, for each $\alpha_i \in \Delta$, we have $(\alpha_i, t) > 0$. If the set $C_+$ were not empty, we would have a sum of strictly positive terms ($c_i > 0$ and $(\alpha_i, t) > 0$), the result of which would be strictly positive. This contradicts the fact that the sum is 0.
Therefore, the set $C_+$ must be empty. Similarly, $C_-$ must be empty.
This implies that all coefficients $c_i$ are zero, demonstrating that $\Delta$ is a linearly independent set.

Since $\Phi$ spans the space $E$ and property (1) shows that every element of $\Phi$ is a linear combination of elements of $\Delta$, it follows that $\Delta$ also spans $E$. Being a spanning and linearly independent set, $\Delta$ is a basis of $E$.
\end{proof}

\begin{example}[Simple roots in $A_n$]
\label{ex:Ansimples}
To find the simple roots of the root system $A_n$ seen in Example \ref{ex:An}, we first define a set of positive roots $\Phi^+$ via a polarization. The standard choice is to consider a root positive if its first index is smaller than the second:
$ \Phi^+ = \{ e_i - e_j \mid 1 \le i < j \le n+1 \}$.

If we consider a root of the form $\beta = e_i - e_k$ with $i < k$, we observe that if the indices are not consecutive (i.e., $k > i+1$), we can decompose it as $e_i - e_k = (e_i - e_j) + (e_j - e_k), $ for any $j$ such that $i < j < k.$
Since $e_i-e_j$ and $e_j-e_k$ are positive roots, $\beta$ is not simple.
Therefore, the only positive roots that cannot be decomposed are those with consecutive indices.

Thus, the set of simple roots for $A_n$ is $\Delta = \{ \alpha_1, \alpha_2, \dots, \alpha_n \}$, where $\alpha_i = e_i - e_{i+1}.$
\end{example}

\begin{example}[Simple roots of $\Phi^\vee$]
\label{ex:simplecorr}
Let $\Delta = \{\alpha_1, \dots, \alpha_n\}$ be the set of simple roots of $\Phi$. Let us see what the simple roots of $\Phi^\vee$ must be. We verify that the set of co-roots $\Delta^\vee = \{\alpha_1^\vee, \dots, \alpha_n^\vee\}$ is naturally the base of simple roots of $\Phi^\vee$.

We know that the dual system $\Phi^\vee$ is a root system and, therefore, must have a base of simple roots. Let us call this base $\Pi = \{\pi_1, \dots, \pi_n\}$. Since $\Pi$ is a base for $\Phi^\vee$, each root $\alpha_i^\vee \in \Delta^\vee$ can be expressed as a linear combination of the elements of $\Pi$ with non-negative integer coefficients.
$$
    \alpha_i^\vee = \sum_{j=1}^n K_{ij} \pi_j, \quad \text{with } K_{ij} \in \mathbb{Z}_{\ge 0}
$$
Conversely, the elements of the base $\Pi$ can also be expressed in terms of the linearly independent set $\Delta^\vee$, with coefficients that are also non-negative integers.
$$
    \pi_j = \sum_{l=1}^n M_{jl} \alpha_l^\vee, \quad \text{with } M_{jl} \in \mathbb{Z}_{\ge 0}
$$
Substituting this equation into the previous one, we have:
$$
    \alpha_i^\vee = \sum_{j=1}^n K_{ij} \left( \sum_{l=1}^n M_{jl} \alpha_l^\vee \right) = \sum_{l=1}^n \left( \sum_{j=1}^n K_{ij} M_{jl} \right) \alpha_l^\vee
$$
Since the vectors $\alpha_l^\vee$ are linearly independent, the above relation is only possible if the coefficient matrix is the identity. If $K=(K_{ij})$ and $M=(M_{jl})$, this means that $ KM = I$.

Furthermore, since the entries of the matrices $K$ and $M$ are non-negative integers, their product can only be the identity if both are permutation matrices. A permutation matrix $K$ implies that the relation in the first equation is simply a reordering. That is, the set $\{\alpha_1^\vee, \dots, \alpha_n^\vee\}$ is the same as the set $\{\pi_1, \dots, \pi_n\}$.

Therefore, $\Delta^\vee$ is the base of simple roots of $\Phi^\vee$.
\end{example}

\subsection{Cartan matrix and Dynkin diagrams}

With the base $\Delta$ of a root system $\Phi$ established, the relative geometry of the simple roots—that is, the angles and relative lengths between them—can be completely encoded by the matrix of Cartan integers as follows.

\begin{definition}[Cartan matrix]
Let $\Delta = \{\alpha_1, \dots, \alpha_r\}$ be a base for a root system $\Phi$ in a vector space $E$. The \textbf{Cartan matrix} $A(\Phi) = (A_{ij})$ of $\Phi$ is the $r \times r$ matrix defined by:
$$ A_{ij} = ( \alpha_j, \alpha_i^\vee ) = n_{\alpha_j \alpha_i} = 2 \dfrac{( \alpha_j, \alpha_i )}{( \alpha_i, \alpha_i )}. $$
\end{definition}

Note that from the properties of root systems, it follows that $A_{ii} = 2$ for all $i$, and that $A_{ij}$ is a non-positive integer for $i \neq j$.

In the case of the rank 2 root systems seen previously, we can construct the Cartan matrices as follows.

\begin{example}[$2 \times 2$ Cartan matrices]
\leavevmode
By Theorem \ref{theorem:angulosposibles}, we can characterize the Cartan matrices of the root systems seen in Figure \ref{fig:1-1} as follows.

\begin{itemize}
    \item \textbf{Root system $ A_1 \cup A_1 $:} This system consists of two orthogonal simple roots.
The Cartan matrix is:
$$
A(A_1 \cup A_1 ) = 
\begin{pmatrix}
2 & 0 \\
0 & 2
\end{pmatrix}
$$

\item \textbf{Root system $ A_2 $:} Its Cartan matrix is:
$$
A(A_2 ) = 
\begin{pmatrix}
2 & -1 \\
-1 & 2
\end{pmatrix}
$$

\item \textbf{Root system $ B_2 $:} Its Cartan matrix is:
$$
A(B_2 ) = 
\begin{pmatrix}
2 & -2 \\
-1 & 2
\end{pmatrix}
$$

\item \textbf{Root system $ G_2 $:} Its Cartan matrix is:
$$
A(G_2 ) = 
\begin{pmatrix}
2 & -3 \\
-1 & 2
\end{pmatrix}
$$
\end{itemize}
    
\end{example}

As mentioned at the beginning of the section, the set formed by the co-roots $\alpha^\vee = 2 \dfrac{\alpha}{(\alpha, \alpha)}$ with $\alpha \in \Phi$ a root system, is also a root system called the dual root system $\Phi^\vee$. A natural question is what relation the Cartan matrix of $\Phi$ has with that of $\Phi^\vee$. We will see below that this relation is quite elegant.

\begin{theorem}
\label{prop:cartandual}
The Cartan matrix of the dual system $\Phi^\vee$, denoted by $A^\vee$, is the transpose of the Cartan matrix of the original system $\Phi$. That is:
$$
A^\vee = A^T
$$
\end{theorem}

\begin{proof}
To demonstrate this, we first establish a key result: the co-root of a co-root is the original root. Note that:
$$
    (\alpha^\vee)^\vee = \dfrac{2\alpha^\vee}{(\alpha^\vee, \alpha^\vee)} = \dfrac{2 \left( \dfrac{2\alpha}{(\alpha, \alpha)} \right)}{\left( \dfrac{2\alpha}{(\alpha, \alpha)}, \dfrac{2\alpha}{(\alpha, \alpha)} \right)} = \dfrac{\dfrac{4\alpha}{(\alpha, \alpha)}}{\dfrac{4}{(\alpha, \alpha)^2}(\alpha, \alpha)} = \dfrac{\dfrac{4\alpha}{(\alpha, \alpha)}}{\dfrac{4}{(\alpha, \alpha)}} = \alpha
$$
Recalling Example \ref{ex:simplecorr}, we know that the base of the dual system $\Phi^\vee$ is $\Delta^\vee = \{\alpha_1^\vee, \dots, \alpha_n^\vee\}$. Then, by definition and the previous result, the $(i,j)$-entry of its Cartan matrix $A^\vee$ will be:
$$
(A^\vee)_{ij} = ((\alpha_j^\vee), (\alpha_i^\vee)^\vee) = (\alpha_j^\vee, \alpha_i) = \left( \dfrac{2\alpha_j}{(\alpha_j, \alpha_j)}, \alpha_i \right) = \dfrac{2(\alpha_j, \alpha_i)}{(\alpha_j, \alpha_j)}
$$
Now, we compare this expression with the $(j,i)$-entry of the original Cartan matrix $A$:
$$
A_{ji} = (\alpha_i, \alpha_j^\vee) = \dfrac{2(\alpha_i, \alpha_j)}{(\alpha_j, \alpha_j)}
$$
We see that $(A^\vee)_{ij} = A_{ji}$. Since this is true for all $i,j$, we conclude that the matrix $A^\vee$ is the transpose of $A$.
\end{proof}

The Cartan matrix $A = (A_{ij})$ contains all the necessary information to reconstruct the root system $\Phi$ from its base $\Delta$. However, direct inspection of this matrix of integers is not always the most intuitive way to understand the underlying geometric structure. Dynkin diagrams, originally introduced by Eugene Dynkin in 1947 in his paper ``The structure of semi-simple Lie algebras'' \cite{dynkin1947}, offer a remarkably compact and complete graphical representation of the Cartan matrix, and hence, of the root system.

\begin{definition}[Dynkin diagram]
\label{def:dynkin_diagram}
The \textbf{Dynkin diagram} associated with a base of simple roots $\Delta = \{\alpha_1, \dots, \alpha_r\}$ is a graph constructed according to the following rules:
\begin{enumerate}
    \item \textbf{Vertices:} Each simple root $\alpha_i \in \Delta$ corresponds to a vertex of the diagram, labeled by the index $i$.
    \item \textbf{Edges:} The number of edges connecting two distinct vertices $i$ and $j$ is equal to the product of the corresponding off-diagonal entries of the Cartan matrix:
        $$ \text{Number of edges between } i \text{ and } j = A_{ij}A_{ji}. $$
    \item \textbf{Arrows (Relative length):} If the roots $\alpha_i$ and $\alpha_j$ have different lengths, the product $A_{ij}A_{ji}$ will be greater than 1. In this case, an arrow is drawn on the edges pointing from the vertex corresponding to the \textbf{long} root towards the corresponding vertex of the \textbf{short} root.
\end{enumerate}
\end{definition}

\begin{remark}[Geometric Interpretation of the Rules]
Each element of the Dynkin diagram has a precise geometric interpretation derived from condition 4 of the root system axioms. If $\theta_{ij}$ is the angle between $\alpha_i$ and $\alpha_j$:
$$ A_{ij}A_{ji} = \left(2 \dfrac{( \alpha_j, \alpha_i )}{( \alpha_i, \alpha_i )}\right) \left(2 \dfrac{( \alpha_i, \alpha_j )}{( \alpha_j, \alpha_j )}\right) = 4 \dfrac{( \alpha_i, \alpha_j )^2}{( \alpha_i, \alpha_i ) ( \alpha_j, \alpha_j )} = 4 \cos^2\theta_{ij}. $$
Since $A_{ij}A_{ji}$ must be a non-negative integer (0, 1, 2, or 3), the possible angles between distinct simple roots are severely restricted:
\begin{itemize}
    \item \textbf{0 edges} ($A_{ij}A_{ji}=0$): $\cos\theta_{ij} = 0 \implies \theta_{ij} = \dfrac{\pi}{2}$. The roots are orthogonal.
    \item \textbf{1 edge} ($A_{ij}A_{ji}=1$): $\cos^2\theta_{ij} = 1/4 \implies \theta_{ij} = \dfrac{2\pi}{3}$. This occurs when $A_{ij}=A_{ji}=-1$, which implies $(\alpha_i, \alpha_i) = (\alpha_j, \alpha_j)$ (roots of equal length).
    \item \textbf{2 edges} ($A_{ij}A_{ji}=2$): $\cos^2\theta_{ij} = 2/4 \implies \theta_{ij} = \dfrac{3\pi}{4}$. This occurs if $A_{ij}=-1, A_{ji}=-2$ (or vice versa), implying $\dfrac{(\alpha_i, \alpha_i)}{(\alpha_j, \alpha_j)}=2$ (one root is $\sqrt{2}$ times longer than the other).
    \item \textbf{3 edges} ($A_{ij}A_{ji}=3$): $\cos^2\theta_{ij} = 3/4 \implies \theta_{ij} = \dfrac{5\pi}{6}$. This occurs if $A_{ij}=-1, A_{ji}=-3$ (or vice versa), implying $\dfrac{(\alpha_i, \alpha_i)}{(\alpha_j, \alpha_j)}=3$ (one root is $\sqrt{3}$ times longer than the other).
\end{itemize}
\end{remark}

\begin{example}[Some fundamental Dynkin diagrams]
\leavevmode
\begin{enumerate}
    \item \textbf{Type $A_n$:} All simple roots have the same length and each $\alpha_i$ is non-orthogonal only to its immediate neighbors $\alpha_{i-1}$ and $\alpha_{i+1}$. Its Cartan matrix has $A_{i,i+1}=A_{i+1,i}=-1$ and the rest of the off-diagonal entries are zero, that is:
    $$
    A = 
    \begin{pmatrix}
    2 & -1 & 0 & \cdots & 0 \\
    -1 & 2 & -1 & \cdots & 0 \\
    0 & -1 & 2 & \ddots & \vdots \\
    \vdots & \vdots & \ddots & \ddots & -1 \\
    0 & 0 & \cdots & -1 & 2
    \end{pmatrix}.
    $$
    
    The diagram is a simple chain of $n$ vertices.
    
    
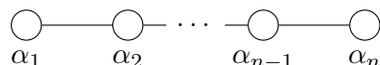
\begin{center}
\begin{tikzpicture}[scale=0.9]
    \node[circle, draw, fill=white, minimum size=4mm, inner sep=0pt, label=below:{$\alpha_1$}] (v1) at (0,0) {};
    \node[circle, draw, fill=white, minimum size=4mm, inner sep=0pt, label=below:{$\alpha_2$}] (v2) at (1.5,0) {};
    
    \node (dots) at (2.5,0) {\ldots};
    
    \node[circle, draw, fill=white, minimum size=4mm, inner sep=0pt, label=below:{$\alpha_{n-1}$}] (vn-1) at (3.5,0) {};
    \node[circle, draw, fill=white, minimum size=4mm, inner sep=0pt, label=below:{$\alpha_n$}] (vn) at (5,0) {};
    
    \draw (v1) -- (v2);
    \draw (v2) -- (dots);
    \draw (dots) -- (vn-1);
    \draw (vn-1) -- (vn);
\end{tikzpicture}
    \end{center}
    
    \item \textbf{Type $B_2$:} $\Delta = \{\alpha_1, \alpha_2\}$ with $\alpha_1$ short and $\alpha_2$ long. The Cartan matrix is $A(B_2) = \begin{pmatrix} 2 & -2 \\ -1 & 2 \end{pmatrix}$. The product $A_{12}A_{21} = (-2)(-1)=2$. The arrow goes from the long root ($\alpha_2$) to the short one ($\alpha_1$).
        
    
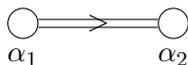
\begin{center}
\begin{tikzpicture}
    \node[circle, draw, fill=white, minimum size=4mm, inner sep=0pt, label=below:{$\alpha_1$}] (v1) at (0,0) {};
    \node[circle, draw, fill=white, minimum size=4mm, inner sep=0pt, label=below:{$\alpha_2$}] (v2) at (2,0) {};
    
    \path (v2) -- (v1) node[midway] {$>$};
    \draw[transform canvas={yshift=1.5pt}] (v2.west) -- (v1.east);
    \draw[transform canvas={yshift=-1.5pt}] (v2.west) -- (v1.east);
\end{tikzpicture}
    \end{center}
    
    \item \textbf{Type $G_2$:} This is the smallest exceptional case. Its Cartan matrix is $A(G_2) = \begin{pmatrix} 2 & -1 \\ -3 & 2 \end{pmatrix}$. The product $A_{12}A_{21} = (-1)(-3)=3$, which gives rise to a triple edge with an arrow pointing from the long root (associated with vertex 2) to the short one (vertex 1).
        
    
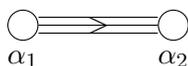
\begin{center}
\begin{tikzpicture}
    \node[circle, draw, fill=white, minimum size=4mm, inner sep=0pt, label=below:{$\alpha_1$}] (v1) at (0,0) {};
    \node[circle, draw, fill=white, minimum size=4mm, inner sep=0pt, label=below:{$\alpha_2$}] (v2) at (2,0) {};
    
    \path (v2) -- (v1) node[midway] {$>$};
    \draw (v2.west) -- (v1.east);
    \draw[transform canvas={yshift=3pt}] (v2.west) -- (v1.east);
    \draw[transform canvas={yshift=-3pt}] (v2.west) -- (v1.east);
\end{tikzpicture}
    \end{center}
    
\end{enumerate}
\end{example}

\begin{remark}
    The co-root system of $B_2$ is denoted as $C_2$. By Proposition \ref{prop:cartandual}, the system $C_2$ satisfies $A(C_2) = A(B_2)^T$. The Dynkin diagram of co-roots of $B_2$, $\tilde{\Gamma}(B_2)$, is constructed from the transposed Cartan matrix, so $\tilde{\Gamma}(B_2) = \Gamma(C_2)$. This duality will be essential for the correct algebraic simulation of the Kostant game, as will be seen in Section \ref{Chapter3}.
\end{remark}

We will now see that the structure of the Dynkin diagram can directly reflect a fundamental property of the root system: its irreducibility.

\begin{theorem}[Irreducibility and Connectedness]
A root system $\Phi$ is irreducible if and only if its Dynkin diagram is connected.
\end{theorem}
\begin{proof}
($\Rightarrow$) Suppose the Dynkin diagram $\Gamma$ is disconnected. Then, the set of vertices (and hence the base $\Delta$) can be partitioned into two non-empty subsets $\Delta_1$ and $\Delta_2$ such that there are no edges between any vertex of $\Delta_1$ and any vertex of $\Delta_2$. By definition, this means that for any $\alpha_i \in \Delta_1$ and $\alpha_j \in \Delta_2$, we have $A_{ij}A_{ji}=0$, which implies $( \alpha_i, \alpha_j ) = 0$. The roots in $\Delta_1$ are orthogonal to those in $\Delta_2$. Let $\Phi_1$ be the set of roots that are linear combinations of $\Delta_1$ and $\Phi_2$ be the analogue for $\Delta_2$. It can be shown that $\Phi = \Phi_1 \cup \Phi_2$, and since $( \alpha, \beta )=0$ for $\alpha\in\Phi_1, \beta\in\Phi_2$, the system $\Phi$ is reducible. This proves the contrapositive.

($\Leftarrow$) Suppose $\Phi$ is reducible, $\Phi = \Phi_1 \cup \Phi_2$ with $(\Phi_1, \Phi_2)=0$, that is, for all $\phi_1 \in \Phi_1$ and for all $\phi_2 \in \Phi_2$, it holds that $(\phi_1, \phi_2)=0$. The base $\Delta$ must be the union of a base $\Delta_1$ for $\Phi_1$ and a base $\Delta_2$ for $\Phi_2$. But then, every element of $\Delta_1$ is orthogonal to every element of $\Delta_2$, which implies that the Dynkin diagram is disconnected.
\end{proof}

The remarkable rigidity of condition 4 of the root system axioms leads to one of the most beautiful results in mathematics: the complete classification of irreducible root systems.

\begin{theorem}[Classification of Connected Dynkin Diagrams]
\label{teo:clasifidynkin}
The only possible connected Dynkin diagrams are those corresponding to the four classical infinite families $A_n, B_n, C_n, D_n$ and the five exceptional diagrams $E_6, E_7, E_8, F_4, G_2$.
\end{theorem}

\begin{figure}[H]
\centering
\tikzset{
    vertex/.style={circle, draw, fill=white, minimum size=6pt, inner sep=0pt},
    edge/.style={-},
    wedge line/.style={-, very thin}
}

\scalebox{1.3}{
\begin{tikzpicture}[scale=1, baseline]
\node at (-1.7,0) {$A_n\ (n\geq 1):$};
\node[vertex] (1) at (0,0) {};
\node[vertex] (2) at (1,0) {};
\node at (2,0) {$\cdots$};
\node[vertex] (3) at (3,0) {};
\node[vertex] (4) at (4,0) {};
\draw[edge] (1) -- (2);
\draw[edge] (2) -- (1.7,0);
\draw[edge] (2.3,0) -- (3);
\draw[edge] (3) -- (4);
\end{tikzpicture}
}

\scalebox{1.3}{
\begin{tikzpicture}[scale=1, baseline]
\node at (-1.7,0) {$B_n\ (n\geq 2):$};
\node[vertex] (1) at (0,0) {};
\node[vertex] (2) at (1,0) {};
\node at (2,0) {$\cdots$};
\node[vertex] (3) at (3,0) {};
\node[vertex] (4) at (4,0) {};
\draw[edge] (1) -- (2);
\draw[edge] (2) -- (1.7,0);
\draw[edge] (2.3,0) -- (3);
\draw[wedge line] (3.1,0.05) -- (3.9,0.05);
\draw[wedge line] (3.1,-0.05) -- (3.9,-0.05);
\draw[-] (3.45,0.1) -- (3.55,0) -- (3.45,-0.1);
\end{tikzpicture}
}

\scalebox{1.3}{
\begin{tikzpicture}[scale=1, baseline]
\node at (-1.7,0) {$C_n\ (n\geq 2):$};
\node[vertex] (1) at (0,0) {};
\node[vertex] (2) at (1,0) {};
\node at (2,0) {$\cdots$};
\node[vertex] (3) at (3,0) {};
\node[vertex] (4) at (4,0) {};
\draw[edge] (1) -- (2);
\draw[edge] (2) -- (1.7,0);
\draw[edge] (2.3,0) -- (3);
\draw[wedge line] (3.1,0.05) -- (3.9,0.05);
\draw[wedge line] (3.1,-0.05) -- (3.9,-0.05);
\draw[-] (3.55,0.1) -- (3.45,0) -- (3.55,-0.1);
\end{tikzpicture}
}

\scalebox{1.3}{
\begin{tikzpicture}[scale=1, baseline]
\node at (-1.7,0) {$D_n\ (n\geq 4):$};
\node[vertex] (1) at (0,0) {};
\node[vertex] (2) at (1,0) {};
\node at (2,0) {$\cdots$};
\node[vertex] (3) at (3,0) {};
\node[vertex] (4) at (4,0.7) {};
\node[vertex] (5) at (4,-0.7) {};
\draw[edge] (1) -- (2);
\draw[edge] (2) -- (1.7,0);
\draw[edge] (2.3,0) -- (3);
\draw[edge] (3) -- (4);
\draw[edge] (3) -- (5);
\end{tikzpicture}
}

\scalebox{1.3}{
\begin{tikzpicture}[scale=1, baseline]
\node at (-1,0) {$E_6:$};
\node[vertex] (1) at (0,0) {};
\node[vertex] (2) at (1,0) {};
\node[vertex] (3) at (2,0) {};
\node[vertex] (4) at (3,0) {};
\node[vertex] (5) at (4,0) {};
\node[vertex] (6) at (2,1) {};
\draw[edge] (1) -- (2) -- (3) -- (4) -- (5);
\draw[edge] (3) -- (6);
\end{tikzpicture}
}

\scalebox{1.3}{
\begin{tikzpicture}[scale=1, baseline]
\node at (-1,0) {$E_7:$};
\node[vertex] (1) at (0,0) {};
\node[vertex] (2) at (1,0) {};
\node[vertex] (3) at (2,0) {};
\node[vertex] (4) at (3,0) {};
\node[vertex] (5) at (4,0) {};
\node[vertex] (6) at (5,0) {};
\node[vertex] (7) at (2,1) {};
\draw[edge] (1) -- (2) -- (3) -- (4) -- (5) -- (6);
\draw[edge] (3) -- (7);
\end{tikzpicture}
}

\scalebox{1.3}{
\begin{tikzpicture}[scale=1, baseline]
\node at (-1,0) {$E_8:$};
\node[vertex] (1) at (0,0) {};
\node[vertex] (2) at (1,0) {};
\node[vertex] (3) at (2,0) {};
\node[vertex] (4) at (3,0) {};
\node[vertex] (5) at (4,0) {};
\node[vertex] (6) at (5,0) {};
\node[vertex] (7) at (6,0) {};
\node[vertex] (8) at (2,1) {};
\draw[edge] (1) -- (2) -- (3) -- (4) -- (5) -- (6) -- (7);
\draw[edge] (3) -- (8);
\end{tikzpicture}
}

\scalebox{1.3}{
\begin{tikzpicture}[scale=1, baseline]
\node at (-1,0) {$F_4:$};
\node[vertex] (1) at (0,0) {};
\node[vertex] (2) at (1,0) {};
\node[vertex] (3) at (2,0) {};
\node[vertex] (4) at (3,0) {};
\draw[edge] (1) -- (2);
\draw[wedge line] (1.1,0.05) -- (1.9,0.05);
\draw[wedge line] (1.1,-0.05) -- (1.9,-0.05);
\draw[-] (1.45,0.1) -- (1.55,0) -- (1.45,-0.1);
\draw[edge] (3) -- (4);
\end{tikzpicture}
}

\scalebox{1.3}{
\begin{tikzpicture}[scale=1, baseline]
\node at (-1,0) {$G_2:$};
\node[vertex] (1) at (0,0) {};
\node[vertex] (2) at (1,0) {};
\draw[wedge line] (0.07,0.07) -- (0.93,0.07);
\draw[wedge line] (0.1,0) -- (0.9,0);
\draw[wedge line] (0.07,-0.07) -- (0.93,-0.07);
\draw[-] (0.45,0.1) -- (0.55,0) -- (0.45,-0.1);
\end{tikzpicture}
}

\caption{Dynkin diagrams associated with irreducible root systems.}
\label{fig:dynkin_diagram}
\end{figure}
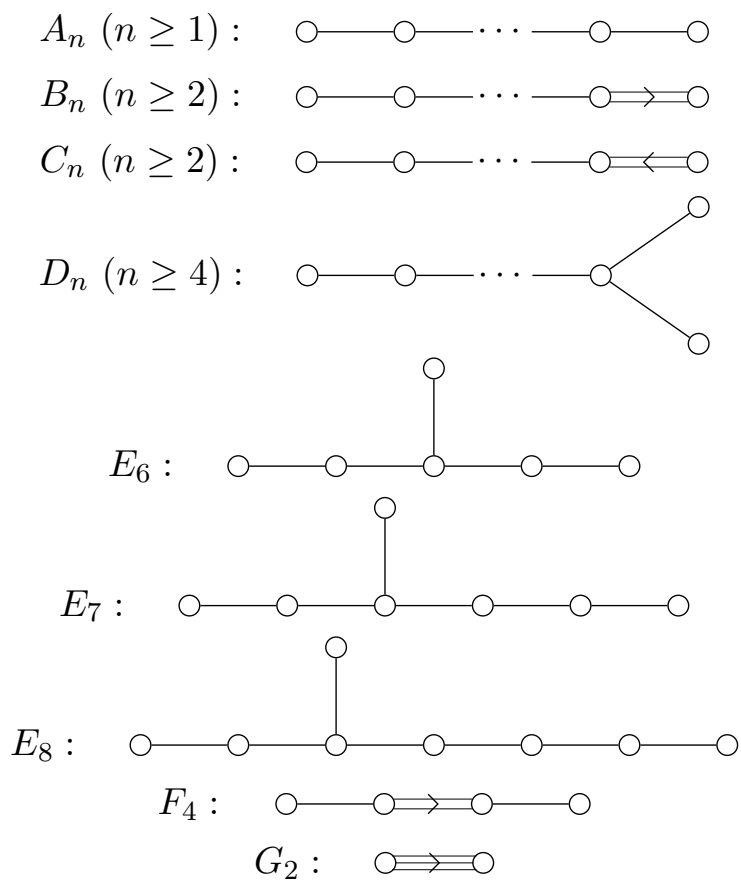

\begin{remark}
The proof of this classification theorem will be discussed in the next section. Its importance lies in the fact that it reduces the classification of complex objects such as simple compact Lie groups to the finite task of analyzing this short list of graphs. The original construction of these root systems and their relationship with the classification of simple compact Lie groups can be consulted in \cite[\S 4]{jacobson1979lie}, and some particular cases are discussed in \cite{Kirillov2005CompactGroups}.
\end{remark}

\begin{definition}
In the case of the diagrams $A_n, D_n, E_6, E_7$, and $E_8$, since these have no double edges or directions, they are called \textbf{simply-laced Dynkin diagrams}. The remaining diagrams $B_n, C_n, F_4$, and $G_2$ are called \textbf{multiply-laced Dynkin diagrams}.
\end{definition}

\subsection{Classification of irreducible root systems}
\label{sec1.3}

Although a complete and detailed proof of the Classification Theorem \ref{teo:clasifidynkin} is an extensive combinatorial exercise (see [\cite{Humphreys1972}, \S 11.4]), it is instructive to outline the central ideas underlying the proof. The argument requires not advanced tools, but a systematic application of graph-theoretic restrictions derived from the definition of a root system and its consequent Dynkin diagram. The general strategy consists of demonstrating which graph configurations cannot be Dynkin diagrams, thus eliminating all possibilities except for the known finite list.

The proof relies on the following fundamental lemma, which translates the property of being a Dynkin diagram into a condition on a quadratic form.

\begin{definition}[Associated Quadratic Form]
Let $\Gamma$ be a graph with vertices $\{1, \dots, n\}$. To each vertex $i$, we associate a vector $e_i$ from the standard basis of $\R^n$. We define a symmetric bilinear form on $\R^n$ by
$$ (e_i, e_j) = -\cos(\pi/m_{ij}), $$
where $m_{ij}$ is an integer representing the ``strength'' of the connection between vertices $i$ and $j$: $m_{ii}=1$, $m_{ij}=2$ if there is no edge, $m_{ij}=3$ if there is a simple edge, $m_{ij}=4$ if there is a double edge, and $m_{ij}=6$ if there is a triple edge. The associated quadratic form is $Q(x) = (x,x) = \sum_{i,j} (e_i,e_j) x_i x_j$.
\end{definition}

\begin{lemma}[Positive Definiteness Criterion]
A connected graph $\Gamma$ is a Dynkin diagram of an irreducible root system if and only if the associated quadratic form $Q(x)$ is positive definite.
\end{lemma}

The proof of this lemma is extensive and can be found in \cite{Humphreys1972}. However, for the case of simply-laced diagrams, a lighter proof can be consulted in \cite{ngotiaoco2018spectral}, where it is proven that the only connected graphs whose adjacency matrix has all its eigenvalues strictly less than 2 are the diagrams $A_n, D_n, E_6, E_7, E_8$. From this, it follows that the matrix $B = 2I - A$ is positive definite, and therefore the associated quadratic form $Q(x)$ is also positive definite.

With this lemma, the problem of classifying Dynkin diagrams translates into a linear algebra problem: classifying all graphs for which the associated quadratic form is positive definite. Thus, the proof of Theorem \ref{teo:clasifidynkin} proceeds by eliminating ``forbidden subgraphs''.

\textbf{Case 1: Graphs cannot contain cycles.}
If a graph $\Gamma$ contains a cycle, for example, $v_{i_1} - v_{i_2} - \dots - v_{i_k} - v_{i_1}$, one can construct a vector $x = \sum e_{i_k}$ (where the $e_{i_k}$ correspond to the vertices of the cycle). By calculating $Q(x)$, it is shown that $Q(x) \le 0$, which violates the positive definiteness condition. Therefore, Dynkin diagrams must be trees (connected graphs without cycles).

\begin{figure}[H]
\centering
\begin{tikzpicture}[
  nodo/.style={circle, draw, inner sep=1.2pt, fill=white},
  edge/.style={thick}
]
  \node[nodo] (v1) at (0:1) {};
  \node[nodo] (v2) at (90:1) {};
  \node[nodo] (v3) at (180:1) {};
  \node[nodo] (v4) at (270:1) {};
  \foreach \i/\j in {v1/v2,v2/v3,v3/v4,v4/v1}
    \draw[edge] (\i) -- (\j);
  \node at (0,-1.5) {Step 1: Forbidden cycle};
\end{tikzpicture}
\caption{A cycle of length 4 — cannot occur in a Dynkin diagram.}
\end{figure}

\textbf{Case 2: The degree of each vertex is at most 3.}
Consider a vertex $v_{i_0}$ and its neighbors $v_{i_1}, \dots, v_{i_k}$. Construct a specific vector $x$ (for example, $x = 2e_{i_0} + \sum e_{i_k}$). Evaluating $Q(x)$, it is shown that if the number of neighbors (counting edge multiplicity) is greater than or equal to 4, the form ceases to be positive definite. 

Indeed, using the bilinearity of the form and that $(e_i,e_i) = 1$, we have:
$$
Q(x) = (x,x) = 4(e_{i_0},e_{i_0}) + \sum_{j=1}^k (e_{i_j},e_{i_j}) + 4\sum_{j=1}^k (e_{i_0}, e_{i_j}).
$$
If all edges connecting $v_{i_0}$ with its neighbors are simple, then $(e_{i_0}, e_{i_j}) = -\cos\left(\dfrac{\pi}{3}\right) = -\dfrac{1}{2}$, so:
$$
Q(x) = 4 + k - 2k = 4 - k.
$$
Therefore, if $k \geq 4$, we have $Q(x) \leq 0$, and the form is no longer positive definite.

This prohibits a vertex from having four or more neighbors with simple edges, two neighbors with a double edge, etc. The only possibility for a degree greater than 2 is a single branching point with three arms of simple edges.

\begin{figure}[H]
\centering
\begin{tikzpicture}[
  nodo/.style={circle, draw, inner sep=1.2pt, fill=white},
  edge/.style={thick}
]
  \node[nodo] (c) at (0,0) {};
  \foreach \ang in {45,135,225,315} {
    \node[nodo] (n\ang) at (\ang:1) {};
    \draw[edge] (c) -- (n\ang);
  }
  \node at (0,-1.5) {Step 2: Vertex of degree 4};
\end{tikzpicture}
\caption{A vertex with four simple neighbors is impossible in Dynkin diagrams.}
\end{figure}

\textbf{Case 3: There can be only one branching point (or one multiple edge).} 
It is shown that if a graph contains two or more ``special features'' (branching points, double or triple edges), the quadratic form ceases to be positive definite. The key argument relies on considering \emph{contractions} of certain parts of the graph.

For example, if a graph possesses two branching points, one of the arms (a linear chain of vertices) can be contracted to a single vertex, thus obtaining a simpler graph that preserves the essential combinatorial structure. This procedure is called graph contraction.

A fundamental fact is that the contraction of a Dynkin diagram yields another Dynkin diagram, since the restriction of a positive definite form to a subspace remains positive definite. This means that any subgraph obtained by deleting vertices or contracting paths must still belong to the classified list of Dynkin diagrams.

This allows for the systematic elimination of complex configurations: for instance, if contracting the arms of a graph with two branching points yields a subgraph that does not appear in the classification, we conclude that the original graph cannot be of Dynkin type. Thus, it is proven that only one of these ``special features'' (either a branching point or a multiple edge) can exist, and not multiple ones simultaneously.

\begin{figure}[h]
\centering
\begin{tikzpicture}[
  nodo/.style={circle, draw, inner sep=1.2pt, fill=white},
  edge/.style={thick}
]
  \node[nodo] (a1) at (-3,0) {};
  \node[nodo] (a2) at (-2,0) {};
  \node at (-1,0) (dots) {$\cdots$};
  \node[nodo] (a3) at (0,0) {};
  \node[nodo] (a4) at (1,0) {};

  \node[nodo] (b1) at ($(a1)+(-0.5, 0.5)$) {};
  \node[nodo] (b2) at ($(a1)+(-0.5,-0.5)$) {};

  \node[nodo] (b3) at ($(a4)+(0.5, 0.5)$) {};
  \node[nodo] (b4) at ($(a4)+(0.5,-0.5)$) {};

  \draw[edge] (a1) -- (a2);
  \draw[edge] (a2) -- ($(dots.west)+(-0.15,0)$);
  \draw[edge] ($(dots.east)+(0.15,0)$) -- (a3);
  \draw[edge] (a3) -- (a4);

  \draw[edge] (a1) -- (b1) (a1) -- (b2);
  \draw[edge] (a4) -- (b3) (a4) -- (b4);

  \node at (-1,-1.2) {Step 3: Two branching points};
\end{tikzpicture}
\caption{A graph with two branching points cannot be a Dynkin diagram.}
\end{figure}

\textbf{Case 4: Analysis of remaining cases.}
After these eliminations, only a small number of possible graph families remain:
\begin{enumerate}
    \item \textbf{Simple chains ($A_n$):} Always possible.
    \item \textbf{Chains with one double edge ($B_n/C_n, F_4$):} The analysis of how many vertices can be on each side of the double edge is performed. The positive definiteness condition imposes a strict limit, resulting in the series $B_n/C_n$ and the exceptional case $F_4$.
    \item \textbf{Chains with one triple edge ($G_2$):} It is quickly shown that only the two-vertex configuration is possible.
    \item \textbf{Stars with three arms ($D_r, E_6, E_7, E_8$):} A branching point with three arms of lengths $p, q, r$ is considered. A vector $x$ dependent on $p,q,r$ is constructed, and the condition $Q(x)>0$ translates into the Diophantine inequality $\dfrac{1}{p+1} + \dfrac{1}{q+1} + \dfrac{1}{r+1} > 1$. The only integer solutions for $(p,q,r)$ correspond precisely to the configurations of $D_n$ and $E_6, E_7, E_8$. For this complete argument, see [\cite{Fulton1997}, \S 22].
\end{enumerate}

This systematic process of elimination, based on the positive definiteness criterion, guarantees that no other connected Dynkin diagrams exist beyond the list in Figure \ref{fig:dynkin_diagram}.

\begin{remark}[Towards a Combinatorial Proof]
The classical proof is elegant but depends on a linear algebra criterion (the positive definiteness of a quadratic form), a criterion which can be consulted in the references and is by no means trivial. In the following sections, we will develop a completely different perspective. We will argue that the structure of root systems and, consequently, their classification, can be explored through a dynamic game on graphs. The configurations in the Kostant game and the unique termination condition will serve as a combinatorial analogue to the positive definiteness criterion, offering an alternative view of why only these finite structures are possible.
\end{remark}

\subsection{Weyl groups}
\label{sec1.4}

Associated with every root system $\Phi$ is a finite symmetry group generated by reflections across the roots. This structure, the Weyl group, encapsulates the combinatorial structure of $\Phi$. In this section, we will introduce it from the geometric perspective of automorphisms, following Definition \ref{def:isoraices}, and then develop the combinatorial language of words and length, showing how both perspectives complement each other.

\subsubsection{The Weyl group as an automorphism group}

\begin{definition}[Weyl group]
The \textbf{Weyl group} $W$ of a root system $\Phi$ is the subgroup of $\text{Aut}(\Phi)$ generated by the reflections $s_\alpha$ for all roots $\alpha \in \Phi$, that is,
$$ W = \langle s_\alpha \mid \alpha \in \Phi \rangle \subset \text{Aut}(\Phi). $$
\end{definition}

Before looking at examples of these groups, we will see some fundamental results that allow us to construct these structures more simply.

\begin{lemma}[]
\leavevmode
\begin{enumerate}
    \item The Weyl group $ W $ is a finite subgroup of the orthogonal group $ O(E) $, and the root system $ \Phi $ is invariant under the action of $ W $.
    \item For all $ w \in W $ and $ \alpha \in \Phi $, it holds that
    $$
    s_{w(\alpha)} = w s_\alpha w^{-1}.
    $$
\end{enumerate}
\end{lemma}

\begin{proof}
Since every reflection $ s_\alpha $ is an orthogonal transformation, we have $ W \subset O(E) $. As $ s_\alpha(\Phi) = \Phi $ (by the axioms of a root system), it also holds that $ w(\Phi) = \Phi $ for all $ w \in W $. Furthermore, if some $ w \in W $ leaves every root invariant, then $ w = \mathrm{id} $ (since $ \Phi $ spans $ E $).

Therefore, $ W $ is a subgroup of the group $ \mathrm{Aut}(\Phi) $ of all automorphisms of $ \Phi $. Since $ \Phi $ is a finite set, $ \mathrm{Aut}(\Phi) $ is also finite; and consequently, $ W $ is also finite.

The second identity is immediate: indeed, $ w s_\alpha w^{-1} $ acts as the identity on the hyperplane $ w H_\alpha = H_{w(\alpha)} $, and
$$
w s_\alpha w^{-1}(w(\alpha)) = w s_\alpha(\alpha) = w(-\alpha) = -w(\alpha),
$$
thus, it corresponds to the reflection associated with the root $ w(\alpha) $.
\end{proof}

We now see a result that allows describing the Weyl group via a finite set of generators (the simple reflections) and relations between them.

\begin{lemma}
\label{lemm:raiconj}
Let $ \Phi $ be a root system with a base of simple roots $ \Delta $. Then, for every positive root $ \beta \in \Phi^+ $, there exists an element $ w \in \langle s_\alpha \mid \alpha \in \Delta \rangle $ such that $ w(\beta) \in \Delta $. 
\end{lemma}

\begin{proof}
Let $ \beta \in \Phi^+ $. By definition, we can write $ \beta $ as a linear combination with non-negative integer coefficients:
$$
\beta = \sum_{\alpha \in \Delta} c_\alpha \alpha, \quad c_\alpha \in \mathbb{Z}_{\geq 0}, \quad c_\alpha \neq 0.
$$
We define the \emph{height} of $ \beta $ as:
$$
\operatorname{ht}(\beta) = \sum_{\alpha \in \Delta} c_\alpha.
$$
We proceed by induction on the height.

If $ \operatorname{ht}(\beta) = 1 $, then $ \beta = \alpha \in \Delta $, and there is nothing to prove.

Suppose that every positive root of height less than $ h $ is conjugate to a simple root. Let $ \beta \in \Phi^+ $ with $ \operatorname{ht}(\beta) = h > 1 $. Then there exists some $ \alpha_i \in \Delta $ such that $ (\beta, \alpha_i) > 0 $, and by the properties of root systems:
$$
s_{\alpha_i}(\beta) = \beta - ( \beta, \alpha_i^\vee ) \alpha_i \in \Phi^+,
$$
and $ \operatorname{ht}(s_{\alpha_i}(\beta)) < \operatorname{ht}(\beta) $, since a positive multiple of $ \alpha_i $ is subtracted.

By the inductive hypothesis, there exists $ w $ such that $ w(s_{\alpha_i}(\beta)) \in \Delta $. Then $ ws_{\alpha_i}(\beta) \in \Delta $, which implies that $ \beta $ is conjugate to a simple root under an element of $ \langle s_\alpha \rangle $, as desired.
\end{proof}

\begin{remark}
The definition of height for a positive root used in this proof is standard in root system theory. This notion has several important interpretations and utilities: Height measures how ``far'' a positive root is in terms of combinations of simple roots. Simple roots are precisely those of height 1. Furthermore, this definition allows organizing the set $\Phi^+$ as a graded structure, which is useful for various inductive arguments. This notion will be fundamental later when we analyze the length of elements in terms of inversion sets.
\end{remark}

\begin{theorem}
\label{teo:weyl_simple}
Let $ \Phi $ be a reduced root system and $ \Delta = \{ \alpha_1, \dots, \alpha_n \} $ a base of simple roots. Then the Weyl group $ W $ is generated by the reflections with respect to the simple roots:
$$
W = \langle s_{\alpha_1}, \dots, s_{\alpha_n} \rangle.
$$
\end{theorem}

\begin{proof}
Let $ W' = \langle s_{\alpha_1}, \dots, s_{\alpha_n} \rangle \subseteq W $. We will prove that $ W = W' $.

Since $ W $ is generated by reflections $ s_\beta $ with $ \beta \in \Phi $, it suffices to prove that every reflection $ s_\beta $ belongs to $ W' $.

By the previous lemma, there exists $ w \in W' $ such that $ w(\beta) = \alpha_i \in \Delta $. Using the fact that $ s_{w(\beta)} = w s_\beta w^{-1} $, we have:
$$
s_\beta = w^{-1} s_{w(\beta)} w = w^{-1} s_{\alpha_i} w \in W',
$$
since $ s_{\alpha_i} \in W' $ and $ W' $ is a group.

This proves that all reflections $ s_\beta \in W $ are in $ W' $, so $ W \subseteq W' $. Since $ W' \subseteq W $ by construction, we conclude $ W = W' $.
\end{proof}

The above justifies why the global properties of the system (such as symmetries, relative lengths, and root combinatorics) can be studied starting from the base of simple roots.

\begin{example}[The Weyl group of $ A_2 $]
The root system $ A_2 $ can be represented in the subspace of $ \mathbb{R}^3 $ orthogonal to the vector $ e_1 + e_2 + e_3 $, with roots:
$$
\Phi = \{ \pm(e_i - e_j) \mid 1 \leq i < j \leq 3 \}
$$
The simple roots can be taken as:
$$
\alpha_1 = e_1 - e_2, \quad \alpha_2 = e_2 - e_3
$$

The Weyl group $ W(A_2) $ is generated by the reflections $ s_{\alpha_1} $ and $ s_{\alpha_2} $. Algebraically, it can be easily shown that:
$$
W(A_2) \cong S_3
$$
That is, the Weyl group of $ A_2 $ is isomorphic to the symmetric group on three elements. Geometrically, it corresponds to the symmetry group of an equilateral triangle.
\end{example}

\begin{remark}
    The pair $(W, S)$, where $S = \{s_{\alpha_i} \mid \alpha_i \in \Delta\}$, forms what is commonly called a \textbf{Coxeter system}, with $W$ being a Coxeter group and $S$ the set of generators. Formally, a \textbf{Coxeter group} is one that can be determined by the presentation $\langle r_1,r_2, \ldots, r_n \mid (r_i r_j)^m_{ij}=1 \rangle$, where $m_{ii}=1$ and $m_{ij}\geq 2$ for $i \neq j$. By this particular definition, every Weyl group is a Coxeter group, but not every Coxeter group is a Weyl group (e.g., if $S$ is infinite). For more on the theory of these particular groups, see \cite{HumphreysReflection1990}.
\end{remark}

\subsubsection{Length and reduced expressions}

The fact that $W$ is generated by the finite set $S$ allows us to study it via ``words'' formed by these generators.

\begin{definition}[Length and reduced expressions]
Let $w \in W$. The \textbf{length} of $w$, denoted by $\ell(w)$, is the smallest non-negative integer $t$ for which there exist simple reflections $s_{i_1}, \dots, s_{i_t}$ such that $w = s_{i_1} s_{i_2} \cdots s_{i_t}$. Such a minimal length expression is called a \textbf{reduced expression} for $w$.
\end{definition}

\begin{theorem}[Exchange Condition \cite{HumphreysReflection1990} \S 1.7]
\label{thm:exchange_condition}
Let $w = s_1 \cdots s_t$ (not necessarily reduced) and $s_i \in S$. If $\ell(ws) < \ell(w)$ for some simple reflection $s=s_\alpha$, then there exists an index $k \in \{1, \dots, t\}$ such that $ws = s_1 \cdots \hat{s}_k \cdots s_t$, where the hat indicates that the term is omitted.
In particular, $w$ has a reduced expression ending in $s$ if and only if $\ell(ws)< \ell(w)$.
\end{theorem}

Length, a purely algebraic notion, has a surprisingly direct geometric and combinatorial counterpart.

\begin{definition}[Inversion set]
Fixing a base $\Delta$ and its corresponding set of positive roots $\Phi^+$, the \textbf{inversion set} of an element $w \in W$ is the set of positive roots mapped to negative roots by the action of $w$:
$$ \mathcal{I}(w) = \{\alpha \in \Phi^+ \mid w(\alpha) \in \Phi^-\}. $$
\end{definition}

The following result is the most important conceptual bridge between the algebra of words in $W$ and the combinatorics of roots in $\Phi$.

\begin{proposition}[Length as cardinality of inversions \cite{HumphreysReflection1990} \S 1.6]
\label{prop:longitud_inversiones}
For any $w \in W$, the length of $w$ is precisely the number of its inversions:
$$ \ell(w) = |\mathcal{I}(w)|. $$
\end{proposition}
\begin{proof}
Let $s_i \in S$ be a simple reflection. Consider the inversion set of $ws_i$. A root $\alpha \in \Phi^+$ is in $I(ws_i)$ if and only if $ws_i(\alpha) \in \Phi^-$. The action of $s_i$ permutes all roots in $\Phi^+ \setminus \{\alpha_i\}$ and sends $\alpha_i$ to $-\alpha_i$.

Let $w \in W$. The condition $\ell(ws_i) = \ell(w) + 1$ is equivalent to $w(\alpha_i) \in \Phi^+$. In this case, the inversion set of $ws_i$ is $I(ws_i) = s_i(\mathcal{I}(w)) \cup \{\alpha_i\}$, and its cardinality is $|\mathcal{I}(w)|+1$.
On the other hand, the condition $\ell(ws_i) = \ell(w) - 1$ is equivalent to $w(\alpha_i) \in \Phi^-$. In this case, $I(ws_i) = s_i(\mathcal{I}(w) \setminus \{\alpha_i\})$, and its cardinality is $|\mathcal{I}(w)|-1$.

We now proceed by induction on the length $\ell(w)$.
\begin{itemize}
    \item If $\ell(w)=0$, then $w = \text{id}$. $I(\text{id}) = \emptyset$ since no positive root is sent to a negative one. Therefore, $\ell(\text{id}) = 0 = |I(\text{id})|$. The proposition holds.
    \item Suppose the proposition is true for all elements of length $k$. Let $w \in W$ with $\ell(w) = k+1$. Write $w = w's_i$ where $\ell(w')=k$ and $s_i \in S$. Since $\ell(w) > \ell(w')$, we must be in the first case of our analysis: $w'(\alpha_i) \in \Phi^+$. Then, by the previous argument, the cardinality of the inversion sets is related by $|\mathcal{I}(w)| = |\mathcal{I}(w's_i)| = |\mathcal{I}(w')|+1$. By the induction hypothesis, $|\mathcal{I}(w')|=\ell(w')=k$. Therefore, $|\mathcal{I}(w)| = k+1 =\ell (w)$.
\end{itemize}
The proposition is proven.
\end{proof}

\begin{remark}[The Bridge to the Kostant Game]
This proposition is the theoretical justification for why the Kostant game works. The game, as we shall see, consists of a sequence of moves transforming a configuration (a linear combination of simple roots) into another of greater ``height''. We will demonstrate that each valid move of the game corresponds to multiplying the current Weyl group element by a simple reflection such that the length increases. Proposition \ref{prop:longitud_inversiones} guarantees that this algebraic increase in length corresponds to the addition of a new root to the inversion set, which will manifest as an increase in the sum of the coefficients of the game configuration.
\end{remark}

\begin{lemma}[Length Criterion]
    \label{lem:criterio_longitud}
    Let $w \in W$ and $s = s_\alpha$ be a simple reflection with $\alpha \in \Delta$.
    \begin{enumerate}
        \item If $\ell(ws) > \ell(w)$, then $w(\alpha) \in \Phi^+$.
        \item If $\ell(ws) < \ell(w)$, then $w(\alpha) \in \Phi^-$.
    \end{enumerate}
\end{lemma}
\begin{proof}
    The proof follows directly from Proposition \ref{prop:longitud_inversiones} and the analysis of how the inversion set $I(ws)$ relates to $\mathcal{I}(w)$, as sketched in the proof of that proposition. If $w(\alpha) \in \Phi^+$, then $\alpha \notin I(w^{-1})$, and it can be shown that $\ell(ws) = \ell(w)+1$. If $w(\alpha) \in \Phi^-$, then $\alpha \in I(w^{-1})$, and $\ell(ws)=\ell(w)-1$.
\end{proof}

The converse of the first part of the lemma is also true, and its proof is an elegant application of the Exchange Condition \ref{thm:exchange_condition}.

\begin{lemma}[Reciprocal of the Length Criterion]
    \label{lem:reciproco_longitud}
    Let $ W $ be a Coxeter group generated by simple reflections $S$, and let $ s \in S$. For all $ w \in W $:
    $$
    w\alpha_s > 0 \implies \ell(ws) > \ell(w).
    $$
\end{lemma}

\begin{proof}
    Suppose $ w\alpha_s > 0 $. For the sake of contradiction, assume $ \ell(ws) \leq \ell(w) $.
    
    By Theorem \ref{thm:exchange_condition}, if $\ell(ws) < \ell(w)$, then there exists a reduced expression of $ w $ ending in $ s $. That is, we can write
    $$
    w = t_1 t_2 \cdots t_k s,
    $$
    where $ t_1, \ldots, t_k \in \Delta$ and $ \ell(w) = k + 1 $. Applying $ w $ to $ \alpha_s $:
    $$
    w\alpha_s = t_1 t_2 \cdots t_k (s\alpha_s) = t_1 t_2 \cdots t_k (-\alpha_s) = -\alpha_{t_1 t_2 \cdots t_k}.
    $$
    This implies that $ w\alpha_s = -\alpha_{t_1 t_2 \cdots t_k} \in -\Phi^+ $, contradicting the hypothesis $ w\alpha_s > 0 $. Therefore, $ \ell(ws) \nless \ell(w) $.
    
    It remains to rule out $ \ell(ws) = \ell(w) $. We know that $s^2 = \text{id}$, so $w = w s \cdot s$. If $ \ell(ws) = \ell(w) $, then $ws = w s \cdot s$, which would imply $ s = \text{id} $, which is impossible since $ s $ is a simple reflection. Thus, $ \ell(ws) > \ell(w) $.
\end{proof}

\subsubsection{Parabolic subgroups}

We now explore the structure of subgroups of Weyl groups $W$, specifically those which, thanks to Theorem \ref{teo:weyl_simple}, are not generated by the entire set of reflections $S$ (i.e., those associated with the simple roots $\Delta$).

\begin{definition}[Parabolic subgroup]
Given $ J \subset S$, the \textbf{parabolic subgroup} $ W_J $ is the subgroup of $ W $ generated by $ J $.
\end{definition}

In the study of the generalized Kostant game to be seen in Section \ref{Chapter3}, the Weyl group quotients $W/W_J$ play a central role. For each coset in this quotient, we are interested in its unique minimal length representative. The set of all such representatives, denoted $W^J$, admits several equivalent characterizations that are fundamental for connecting the game dynamics with the root system structure.

Naturally, we can consider the quotient of a Weyl group $W$ with one of its parabolic subgroups $W_J$. Although this quotient does not necessarily yield a group, we can analyze this structure as a set of cosets, and we are particularly interested in the minimal length representatives of these cosets.

We will now see that this set of representatives can be characterized in two different ways. The first is algebraic, based on the length function and traditionally presented as in \cite{HumphreysReflection1990}, while the second is combinatorial, based on the inversion set, which will be the key form used in later sections to work with the Kostant game.

\begin{definition}[$W^J$, Algebraic Characterization]
\label{def:WJ_longitud}
Let $J \subseteq S$. The set of minimal length representatives for the right cosets of $W_J$ in $W$ is defined as:
$$ W^J := \{w \in W \mid \ell(ws) > \ell(w) \text{ for all } s \in S_J \}, $$
where $S_J = \{s_j \mid j \in J\}$.
\end{definition}

\begin{definition}[$W^J$, Combinatorial Characterization]
\label{def:WJ_inversion}
Alternatively, $W^J$ can be defined via inversion sets:
$$ W^J := \{w \in W \mid \mathcal{I}(w) \subseteq \Phi^+ \setminus \Phi_J^+\}, $$
where $\Phi_J^+$ is the set of positive roots in the root system generated by the base $\{\alpha_j \mid j \in J\}$, that is, $ \Phi_J^+ = \Phi^+ \cap \text{Span}_{\mathbb{R}}(\{\alpha_j\}_{j \in J})$.
\end{definition}

Thanks to Lemmas \ref{lem:criterio_longitud} and \ref{lem:reciproco_longitud}, we are ready to demonstrate the equivalence between these definitions as follows.

\begin{theorem}[Equivalence of definitions of $W^J$]
\label{thm:equivalencia_WJ}
The two definitions of $W^J$ are equivalent:
$$ \{w \in W \mid \ell (ws) > \ell(w) \; \forall s \in S_J \} = \{w \in W \mid \mathcal{I}(w) \subseteq \Phi^+ \setminus \Phi_J^+ \}. $$
\end{theorem}

\begin{proof}
($ \subseteq $): Let $w\in W^J$ in the sense of Definition 1; that is, $w$ satisfies
$$
\ell(ws) > \ell(w) \quad \text{for all } s\in J.
$$
By Lemma \ref{lem:criterio_longitud}, this implies that for each $s\in J$, we have $w(\alpha_s)>0$. Recall that $\Phi^+_J$ is defined as the set of positive roots obtained as linear combinations with non-negative coefficients of the simple roots $\{\alpha_s\}_{s\in J}$, i.e.,
$$
\Phi^+_J=\Bigl\{\beta\in\Phi^+ \;\Big|\; \beta=\sum_{s\in J}c_s\,\alpha_s,\quad c_s\geq 0\Bigr\}.
$$
Since $w$ is a linear transformation, for any $\beta\in\Phi^+_J$ we have:
$$
w(\beta)=w\Bigl(\sum_{s\in J}c_s\,\alpha_s\Bigr)=\sum_{s\in J}c_s\,w(\alpha_s).
$$
As each $w(\alpha_s)$ is positive and the coefficients $c_s$ are non-negative, it follows that $w(\beta)$ is a linear combination of positive roots; that is, $w(\beta)$ is positive, so $w(\Phi^+_J)\subseteq \Phi^+$. By definition, the inversion set of $w$, denoted $\mathcal{I}(w)$, consists of the positive roots that $w$ sends to negative roots. Given the above, no root in $\Phi^+_J$ can belong to $\mathcal{I}(w)$, meaning $\mathcal{I}(w)\cap \Phi^+_J=\varnothing$. Since $\mathcal{I}(w)\subseteq \Phi^+$, it must hold that
$$
\mathcal{I}(w)\subset \Phi^+\setminus \Phi^+_J.
$$
Thus concluding that
$$
w\in\{w\in W \mid \mathcal{I}(w)\subset \Phi^+\setminus \Phi^+_J\}.
$$\\

($ \supseteq $): Let $w\in W$ such that $\mathcal{I}(w)\subset \Phi^+\setminus \Phi^+_J$. Recall that $\mathcal{I}(w)$ is the set of positive roots that $w$ sends to negative roots, i.e.,
$$
\mathcal{I}(w)=\{\beta\in \Phi^+ \mid w(\beta)\in \Phi^-\}.
$$
Furthermore, $\Phi^+_J$ is the set of positive roots obtained as linear combinations with non-negative coefficients of the simple roots $\{\alpha_s\}_{s\in J}$. In particular, for each $s\in J$, we have $\alpha_s\in \Phi^+_J$.

If for some $s\in J$ it were the case that $w(\alpha_s)\in \Phi^-$, then $\alpha_s$ would belong to $\mathcal{I}(w)$, contradicting the hypothesis $\mathcal{I}(w)\subset \Phi^+\setminus \Phi^+_J$. Therefore, for all $s\in J$, $w(\alpha_s)>0$ holds.

Now, applying Lemma \ref{lem:reciproco_longitud}, which states that for any $w\in W$ and $s\in S$ we have
$$
w(\alpha_s)>0 \implies \ell(ws)>\ell(w),
$$
we conclude that for all $s\in J$, $\ell(ws)>\ell(w)$. Thus concluding that
$$
w\in \{w\in W \mid \ell(ws)>\ell(w)\quad \forall s\in J\}.
$$
\end{proof}

One of the most important structural properties of Weyl groups is that their elements admit a unique decomposition with respect to any parabolic subgroup. This result is key to understanding the structure of the quotients $W/W_J$ and will be the basis for demonstrating how the set of positive roots is partitioned.

\begin{theorem}[Parabolic Decomposition Theorem]
    \label{thm:descomposicion_parabolica}
    Let $W_J$ be a parabolic subgroup of $W$. Every element $w \in W$ admits a \textbf{unique decomposition} of the form:
    $$ w = w^J \cdot w_J $$
    where $w^J \in W^J$ (it is a minimal length representative) and $w_J \in W_J$. Furthermore, length is additive in this decomposition:
    $$ \ell(w) = \ell(w^J) + \ell(w_J) $$
\end{theorem}

\begin{proof}
    Suppose an element $w$ admits two additive decompositions:
    $$ w = u_1 v_1 = u_2 v_2 $$
    where $u_1, u_2 \in W^J$, $v_1, v_2 \in W_J$, and $\ell(w) = \ell(u_1) + \ell(v_1) = \ell(u_2) + \ell(v_2)$.
    
    Rearranging the equation, we obtain:
    $$ u_2^{-1} u_1 = v_2 v_1^{-1} $$
    Let us call this element $z$. Since $v_1, v_2 \in W_J$, their product $z = v_2 v_1^{-1}$ must also be in $W_J$.
    We will now show that $z$ must also be the identity by proving that $\ell(z) = 0$.
    
    Consider the expression $u_1 = u_2 z$. By the additivity of length in the decomposition of $w$, we know that:
    $$ \ell(u_1) + \ell(v_1) = \ell(u_2) + \ell(v_2) $$
    Since $z = v_2 v_1^{-1}$, we have $v_2 = z v_1$. The length of a product is always less than or equal to the sum of the lengths, so:
    $$ \ell(v_2) \le \ell(z) + \ell(v_1) $$
    Substituting this into the length equation:
    $$ \ell(u_1) + \ell(v_1) \le \ell(u_2) + \ell(z) + \ell(v_1) $$
    Canceling $\ell(v_1)$, we obtain:
    $$ \ell(u_1) \le \ell(u_2) + \ell(z) $$
    Similarly, starting from $u_2 = u_1 z^{-1}$, we can prove the inverse inequality: $\ell(u_2) \le \ell(u_1) + \ell(z^{-1})$. Given that $\ell(z) = \ell(z^{-1})$, this is $\ell(u_2) \le \ell(u_1) + \ell(z)$.
    
    Now, recall that $z \in W_J$. Consider the expression $u_1 = u_2 z$. Since $u_2 \in W^J$, by definition $\ell(u_2 s) > \ell(u_2)$ for all $s \in J$. This property can be extended by induction to show that:
    $$ \ell(u_2 v) = \ell(u_2) + \ell(v) \quad \text{for all } v \in W_J. $$
    Applying this property to our case with $v=z$, we have:
    $$ \ell(u_1) = \ell(u_2 z) = \ell(u_2) + \ell(z). $$
    We had already deduced that $\ell(u_1) \le \ell(u_2) + \ell(z)$. The only way both statements can be true is if equality holds.
    Substituting this equality into the inequality we obtained earlier ($\ell(u_2) \le \ell(u_1) + \ell(z)$):
    $$ \ell(u_2) \le (\ell(u_2) + \ell(z)) + \ell(z) $$
    $$ 0 \le 2\ell(z) $$
    This is always true. However, if we use the first equality $\ell(u_1) = \ell(u_2)+\ell(z)$ and compare it with $\ell(u_2) \le \ell(u_1) + \ell(z)$, upon substituting $\ell(u_1)$ we obtain:
    $$ \ell(u_2) \le (\ell(u_2)+\ell(z)) + \ell(z) = \ell(u_2) + 2\ell(z) \implies 0 \le 2\ell(z). $$
    The key lies in symmetry. We can also write $u_2 = u_1 z^{-1}$, and since $z \in W_J \implies z^{-1} \in W_J$, the same logic of additivity applies:
    $$ \ell(u_2) = \ell(u_1 z^{-1}) = \ell(u_1) + \ell(z^{-1}) = \ell(u_1) + \ell(z). $$
    Now we have a system of two equations:
    \begin{enumerate}
        \item $\ell(u_1) = \ell(u_2) + \ell(z)$
        \item $\ell(u_2) = \ell(u_1) + \ell(z)$
    \end{enumerate}
    Substituting (1) into (2):
    $$ \ell(u_2) = (\ell(u_2) + \ell(z)) + \ell(z) = \ell(u_2) + 2\ell(z) $$
    Subtracting $\ell(u_2)$ from both sides, we obtain:
    $$ 0 = 2\ell(z) $$
    The only way this is possible is if $\ell(z) = 0$. If the length of an element is zero, that element must be the identity, $z = \mathrm{id}$.
    
    If $z=\mathrm{id}$, then $(w_2^J)^{-1} w_1^J = \mathrm{id}$ and $w_{2,J} (w_{1,J})^{-1} = \mathrm{id}$, from which it follows that $w_1^J = w_2^J$ and $w_{1,J} = w_{2,J}$. The decomposition is unique.
\end{proof}

Beyond defining a measure for words, the length function $\ell(w)$ induces a partial order structure on the Weyl group, known as the \textbf{Bruhat order} (or Bruhat-Chevalley order). There are several equivalent definitions. One of the most intuitive is based on reduced expressions.

\begin{definition}[Bruhat Order]
\label{def:bruhat_order}
Let $u, w \in W$. We say that $u$ is less than or equal to $w$ in the Bruhat order, denoted $u \le w$, if given a reduced expression for $w$, $w = s_{i_1} \cdots s_{i_t}$, there exists a subword (obtained by deleting some generators) $s_{j_1} \cdots s_{j_k}$ that is a reduced expression for $u$.
\end{definition}

\begin{remark}[Chain Definition]
Equivalently, $u \le w$ if and only if there exists a chain of elements $w_0, w_1, \dots, w_k$ in $W$ such that:
$$ u = w_0 \to w_1 \to \cdots \to w_k = w $$
where each step $w_j = w_{j-1}s$ for some simple reflection $s \in S$, and the length increases exactly by one at each step: $\ell(w_j) = \ell(w_{j-1}) + 1$.
\end{remark}

This order has very important structural properties. The identity element, $\text{id}$, is the unique minimum of the poset (i.e., $\text{id} \le w$ for all $w \in W$). Furthermore, if $W$ is a finite Coxeter group (like all Weyl groups we consider), there exists a unique maximum element, denoted $w_0$. This element is the longest possible and is above all others in the Bruhat order.

The existence of this maximum element and the ``ascending chain'' structure of the Bruhat order are precisely the tools we need to understand the property presented below.

When the subset of generators $J \subseteq S$ corresponds to a finite set, the parabolic subgroup $W_J$ is itself a finite Coxeter group. As such, it inherits the structural properties of these groups, including the existence of a unique element of maximum length.

\begin{definition}[Longest Element of $W_J$]
If the parabolic subgroup $W_J$ is finite, there exists a unique element $w_{0,J} \in W_J$ of maximum length. This element is characterized by being the unique $w \in W_J$ such that $\ell(ws) < \ell(w)$ for all $s \in J$. Geometrically, $w_{0,J}$ is the element that transforms every positive root of the subsystem $\Phi_J^+$ into a negative root: $w_{0,J}(\Phi_J^+) = \Phi_J^-$.
\end{definition}

This maximum element is not merely a curiosity; it is intimately linked to all other elements of the subgroup through an elegant combinatorial property. This property, known as the \textbf{subword property} (or completion property), is a classic result of the theory (see, for example, \cite[§1.8]{HumphreysReflection1990}) and asserts that any reduced expression for an element of $W_J$ can be viewed as the ``start'' of a reduced expression for the longest element.

\begin{theorem}[Completion of Reduced Expressions]
\label{thm:completacion_w0j}
Let $W_J$ be a finite parabolic subgroup and let $w \in W_J$. Every reduced expression for $w$, $w = s_1 s_2 \cdots s_k$ (with $s_i \in J$), can be extended to a reduced expression for the longest element $w_{0,J}$. That is, there exist generators $s_{k+1}, \dots, s_m \in J$ such that the expression $w_{0,J} = s_1 s_2 \cdots s_k s_{k+1} \cdots s_m$ is reduced.
\end{theorem}

\begin{proof}
The process is algorithmic and relies on Lemma \ref{lem:criterio_longitud}. If an element $w \in W_J$ is not the longest element $w_{0,J}$, by definition there must exist at least one generator $s \in J$ such that $\ell(ws) > \ell(w)$. In fact, $\ell(ws) = \ell(w)+1$. Selecting one of these generators and appending it to the right of the expression for $w$, we obtain a reduced expression for a new element $w' = ws$, which is ``closer'' to $w_{0,J}$ in the Bruhat order of $W_J$. This process can be repeated until it is no longer possible to increase the length, at which point $w_{0,J}$ has necessarily been reached.
\end{proof}

\begin{example}[Completion in type $A_3$]
Consider the Weyl group $W = W(A_3) \cong S_4$, with generators $S = \{s_1, s_2, s_3\}$. Let $J = \{s_1, s_2\}$. The parabolic subgroup $W_J$ is generated by $s_1$ and $s_2$, and is isomorphic to $W(A_2) \cong S_3$.

The longest element in $W_J$ is $w_{0,J} = s_1s_2s_1 = s_2s_1s_2$, with $\ell(w_{0,J}) = 3$.

Take an arbitrary element of $W_J$, for example, $w = s_2$, whose reduced expression is simply $s_2$ and $\ell(s_2)=1$. To complete its expression, we look for an $s \in J=\{s_1, s_2\}$ such that the length increases:
\begin{itemize}
    \item Multiply by $s_2$: $\ell(s_2 s_2) = \ell(\text{id}) = 0 < 1$. Does not work.
    \item Multiply by $s_1$: $\ell(s_2 s_1) = 2 > 1$. This works, and our new element is $w' = s_2s_1$.
\end{itemize}
Now, start from $w' = s_2s_1$, with $\ell(w')=2$. Repeat the process:
\begin{itemize}
    \item Multiply by $s_1$: $\ell(s_2 s_1 s_1) = \ell(s_2) = 1 < 2$. Does not work.
    \item Multiply by $s_2$: $\ell(s_2 s_1 s_2) = 3 > 2$. Perfect.
\end{itemize}
We have reached the element $s_2s_1s_2$, whose length is 3, the maximum possible in $W_J$. Therefore, we have completed the reduced expression $s_2$ to the reduced expression $s_2s_1s_2$ for $w_{0,J}$. This illustrates that the word ``$s_2$'' is a prefix of the word ``$s_2s_1s_2$''.
\end{example}

\section{The Kostant game}
\label{Chapter2}

A problem that may arise when studying root systems is determining in a simple way which linear combinations of simple roots remain roots within the system. The ``game of finding the highest root,'' also known as the Kostant game, provides a solution to this problem. This game does not appear in conventional literature but is discussed in a MathOverflow post where mathematician Allen Knutson discusses a combinatorial and constructive way to generate all positive roots via the Kostant game \cite{Nie2014}.

\subsection{Definition of the game and examples}
\label{sec2.1}

The main motivation behind the following game is that, given a root system, the highest root is a linear combination of simple roots with positive integer coefficients. Starting from this highest root, we can obtain other roots by systematically subtracting simple roots. This procedure generates all positive roots of the system in an orderly and structured manner. The Kostant game offers a constructive method to visualize the structure of the root system without needing to resort directly to its enumeration. Furthermore, it allows for the identification of linear combinations of simple roots that effectively result in roots.

Although this game can be defined for arbitrary graphs, including those with multiple edges, to begin familiarizing ourselves with it, we will initially ignore its relationship with root systems and focus on the particular case of simple graphs as worked on in \cite{chen2017}.

\begin{definition}[Kostant Game]
\label{def:kostant_game}
Let $ \Gamma = (V, E) $ be a simple graph, with $V$ the set of vertices $i$ enumerated $1\leq i \leq n$ and $E$ the set of edges. For $ i \in V $, let $ N(i) $ be the set of neighboring vertices of $ i $; that is, for all $j \in N(i)$ there exists an edge joining $i$ with $j$. A \textbf{configuration} is defined as the vector $c = (c_i)_{1 \leq i \leq n} $ whose entries satisfy $ c_i \in \mathbb{N}$. This vector can be interpreted as an arrangement of chips on the graph $\Gamma$, where the entry $c_i$ represents the amount of chips located at vertex $i$. We say that a vertex $ i $ is:
\begin{itemize}
    \item \textbf{Happy} if $ c_i = \dfrac{1}{2} \sum_{j \in N(i)} c_j $.
    \item \textbf{Sad} if $ c_i < \dfrac{1}{2} \sum_{j \in N(i)} c_j $.
    \item \textbf{Excited} if $ c_i > \dfrac{1}{2} \sum_{j \in N(i)} c_j $.
\end{itemize}

The goal of the game is for all vertices to be happy or excited. The game is played as follows: Initially, there are no chips present (therefore, $ c_i = 0 $ for all $ i \in V $, and all vertices are happy). Then, we place a chip on a vertex $i_0$, setting $ c_{i_0} = 1 $, so that $ v_{i_0} $ is excited, but the neighbors of $ i_0 $ are unhappy (sad). Subsequently, we perform the following ``reflection'':\\
Choose any sad vertex $i$, and replace $c_i$ according to the rule:
    $$
    c_i \mapsto -c_i + \sum_{j \in N(i)} c_j.
    $$
These reflections are repeated on every sad vertex until all vertices become happy or excited, at which point the game ends. If it is not possible to reach such a configuration, we say that the graph is Kostant infinite.
\end{definition}

\begin{remark}[On the factor $\dfrac{1}{2}$]
    It is important to note the factor $\dfrac{1}{2}$ in the definition of the vertex states. This normalization, although not standard in the general theory of chip-firing games, is essential for connecting the dynamics of the game with the geometry of the root systems we wish to study and to guarantee that the game progresses unequivocally.
    
    Specifically, it ensures that applying the reflection rule on a sad vertex necessarily transforms it into a happy or excited vertex, avoiding immediate cycles. If a vertex $i$ is sad, it holds that $c_i < \dfrac{1}{2} \sum_{j \in N(i)} c_j$. Its new value will be $c'_i = -c_i + \sum_{j \in N(i)} c_j$. The original sadness condition implies that $2c_i < \sum_{j \in N(i)} c_j$, which in turn guarantees that:
    $$ c'_i = -c_i + \sum_{j \in N(i)} c_j > -c_i + 2c_i = c_i $$
    More importantly, it holds that $c'_i > \dfrac{1}{2} \sum_{j \in N(i)} c_j$, converting vertex $i$ into an excited vertex and ensuring that the game always advances towards configurations of greater ``height''.
\end{remark}

\begin{remark}
	In Section \ref{java}, we present a Java application that allows simulating and visualizing the Kostant game just defined and all other variants discussed throughout this work. This tool facilitates the analysis of different types of graphs where the game either terminates or remains indefinitely with sad vertices.
\end{remark}

Let us start by looking at two examples illustrating the operation of the game on graphs where the game starts but does not end; that is, a configuration where all vertices are happy or excited is never reached. In these representations, each vertex will contain the chips assigned to it, and the arrows show the transition between configurations.

\begin{example}[The cycle graph $\tilde{A}_2$]
The graph $\tilde{A}_2$ is a simple triangle. Let us see how a sequence of moves leads to a perpetual growth of values.

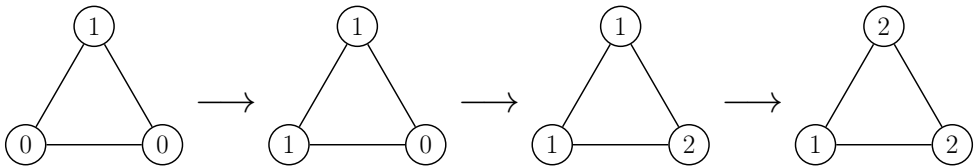
\begin{figure}[H]
\centering
\scalebox{0.7}{
\begin{tikzpicture}

    \begin{scope}[xshift=0cm]
        \node[circle, draw, thick, minimum size=20pt, font=\Large] (v1) at (90:1.5) {1};
        \node[circle, draw, thick, minimum size=20pt, font=\Large] (v2) at (210:1.5) {0};
        \node[circle, draw, thick, minimum size=20pt, font=\Large] (v3) at (330:1.5) {0};
        \draw[thick] (v1) -- (v2) -- (v3) -- (v1);
    \end{scope}

    \node at (2.5, 0) {\huge$\longrightarrow$};

    \begin{scope}[xshift=5cm]
        \node[circle, draw, thick, minimum size=20pt, font=\Large] (v1) at (90:1.5) {1};
        \node[circle, draw, thick, minimum size=20pt, font=\Large] (v2) at (210:1.5) {1};
        \node[circle, draw, thick, minimum size=20pt, font=\Large] (v3) at (330:1.5) {0};
        \draw[thick] (v1) -- (v2) -- (v3) -- (v1);
    \end{scope}

    \node at (7.5, 0) {\huge$\longrightarrow$};

    \begin{scope}[xshift=10cm]
        \node[circle, draw, thick, minimum size=20pt, font=\Large] (v1) at (90:1.5) {1};
        \node[circle, draw, thick, minimum size=20pt, font=\Large] (v2) at (210:1.5) {1};
        \node[circle, draw, thick, minimum size=20pt, font=\Large] (v3) at (330:1.5) {2};
        \draw[thick] (v1) -- (v2) -- (v3) -- (v1);
    \end{scope}

    \node at (12.5, 0) {\huge$\longrightarrow$};

    \begin{scope}[xshift=15cm]
        \node[circle, draw, thick, minimum size=20pt, font=\Large] (v1) at (90:1.5) {2};
        \node[circle, draw, thick, minimum size=20pt, font=\Large] (v2) at (210:1.5) {1};
        \node[circle, draw, thick, minimum size=20pt, font=\Large] (v3) at (330:1.5) {2};
        \draw[thick] (v1) -- (v2) -- (v3) -- (v1);
    \end{scope}

\end{tikzpicture}

}
\caption{A sequence of moves of the Kostant game on the graph $\tilde{A}_2$.}
\end{figure}

After the third configuration, the game is far from over. The next move would be on the vertex that has a single chip, updating it to $4-1=3$. From there, the next vertex to be played can be any that has two chips, which would become $(2+3)-2=3$, and so on. In each turn, there is always a sad vertex, and the values continue to grow in an endless cycle. This visually demonstrates that the game on $\tilde{A}_2$ will never terminate.
\end{example}

\begin{example}[The star graph $\tilde{D}_4$]
	The graph $\tilde{D}_4$ has a central vertex ($v_0$) connected to four peripheral vertices ($v_1, v_2, v_3, v_4$).
	
	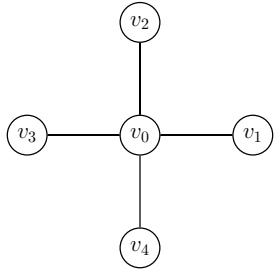
\begin{figure}[H]
		\centering
		\scalebox{0.6}{
			\begin{tikzpicture}
				\node[circle, draw, thick, minimum size=20pt, font=\Large] (v0) at (0,0) {$v_0$};
				\node[circle, draw, thick, minimum size=20pt, font=\Large] (v1) at (0:2.5) {$v_1$};
				\node[circle, draw, thick, minimum size=20pt, font=\Large] (v2) at (90:2.5) {$v_2$};
				\node[circle, draw, thick, minimum size=20pt, font=\Large] (v3) at (180:2.5) {$v_3$};
				\node[circle, draw, thick, minimum size=20pt, font=\Large] (v4) at (270:2.5) {$v_4$};
				\draw[thick] (v0) -- (v1) -- (v0) -- (v2) -- (v0) -- (v3) -- (v0) -- (v4);
			\end{tikzpicture}
		}
		\caption{The structure of the graph $\tilde{D}_4$.}
	\end{figure}
	
	Instead of showing every intermediate step, let us see how a simple initial configuration evolves into a state where the game must continue and the values have already grown. We start with $\mathbf{v} = (0, 1, 0, 0, 0)$ for the vertices $(v_0, v_1, v_2, v_3, v_4)$ and perform a sequence of moves until reaching the state $(1,1,1,1,1)$.
	
	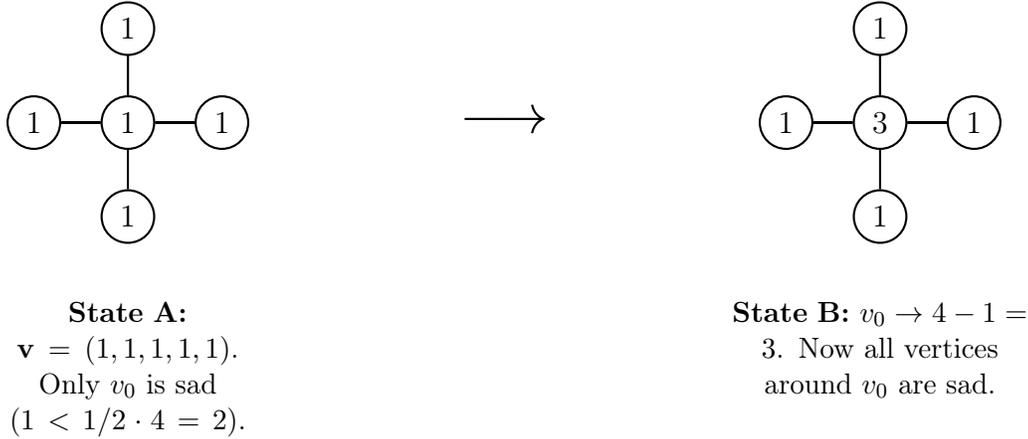
\begin{figure}[H]
		\centering
		
		\begin{tikzpicture}
			\begin{scope}[xshift=-5cm, scale=0.5]
				\node[circle, draw, thick, minimum size=10pt, font=\large] (v0) at (0,0) {1};
				\node[circle, draw, thick, minimum size=10pt, font=\large] (v1) at (0:2.5) {1};
				\node[circle, draw, thick, minimum size=10pt, font=\large] (v2) at (90:2.5) {1};
				\node[circle, draw, thick, minimum size=10pt, font=\large] (v3) at (180:2.5) {1};
				\node[circle, draw, thick, minimum size=10pt, font=\large] (v4) at (270:2.5) {1};
				\draw[thick] (v0) -- (v1) -- (v0) -- (v2) -- (v0) -- (v3) -- (v0) -- (v4);
				\node[text width=4cm, text centered, below=0.5cm] at (0,-3.5) {\textbf{State A:} $\mathbf{v}=(1,1,1,1,1)$. Only $v_0$ is sad ($1 < 1/2 \cdot 4 = 2$).};
			\end{scope}
			
			\node at (0, 0) {\huge$\longrightarrow$};
			
			
			\begin{scope}[xshift=5cm, scale=0.5]
				\node[circle, draw, thick, minimum size=10pt, font=\large] (v0) at (0,0) {3};
				\node[circle, draw, thick, minimum size=10pt, font=\large] (v1) at (0:2.5) {1};
				\node[circle, draw, thick, minimum size=10pt, font=\large] (v2) at (90:2.5) {1};
				\node[circle, draw, thick, minimum size=10pt, font=\large] (v3) at (180:2.5) {1};
				\node[circle, draw, thick, minimum size=10pt, font=\large] (v4) at (270:2.5) {1};
				\draw[thick] (v0) -- (v1) -- (v0) -- (v2) -- (v0) -- (v3) -- (v0) -- (v4);
				\node[text width=4cm, text centered, below=0.5cm] at (0,-3.5) {\textbf{State B:} $v_0 \to 4-1=3$. Now all vertices around $v_0$ are sad.};
			\end{scope}
		\end{tikzpicture}
		\caption{Evolution of the Kostant game on $\tilde{D}_4$.}
	\end{figure}
	
	In \textbf{State B}, with the configuration $\mathbf{v}=(3,1,1,1,1)$, the situation is as follows:
	\begin{itemize}
		\item \textbf{Central vertex $v_0$:} Its value is 3. The sum of its neighbors is $1+1+1+1=4$. Since $3 > 4/2 =2$, $v_0$ is excited.
		\item \textbf{Peripheral vertices $v_1, v_2, v_3, v_4$:} Each has value 1. Their only neighbor is $v_0$, which has value 3. Since $1 < 3/2 = 1.5$, all peripheral vertices are sad; consequently, each of them must increase its value to $3-1=2$.
		\item \textbf{Infinite cycle of sadness:} Since all peripheral vertices will increase their value to $2$, the central vertex will become sad again, as $3 < (2 + 2 + 2 + 2)/2 = 4$, then $v_0$ will be assigned a value of $5$ and the sadness condition for the peripheral vertices will repeat.
	\end{itemize}
	
	We have reached a cyclic situation of instability: upon firing the peripheral vertices, they will increase their value and make the central vertex sad again; when the center becomes excited, it will induce sadness in its neighbors. Thus, indefinitely, any future move will only increase the total sum of the graph's values and the situation will repeat. This confirms that the game on $\tilde{D}_4$ will never end.
\end{example}

\begin{remark}
    These last examples actually show us that if a graph possesses cycles or vertices of degree greater than or equal to 4, the Kostant game will not be able to terminate. This will be proven later within a more general theorem.
\end{remark}

\begin{definition}[Kostant finite graphs]
We say that a graph $\Gamma$ is of Kostant finite type if the Kostant game terminates when played on $\Gamma$.    
\end{definition}

For this last definition to make sense, it should be the case that if there exists a way to play such that the game terminates, then any sequence of moves eventually leads to a terminal state. Furthermore, the final configuration vector should not depend on the choice of moves, nor on the initial vertex on which a chip was added. Properties that we will see later are true.

\begin{definition}[Configuration graph of the Kostant game]
    \label{def:grafoconfig}
    Let $\Gamma=(V,E)$ be a graph on which the Kostant game is played. We denote by $\mathcal{C}(\Gamma)$ the associated \textbf{configuration graph}, defined as the graph whose vertices are the set of valid chip distributions when playing on $\Gamma$ and whose edges represent valid moves between states of the graph. Thus, each legal move of the Kostant game corresponds to traversing an edge in $\mathcal{C}(\Gamma)$.
\end{definition}

\begin{example}
Below are illustrative examples of the game on graphs that are Kostant finite, one on $A_4$ and another on $D_4$, reviewing the entire set of possible configurations on this graph with their respective final configuration.


\begin{figure}[H]
\centering
\begin{tikzpicture}[scale=0.9, transform shape]
	
  \node[draw, shape=circle, scale=0.7] (a) at (-7, 0) {$1$};
  \node[draw, fill=white, shape=circle, scale=0.7] (b) at (-6, 0) {\phantom{$0$}};
  \node[draw, fill=white, shape=circle, scale=0.7] (c) at (-5, 0) {\phantom{$0$}};
  \node[draw, fill=white, shape=circle, scale=0.7] (d) at (-4, 0) {\phantom{$0$}};
  \draw (a) -- (b) -- (c) -- (d);
  \draw[->, thick] (-5, -0.5) -- (-4, -1.5);

  \node[draw, shape=circle, scale=0.7] (a) at (-2, 0) {\phantom{$0$}};
  \node[draw, fill=white, shape=circle, scale=0.7] (b) at (-1, 0) {$1$};
  \node[draw, fill=white, shape=circle, scale=0.7] (c) at (0, 0) {\phantom{$0$}};
  \node[draw, fill=white, shape=circle, scale=0.7] (d) at (1, 0) {\phantom{$0$}};
  \draw (a) -- (b) -- (c) -- (d);
  \draw[->, thick] (-1, -0.5) -- (-2, -1.5);
  \draw[->, thick] (0, -0.5) -- (1, -1.5);

  \node[draw, shape=circle, scale=0.7] (a) at (3, 0) {\phantom{$0$}};
  \node[draw, fill=white, shape=circle, scale=0.7] (b) at (4, 0) {\phantom{$0$}};
  \node[draw, fill=white, shape=circle, scale=0.7] (c) at (5, 0) {$1$};
  \node[draw, fill=white, shape=circle, scale=0.7] (d) at (6, 0) {\phantom{$0$}};
  \draw (a) -- (b) -- (c) -- (d);
  \draw[->, thick] (4, -0.5) -- (3, -1.5);
  \draw[->, thick] (5, -0.5) -- (6, -1.5);

  \node[draw, shape=circle, scale=0.7] (a) at (8, 0) {\phantom{$0$}};
  \node[draw, fill=white, shape=circle, scale=0.7] (b) at (9, 0) {\phantom{$0$}};
  \node[draw, fill=white, shape=circle, scale=0.7] (c) at (10, 0) {\phantom{$0$}};
  \node[draw, fill=white, shape=circle, scale=0.7] (d) at (11, 0) {$1$};
  \draw (a) -- (b) -- (c) -- (d);
  \draw[->, thick] (9, -0.5) -- (8, -1.5);


  \node[draw, shape=circle, scale=0.7] (a) at (-4.5, -2) {$1$};
  \node[draw, fill=white, shape=circle, scale=0.7] (b) at (-3.5, -2) {$1$};
  \node[draw, fill=white, shape=circle, scale=0.7] (c) at (-2.5, -2) {\phantom{$0$}};
  \node[draw, fill=white, shape=circle, scale=0.7] (d) at (-1.5, -2) {\phantom{$0$}};
  \draw (a) -- (b) -- (c) -- (d);
  \draw[->, thick] (-2.5, -2.5) -- (-1.5, -3.5);

  \node[draw, shape=circle, scale=0.7] (a) at (0.5, -2) {\phantom{$0$}};
  \node[draw, fill=white, shape=circle, scale=0.7] (b) at (1.5, -2) {$1$};
  \node[draw, fill=white, shape=circle, scale=0.7] (c) at (2.5, -2) {$1$};
  \node[draw, fill=white, shape=circle, scale=0.7] (d) at (3.5, -2) {\phantom{$0$}};
  \draw (a) -- (b) -- (c) -- (d);
  \draw[->, thick] (1.5, -2.5) -- (0.5, -3.5);
  \draw[->, thick] (2.5, -2.5) -- (3.5, -3.5);

  \node[draw, shape=circle, scale=0.7] (a) at (5.5, -2) {\phantom{$0$}};
  \node[draw, fill=white, shape=circle, scale=0.7] (b) at (6.5, -2) {\phantom{$0$}};
  \node[draw, fill=white, shape=circle, scale=0.7] (c) at (7.5, -2) {$1$};
  \node[draw, fill=white, shape=circle, scale=0.7] (d) at (8.5, -2) {$1$};
  \draw (a) -- (b) -- (c) -- (d);
  \draw[->, thick] (6.5, -2.5) -- (5.5, -3.5);


  \node[draw, shape=circle, scale=0.7] (a) at (-2, -4) {$1$};
  \node[draw, fill=white, shape=circle, scale=0.7] (b) at (-1, -4) {$1$};
  \node[draw, fill=white, shape=circle, scale=0.7] (c) at (0, -4) {$1$};
  \node[draw, fill=white, shape=circle, scale=0.7] (d) at (1, -4) {\phantom{$0$}};
  \draw (a) -- (b) -- (c) -- (d);
  \draw[->, thick] (0, -4.5) -- (1, -5.5);

  \node[draw, shape=circle, scale=0.7] (a) at (3, -4) {\phantom{$0$}};
  \node[draw, fill=white, shape=circle, scale=0.7] (b) at (4, -4) {$1$};
  \node[draw, fill=white, shape=circle, scale=0.7] (c) at (5, -4) {$1$};
  \node[draw, fill=white, shape=circle, scale=0.7] (d) at (6, -4) {$1$};
  \draw (a) -- (b) -- (c) -- (d);
  \draw[->, thick] (4, -4.5) -- (3, -5.5);


  \node[draw, shape=circle, scale=0.7] (a) at (0.5, -6) {$1$};
  \node[draw, fill=white, shape=circle, scale=0.7] (b) at (1.5, -6) {$1$};
  \node[draw, fill=white, shape=circle, scale=0.7] (c) at (2.5, -6) {$1$};
  \node[draw, fill=white, shape=circle, scale=0.7] (d) at (3.5, -6) {$1$};
  \draw (a) -- (b) -- (c) -- (d);
\end{tikzpicture}
\caption{The set of possible configurations of the Kostant game on $A_4$.}
\label{fig:kostant_a4}
\end{figure}
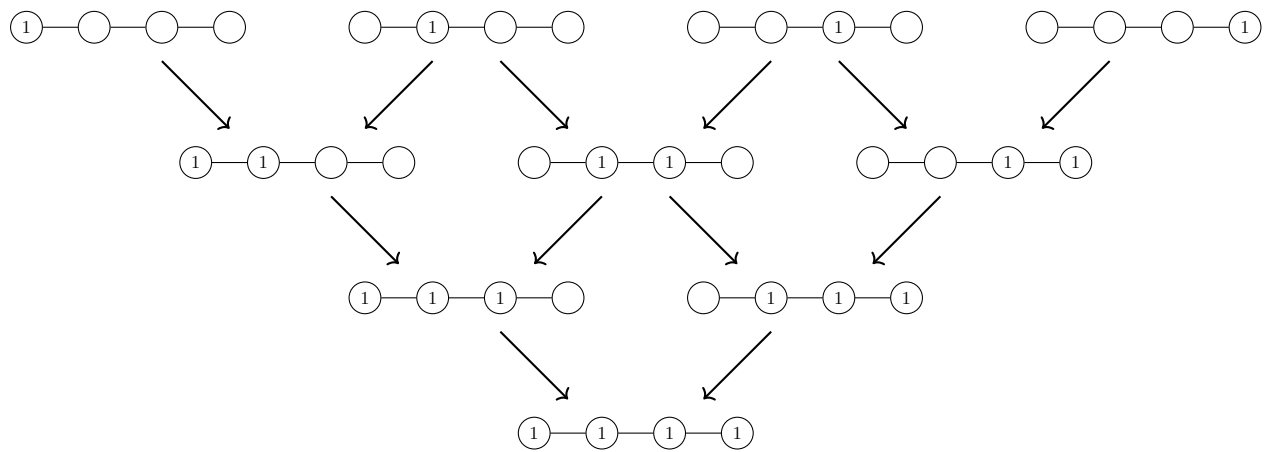

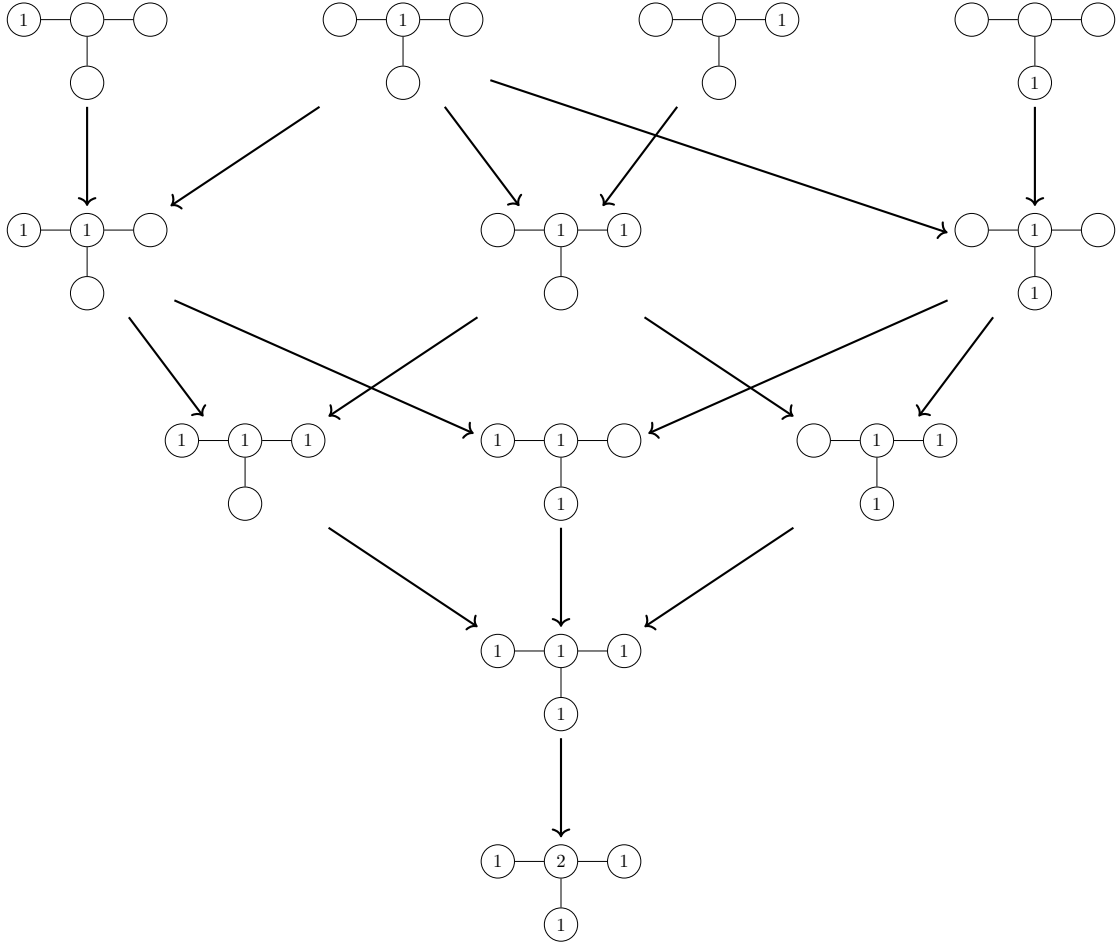
\begin{figure}[H]
    \centering

    \tikzset{
        vtx/.style={circle, draw, fill=white, inner sep=1pt, minimum size=18pt, font=\small},
    }

    \newcommand{\DfourConfig}[4]{
        \begin{tikzpicture}[remember picture]
            \node[vtx] (v1) at (0,0) {$#1$};
            \node[vtx] (v2) at (1.2,0) {$#2$};
            \node[vtx] (v3) at (2.4,0) {$#3$};
            \node[vtx] (v4) at (1.2,-1.2) {$#4$};
            \draw (v1) -- (v2) -- (v3);
            \draw (v2) -- (v4);
        \end{tikzpicture}
    }
    
    \begin{tikzpicture}[scale=0.7, transform shape]
        \node (c1000) at (-9, 10) {\DfourConfig{1}{}{}{}};
        \node (c0100) at (-3, 10) {\DfourConfig{}{1}{}{}};
        \node (c0010) at (3, 10)  {\DfourConfig{}{}{1}{}};
        \node (c0001) at (9, 10)  {\DfourConfig{}{}{}{1}};

        \node (c1100) at (-9, 6) {\DfourConfig{1}{1}{}{}};
        \node (c0110) at (0, 6)  {\DfourConfig{}{1}{1}{}};
        \node (c0101) at (9, 6)  {\DfourConfig{}{1}{}{1}};

        \node (c1110) at (-6, 2) {\DfourConfig{1}{1}{1}{}};
        \node (c1101) at (0, 2)  {\DfourConfig{1}{1}{}{1}};
        \node (c0111) at (6, 2)  {\DfourConfig{}{1}{1}{1}};
        
        \node (c1111) at (0, -2) {\DfourConfig{1}{1}{1}{1}};

        \node (c1211) at (0, -6) {\DfourConfig{1}{2}{1}{1}};

        \draw[->, thick] (c1000) -- (c1100);
        \draw[->, thick] (c0100) -- (c1100);
        
        \draw[->, thick] (c0100) -- (c0110);
        \draw[->, thick] (c0010) -- (c0110);
        
        \draw[->, thick] (c0100) -- (c0101);
        \draw[->, thick] (c0001) -- (c0101);

        \draw[->, thick] (c1100) -- (c1110);
        \draw[->, thick] (c0110) -- (c1110);

        \draw[->, thick] (c1100) -- (c1101);
        \draw[->, thick] (c0101) -- (c1101);
        
        \draw[->, thick] (c0110) -- (c0111);
        \draw[->, thick] (c0101) -- (c0111);

        \draw[->, thick] (c1110) -- (c1111);
        \draw[->, thick] (c1101) -- (c1111);
        \draw[->, thick] (c0111) -- (c1111);
        
        \draw[->, thick] (c1111) -- (c1211);

    \end{tikzpicture}
    \caption{The set of possible configurations of the Kostant game on $D_4$.}
    \label{fig:kostant_d4}
\end{figure}
\end{example}

\subsection{Game dynamics and unique termination}

When reviewing the set of configurations played on a graph where the game terminates, one arrives at a unique possible final configuration. This is an intrinsic property of the Kostant game when implemented on simply-laced Dynkin diagrams \cite[Theorem 4.9]{chen2017}, a property that in itself allows for the characterization of these diagrams. We will now examine this and some other dynamic properties inherent to the implementation of this game.

\begin{theorem}
    \label{theorem:inicio-independiente}
    Let $\Gamma=(V,E)$ be a simple connected graph. The final state of the Kostant game on $\Gamma$ is independent of the initial vertex $v_0 \in V$ chosen. That is, if $F(C)$ denotes the final configuration reached starting from an initial configuration $C$, and $C(v_0)$ is the configuration where vertex $v_0$ has value 1 and the others have value 0, then $F(C(v_a)) = F(C(v_b))$ for any $v_a, v_b \in V$.
\end{theorem}

\begin{proof}
    To prove the theorem, it suffices to show that for any two initial vertices $v_a$ and $v_b$, their initial configurations $C(v_a)$ and $C(v_b)$ can reach a common intermediate configuration $C'$. Since the game is deterministic starting from any configuration (in the sense of the set of reachable states), if both can reach the same intermediate state, they will necessarily reach the same final state.
    
    Since the graph $\Gamma$ is connected, there always exists a path between $v_a$ and $v_b$.
    
    \textbf{Case 1: $v_a$ and $v_b$ are adjacent.}
    
    Suppose that $\{v_a, v_b\} \in E$.
    \begin{enumerate}
        \item We start from the initial configuration $C(v_a)$, where the value function $c: V \to \mathbb{Z}$ is defined as $c(v_a) = 1$ and $c(v) = 0$ for all $v \neq v_a$.
        
        The vertex $v_b$ is ``sad'', since its value $c(v_b) = 0$ is less than half the sum of the values of its neighbors, which is at least $\dfrac{1}{2} c(v_a) = \dfrac{1}{2}$.
        
        We apply the game rule on $v_b$. Its new value $c'(v_b)$ will be:
        $$ c'(v_b) = \left( \sum_{v' \sim v_b} c(v') \right) - c(v_b) = c(v_a) - c(v_b) = 1 - 0 = 1 $$
        This leads us to an intermediate configuration $C'$, where $c'(v_a) = 1$ and $c'(v_b) = 1$.
        
        \item Now, we start from the initial configuration $C(v_b)$, where $c(v_b) = 1$ and $c(v) = 0$ for all $v \neq v_b$.
        
        Analogously, the vertex $v_a$ is ``sad''. We apply the game rule on $v_a$. Its new value $c''(v_a)$ will be:
        $$ c''(v_a) = \left( \sum_{v' \sim v_a} c(v') \right) - c(v_a) = c(v_b) - c(v_a) = 1 - 0 = 1 $$
        This leads us to a configuration $C''$, where $c''(v_b) = 1$ and $c''(v_a) = 1$. It is evident that configurations $C'$ and $C''$ are identical.
    \end{enumerate}
    We have demonstrated that both initial configurations can reach the same intermediate configuration.

    \textbf{Case 2: $v_a$ and $v_b$ are not adjacent.}
    The proof proceeds by induction on the length $k$ of the shortest path between $v_a$ and $v_b$.

    If the path length is 1, the vertices are adjacent. This is precisely Case 1, which has already been proven.

    Suppose now that the theorem is true for any pair of vertices at a distance less than $k$. Let $v_a$ and $v_b$ be two vertices at distance $k > 1$. Let $P = (v_a = u_0, u_1, \dots, u_k = v_b)$ be a minimal path.

    Consider vertices $v_a$ and $u_1$. They are at distance 1, so by Case 1, $F(C(v_a)) = F(C(u_1))$.
    Now, consider vertices $u_1$ and $v_b$. The distance between them is $k-1$. By the induction hypothesis, we know that $F(C(u_1)) = F(C(v_b))$.

    Combining both equalities, we obtain:
    $$ F(C(v_a)) = F(C(u_1)) = F(C(v_b)) $$
    Therefore, the final state is independent of the initial vertex for any pair of vertices in the graph.
\end{proof}

\begin{theorem}[]
\label{proposition:subgrafo-no-finitud}
If $\Gamma' \subset \Gamma$ is a subgraph that is not of Kostant finite type, then $\Gamma$ is not either.
\end{theorem}

\begin{proof}
It suffices to play within $\Gamma'$: we choose the initial vertex in $\Gamma'$ and repeat reflections only on vertices of $\Gamma'$. These reflections ignore the vertices of $\Gamma \setminus \Gamma'$ and reproduce the pure game of $\Gamma'$, which by hypothesis does not terminate.
\end{proof}

In the case of a Dynkin diagram, the game simulates a process of ``root firing'' and is linked to the enumeration of positive roots.

\begin{definition}[Kostant Reflection Operator]
\label{def:oprefle}
Let $\Gamma=(V, E)$ be a graph and $c$ a configuration of the Kostant game on $\Gamma$. A move in the game consists of choosing any \textbf{sad} vertex $i$ and applying the \textbf{reflection} operator $s_i$. This operator transforms a configuration $c$ into a new configuration $c' = s_i(c)$ as follows:
\begin{itemize}
    \item The component at vertex $i$ is updated to:
    $$ (s_i(c))_i = -c_i + \sum_{j \in N(i)} c_j $$
    \item The other components remain unchanged:
    $$ (s_i(c))_k = c_k, \quad \text{for all } k \neq i. $$
\end{itemize}
\end{definition}

The property of local confluence of the game, i.e., whether the order of operations matters, depends crucially on the adjacency of the vertices involved.

\begin{lemma}[Diamond/Hexagon Lemma]
    \label{lemma_diamantehexagono}
Let $\mathcal{C}(\Gamma)=(V, E)$ be the configuration graph of the Kostant game on some graph $\Gamma$, and let $c$ be a configuration (i.e., some vertex of the configuration graph), with $i, j \in V$ two distinct vertices ($i \neq j$) where reflections can be applied. Let $s_i$ and $s_j$ be the reflection operators.
\begin{enumerate}
    \item \textbf{(Diamond Lemma)} If $i$ and $j$ are not adjacent ($(i,j) \notin E$), the operations commute:
    $$ s_j(s_i(c)) = s_i(s_j(c)) $$
    
    \item \textbf{(Hexagon Lemma)} If $i$ and $j$ are adjacent ($(i,j) \in E$), they satisfy the ``braid'' relation:
    $$ s_i(s_j(s_i(c))) = s_j(s_i(s_j(c))) $$
\end{enumerate}

To visualize these lemmas graphically, we can see the possible cases when advancing through the game configurations as in Figure \ref{fig:diamante} and Figure \ref{fig:hexagono}.

\begin{figure}[H]
\centering

\begin{minipage}{0.48\textwidth}
    \centering
    \begin{tikzpicture}[
        node distance=2.2cm and 2cm,
        config/.style={circle, draw, thick, minimum size=1cm},
        operador/.style={font=\large\itshape}
    ]
    
        \node[config] (c) at (0,0) {$c$};
        \node[config] (ci) at (-1.5,-2) {$s_i(c)$};
        \node[config] (cj) at (1.5,-2) {$s_j(c)$};
        \node[config] (cfinal) at (0,-4) {$c_{\text{final}}$};

        \draw[->, thick] (c) -- node[operador, above left] {$s_i$} (ci);
        \draw[->, thick] (c) -- node[operador, above right] {$s_j$} (cj);
        \draw[->, thick] (ci) -- node[operador, below left] {$s_j$} (cfinal);
        \draw[->, thick] (cj) -- node[operador, below right] {$s_i$} (cfinal);
    \end{tikzpicture}
    \caption{
        \textbf{Diamond Lemma.}
        Reflections commute: the order of operations does not alter the final result.
        $s_j \circ s_i = s_i \circ s_j$.
    }
    \label{fig:diamante}
\end{minipage}
\hfill 
\begin{minipage}{0.48\textwidth}
    \centering
    \begin{tikzpicture}[
        node distance=1.8cm and 2cm,
        config/.style={circle, draw, thick, minimum size=1cm},
        operador/.style={font=\large\itshape}
    ]

        \node[config] (c) at (0,0) {$c$};
        \node[config] (ci) at (-1.8,-1.5) {$s_i(c)$};
        \node[config] (cj) at (1.8,-1.5) {$s_j(c)$};
        \node[config] (cji) at (-1.8,-3.5) {$s_j s_i(c)$};
        \node[config] (cij) at (1.8,-3.5) {$s_i s_j(c)$};
        \node[config] (cfinal) at (0,-5) {$c_{\text{final}}$};

        \draw[->, thick] (c) -- node[operador, above left] {$s_i$} (ci);
        \draw[->, thick] (ci) -- node[operador, left] {$s_j$} (cji);
        \draw[->, thick] (cji) -- node[operador, below left] {$s_i$} (cfinal);

        \draw[->, thick] (c) -- node[operador, above right] {$s_j$} (cj);
        \draw[->, thick] (cj) -- node[operador, right] {$s_i$} (cij);
        \draw[->, thick] (cij) -- node[operador, below right] {$s_j$} (cfinal);
    \end{tikzpicture}
    \caption{
        \textbf{Hexagon Lemma.}
        Reflections satisfy the ``braid'' relation. Paths of 3 steps converge.
        $s_i \circ s_j \circ s_i = s_j \circ s_i \circ s_j$.
    }
    \label{fig:hexagono}
\end{minipage}

\end{figure}

\end{lemma}

\begin{proof}
The validity of these identities is purely algebraic and does not depend on the conditions under which the operators are applied, but on their definition. We will prove each case by analyzing the effect of the reflections component by component.

\textbf{Case 1: $i$ and $j$ are not adjacent}

We must prove that $s_j(s_i(c)) = s_i(s_j(c))$. Let $c' = s_i(c)$. Upon applying $s_j$, we obtain $c'' = s_j(c')$.
\begin{itemize}
    \item \textbf{Component $i$:} Since $i \neq j$, $s_j$ does not alter component $i$, so $(c'')_i = (c')_i = -c_i + \sum_{l \in N(i)} c_l$.
    \item \textbf{Component $j$:} $(c'')_j = -(c')_j + \sum_{l \in N(j)} (c')_l$. Since $j \neq i$, $(c')_j = c_j$. Since $i$ is not a neighbor of $j$, $i \notin N(j)$, so for all $l \in N(j)$, we have $l \neq i$ and $(c')_l = c_l$. Thus, $(c'')_j = -c_j + \sum_{l \in N(j)} c_l$.
    \item \textbf{Component $k \notin \{i, j\}$:} Neither $s_i$ nor $s_j$ alters this component, so $(c'')_k = c_k$.
\end{itemize}
The result of $s_j(s_i(c))$ is a configuration where components $i$ and $j$ have been updated independently. By the symmetry of the argument, the calculation of $s_i(s_j(c))$ produces the same result. Therefore, the operators commute.

\textbf{Case 2: $i$ and $j$ are adjacent}

We must prove the braid relation $s_i s_j s_i(c) = s_j s_i s_j(c)$. To simplify notation, let us define $S_i = \sum_{l \in N(i)} c_l$ and $S_j = \sum_{l \in N(j)} c_l$.

\medskip
\noindent\textbf{Calculation of the sequence $c_1 = s_i(s_j(s_i(c)))$:}

\begin{enumerate}
    \item Let $c' = s_i(c)$. Key components: $c'_i = -c_i + S_i$ and $c'_j = c_j$.
    \item Let $c'' = s_j(c')$. Key components:
    \begin{align*}
        c''_j &= -c'_j + \sum_{l \in N(j)} c'_l = -c_j + \left( c'_i + \sum_{l \in N(j), l \neq i} c'_l \right) \\
              &= -c_j + (-c_i + S_i) + (S_j - c_i) = S_i + S_j - 2c_i - c_j \\
        c''_i &= c'_i = -c_i + S_i
    \end{align*}
    \item Finally, $c_1 = s_i(c'')$:
    \begin{align*}
        (c_1)_i &= -c''_i + \sum_{l \in N(i)} c''_l = -c''_i + \left( c''_j + \sum_{l \in N(i), l \neq j} c''_l \right) \\
                 &= -(-c_i + S_i) + (S_i + S_j - 2c_i - c_j) + (S_i - c_j) \\
                 &= S_i + S_j - c_i - 2c_j \\
        (c_1)_j &= c''_j = S_i + S_j - 2c_i - c_j
    \end{align*}
\end{enumerate}

\medskip
\noindent\textbf{Calculation of the sequence $c_2 = s_j(s_i(s_j(c)))$:}

\begin{enumerate}
    \item Let $d' = s_j(c)$. Components: $d'_j = -c_j + S_j$ and $d'_i = c_i$.
    \item Let $d'' = s_i(d')$. Components:
    \begin{align*}
        d''_i &= - d'_i + \sum_{l \in N(i)} d'_l = -c_i + \left( d'_j + \sum_{l \in N(i), l \neq j} d'_l \right) \\
                  &= -c_i + (-c_j + S_j) + (S_i - c_j) = S_i + S_j - c_i - 2c_j \\
        d''_j &= d'_j = -c_j + S_j
    \end{align*}
    \item Finally, $c_2 = s_j(d'')$:
    \begin{align*}
        (c_2)_j &= -d''_j + \sum_{l \in N(j)} d''_l = -d''_j + \left( d''_i + \sum_{l \in N(j), l \neq i} d''_l \right) \\
                          &= -(-c_j + S_j) + (S_i + S_j - c_i - 2c_j) + (S_j - c_i) \\
                          &= S_i + S_j - 2c_i - c_j \\
        (c_2)_i &= d''_i = S_i + S_j - c_i - 2c_j
    \end{align*}
\end{enumerate}

Comparing the final vectors, we see that $(c_1)_k = (c_2)_k$ for all $k \in V$. Therefore, $s_i s_j s_i(c) = s_j s_i s_j(c)$.
\end{proof}

With these lemmas, we can now prove an important lemma of game dynamics: in the case where we play the Kostant game on simple graphs, the final configuration of the Kostant game, if it exists, is unique.

\begin{lemma}[Rome Lemma]
    Let $C$ be the configuration graph $\mathcal{C}(\Gamma)$ of the Kostant game on some graph $\Gamma$, where $C$ is directed, connected, and acyclic (without loops). Suppose that $C$ is \textbf{graded} by a function $h: V(C) \to \mathbb{Z}$ such that if there is an edge $u \to v$, then $h(u) < h(v)$. Furthermore, suppose that $C$ satisfies the \textbf{local confluence property}: if a vertex $u$ has two distinct successors $v_1$ and $v_2$, then there exists a vertex $z$ that is a descendant of both $v_1$ and $v_2$.
    
    Then, $C$ has either zero or exactly one final vertex (maximal element).
\end{lemma}

\begin{proof}   
    Suppose the lemma is false. This means that $C$ must have at least two distinct final vertices. Let $w_1$ and $w_2$ be two of these maximal elements, with $w_1 \neq w_2$.
    
    Since the graph $C$ is connected, the set of common ancestors of $w_1$ and $w_2$ is non-empty. Let $A = \{ u \in V(C) \mid u \text{ is an ancestor of } w_1 \text{ and } w_2 \}$ be this set.
    
    Since $C$ is finite (as the configuration graph of the Kostant game is), we can choose an element $u \in A$ that has the maximum value according to the grading function $h$. That is, we select $u \in A$ such that $h(u) \ge h(u')$ for all $u' \in A$.
    
    As $u$ is an ancestor of $w_1$ and $w_2$, but is not a maximal element itself (since $w_1$ and $w_2$ are its descendants), $u$ must have at least one successor.
    
    Consider the paths going from $u$ to $w_1$ and from $u$ to $w_2$. Let $v_1$ be the first vertex on a path from $u$ to $w_1$ (i.e., $u \to v_1 \to \dots \to w_1$), and let $v_2$ be the first vertex on a path from $u$ to $w_2$ (i.e., $u \to v_2 \to \dots \to w_2$).
    
    We claim that $v_1 \neq v_2$. If they were equal ($v_1 = v_2 = v$), then $v$ would be a common ancestor of $w_1$ and $w_2$. By the property of the grading function, $h(v) > h(u)$. But this contradicts our choice of $u$ as the common ancestor of maximum height. Therefore, $v_1$ and $v_2$ must be distinct successors of $u$.
    
    We now have a vertex $u$ with two distinct outgoing edges, $u \to v_1$ and $u \to v_2$. By the \textbf{local confluence} hypothesis of the lemma, there must exist a vertex $z$ that is a descendant of both $v_1$ and $v_2$.
    
    If $z$ is a descendant of $v_1$, and $v_1$ is a descendant of $u$, then $z$ is a descendant of $u$. For the same reason, $z$ is also a descendant of $u$ via $v_2$.
    
    This means that $z$ is a common ancestor of $w_1$ and $w_2$ (since any descendant of $z$ is also a descendant of $u$, and paths to $w_1$ and $w_2$ must pass through it or its descendants). Therefore, $z \in A$.
    
    However, since $u \to v_1$ and $v_1$ is an ancestor of $z$, by the grading function property we have $h(u) < h(v_1) \le h(z)$. This implies that we have found an element $z$ in the common ancestor set $A$ whose height is strictly greater than that of $u$.
    
    This directly contradicts our choice of $u$ as the common ancestor of \textbf{maximum height}.
    
    The initial assumption that there were two or more final vertices must be false. Therefore, graph $C$ can have at most one final vertex. This completes the proof.
\end{proof}

This lemma is mentioned in \cite[Theorem 3.2]{chen2017} but a formal proof is not provided. An interpretation of the name of this lemma comes from the famous saying ``all roads lead to Rome''; in our context under the given conditions, $\mathcal{C}(\Gamma)$ has a strongly restrictive structure: either it has no sinks, or all paths lead to a unique sink $q$ in a step-synchronized manner.

\begin{remark}
    In the case of the configuration graph $\mathcal{C}(\Gamma)$ when playing the Kostant game on some graph $\Gamma$, we can define the graph weighting by the height function defined by the quantity of chips in a particular configuration, so that the weight on vertex $c_i$ would be $h(c_i)$.
\end{remark}

Applying the Rome Lemma to the configuration graph of the Kostant game, we conclude that if the game ends (i.e., there exists some final vertex), then there must be a unique stable final state $c_{\mathrm{fin}}$, reachable regardless of the order of plays.

\begin{theorem}
\label{teo:unicofin}
If the Kostant game on a graph $\Gamma$ terminates, then the final configuration is unique. In particular, for each simply-laced Dynkin diagram, there exists a unique stable final configuration (the highest root of the system) to which any sequence of moves starting with an initial chip leads.
\end{theorem}

\begin{proof}
    The proof is based on applying the Rome Lemma to the configuration graph $\mathcal{C}(\Gamma)$ associated with the Kostant game on $\Gamma$. To do this, we must verify that $\mathcal{C}(\Gamma)$ satisfies all the hypotheses of the lemma:
    \begin{enumerate}
        \item \textbf{It is a directed, connected, and acyclic graph:} It is directed by definition. It is connected by Theorem \ref{theorem:inicio-independiente}. It is acyclic because, in each move, the quantity of chips on the game graph must increase; that is, each valid move strictly increases the ``height'' (sum of chips) of the configuration, making it impossible to return to a previous state.
        
        \item \textbf{It is graded:} The height function (the sum of chips of a configuration) acts as the required grading function $h$.
        
        \item \textbf{Satisfies local confluence:} This is the strongest hypothesis, and its fulfillment is a direct consequence of the \textbf{Diamond/Hexagon Lemma} \ref{lemma_diamantehexagono}. These lemmas guarantee that if from a configuration $c$ one can play on two distinct vertices, the resulting paths converge, thus fulfilling the confluence condition.
    \end{enumerate}
    Since the game on a simply-laced Dynkin diagram terminates, we know that the graph $\mathcal{C}(\Gamma)$ has at least one final vertex. Since we have verified that $\mathcal{C}(\Gamma)$ satisfies all hypotheses of the Rome Lemma, it is concluded that it must have \textbf{exactly one} final vertex. This demonstrates that the final configuration is unique.
\end{proof}

\subsection{The Kostant game on multiply-laced diagrams}

We will now see that the Kostant game can be played on multiply-laced Dynkin diagrams. This idea arises naturally when attempting to extend the reflections applied in the original game to multiple edges. We will follow the approach described in [\cite{CaviedesCastro2022}, \S 3.3] as follows.

\begin{definition}[Kostant Game on multiply-laced graphs]
\label{def:kostantgamemulti}
Let $\Gamma$ be a Dynkin diagram, and let $\Delta = \{\alpha_1, \ldots, \alpha_n\}$ be the set of simple roots. The set of vertices of the diagram is denoted by $V = \{1, \ldots, n\}$. For $i \in V$, we denote by $N(i)$ the set of vertices adjacent to $i$.

For each pair $j \in N(i)$, we denote by $n_{i,j}$ the number of arrows pointing from $j$ to $i$ in the diagram. In our convention:
\begin{itemize}
    \item If $\alpha_i$ and $\alpha_j$ have the same length, then $n_{i,j} = 1 = n_{j,i}$.
    \item If $\alpha_i$ is long and $\alpha_j$ is short, then $n_{i,j} = 1$ and $n_{j,i} = 2$, or $n_{i,j} = 1$ and $n_{j,i} = 3$.
\end{itemize}

A \textbf{configuration} is an expression of the form
$$
c = \sum_{i=1}^n c_i \alpha_i, \quad \text{with } c_i \in \mathbb{Z}_{\geq 0}.
$$
For each vertex $i$, we define its state as follows:
\begin{itemize}
    \item Vertex $i$ is \emph{happy} if 
    $$
    c_i = \dfrac{1}{2} \sum_{j \in N(i)} n_{i,j} c_j.
    $$
    \item Vertex $i$ is \emph{sad} if 
    $$
    c_i < \dfrac{1}{2} \sum_{j \in N(i)} n_{i,j} c_j.
    $$
    \item Vertex $i$ is \emph{excited} if 
    $$
    c_i > \dfrac{1}{2} \sum_{j \in N(i)} n_{i,j} c_j.
    $$
\end{itemize}

While there are sad vertices, one of them is chosen, say vertex $i$, and a \emph{reflection} is performed at $i$, updating the value of $c_i$ according to the rule:
$$
c_i \mapsto -c_i + \sum_{j \in N(i)} n_{i,j} c_j,
$$
keeping the other coordinates of $c$ unchanged. The process continues until a configuration is reached where all vertices are happy or excited.
\end{definition}

Recall from Theorem \ref{teo:unicofin} that if the Dynkin diagram is simply-laced, the game always leads to a unique final configuration, regardless of the vertex where we place the first chip.

In this new case, where we can play on Dynkin diagrams that are not simply-laced, the sequences of moves in the Kostant game can lead to more than one final configuration. The following example verifies this property.

\begin{example}[The game on a multiply-laced diagram]
    The game on $F_4$ illustrates the dynamics of the Kostant game in the case of multiply-laced graphs. Furthermore, it is evident that this game can terminate in more than one possible final configuration, as seen in Figure \ref{fig:kostant_f4_final}.

\begin{figure}[H]
    \centering
    \begin{tikzpicture}[scale=0.8, transform shape]
        
        \tikzset{
            double-arrow/.style={double, double distance=1.5pt, -implies, shorten >=2pt, shorten <=2pt},
            vtx/.style={circle, draw, fill=white, inner sep=1pt, minimum size=16pt, font=\small},
            connect/.style={->, >=stealth, thick, shorten >=2pt, shorten <=2pt}
        }
        
        \begin{scope}[shift={(0,0)}]
            \node[vtx] (L1_1) at (0,0) {1};
            \node[vtx] (L1_2) at (1.2,0) {};
            \node[vtx] (L1_3) at (2.4,0) {};
            \node[vtx] (L1_4) at (3.6,0) {};
            \draw (L1_1) -- (L1_2);
            \draw[double-arrow] (L1_2) -- (L1_3);
            \draw (L1_3) -- (L1_4);
        \end{scope}
        
        \begin{scope}[shift={(4.2, -1.15)}]
            \node[vtx] (L2_1) at (0,0) {1};
            \node[vtx] (L2_2) at (1.2,0) {1};
            \node[vtx] (L2_3) at (2.4,0) {};
            \node[vtx] (L2_4) at (3.6,0) {};
            \draw (L2_1) -- (L2_2);
            \draw[double-arrow] (L2_2) -- (L2_3);
            \draw (L2_3) -- (L2_4);
        \end{scope}
        \draw[connect] (L1_4) -- (L2_1);
        
        \begin{scope}[shift={(0, -2.3)}]
            \node[vtx] (L3_1) at (0,0) {1};
            \node[vtx] (L3_2) at (1.2,0) {1};
            \node[vtx] (L3_3) at (2.4,0) {2};
            \node[vtx] (L3_4) at (3.6,0) {};
            \draw (L3_1) -- (L3_2);
            \draw[double-arrow] (L3_2) -- (L3_3);
            \draw (L3_3) -- (L3_4);
        \end{scope}
        \draw[connect] (L2_1) -- (L3_4);
        
        \begin{scope}[shift={(4.2, -3.45)}]
            \node[vtx] (L4_1) at (0,0) {1};
            \node[vtx] (L4_2) at (1.2,0) {1};
            \node[vtx] (L4_3) at (2.4,0) {2};
            \node[vtx] (L4_4) at (3.6,0) {2};
            \draw (L4_1) -- (L4_2);
            \draw[double-arrow] (L4_2) -- (L4_3);
            \draw (L4_3) -- (L4_4);
        \end{scope}
        \draw[connect] (L3_4) -- (L4_1);
        
        \begin{scope}[shift={(0, -4.6)}]
            \node[vtx] (L5_1) at (0,0) {1};
            \node[vtx] (L5_2) at (1.2,0) {2};
            \node[vtx] (L5_3) at (2.4,0) {2};
            \node[vtx] (L5_4) at (3.6,0) {2};
            \draw (L5_1) -- (L5_2);
            \draw[double-arrow] (L5_2) -- (L5_3);
            \draw (L5_3) -- (L5_4);
        \end{scope}
        \draw[connect] (L4_1) -- (L5_4);
        
        \begin{scope}[shift={(4.2, -5.75)}]
            \node[vtx] (L6_1) at (0,0) {1};
            \node[vtx] (L6_2) at (1.2,0) {2};
            \node[vtx] (L6_3) at (2.4,0) {4};
            \node[vtx] (L6_4) at (3.6,0) {2};
            \draw (L6_1) -- (L6_2);
            \draw[double-arrow] (L6_2) -- (L6_3);
            \draw (L6_3) -- (L6_4);
        \end{scope}
        \draw[connect] (L5_4) -- (L6_1);
        
        \begin{scope}[shift={(0, -6.9)}]
            \node[vtx] (L7_1) at (0,0) {1};
            \node[vtx] (L7_2) at (1.2,0) {3};
            \node[vtx] (L7_3) at (2.4,0) {4};
            \node[vtx] (L7_4) at (3.6,0) {2};
            \draw (L7_1) -- (L7_2);
            \draw[double-arrow] (L7_2) -- (L7_3);
            \draw (L7_3) -- (L7_4);
        \end{scope}
        \draw[connect] (L6_1) -- (L7_4);
        
        \begin{scope}[shift={(4.2, -8.05)}]
            \node[vtx] (L8_1) at (0,0) {2};
            \node[vtx] (L8_2) at (1.2,0) {3};
            \node[vtx] (L8_3) at (2.4,0) {4};
            \node[vtx] (L8_4) at (3.6,0) {2};
            \draw (L8_1) -- (L8_2);
            \draw[double-arrow] (L8_2) -- (L8_3);
            \draw (L8_3) -- (L8_4);
        \end{scope}
        \draw[connect] (L7_4) -- (L8_1);
        
        \begin{scope}[shift={(9.2, 0)}]
            \node[vtx] (R1_1) at (0,0) {};
            \node[vtx] (R1_2) at (1.2,0) {};
            \node[vtx] (R1_3) at (2.4,0) {};
            \node[vtx] (R1_4) at (3.6,0) {1};
            \draw (R1_1) -- (R1_2);
            \draw[double-arrow] (R1_2) -- (R1_3);
            \draw (R1_3) -- (R1_4);
        \end{scope}
        
        \begin{scope}[shift={(13.4, -1.15)}]
            \node[vtx] (R2_1) at (0,0) {};
            \node[vtx] (R2_2) at (1.2,0) {};
            \node[vtx] (R2_3) at (2.4,0) {1};
            \node[vtx] (R2_4) at (3.6,0) {1};
            \draw (R2_1) -- (R2_2);
            \draw[double-arrow] (R2_2) -- (R2_3);
            \draw (R2_3) -- (R2_4);
        \end{scope}
        \draw[connect] (R1_4) -- (R2_1);
        
        \begin{scope}[shift={(9.2, -2.3)}]
            \node[vtx] (R3_1) at (0,0) {};
            \node[vtx] (R3_2) at (1.2,0) {1};
            \node[vtx] (R3_3) at (2.4,0) {1};
            \node[vtx] (R3_4) at (3.6,0) {1};
            \draw (R3_1) -- (R3_2);
            \draw[double-arrow] (R3_2) -- (R3_3);
            \draw (R3_3) -- (R3_4);
        \end{scope}
        \draw[connect] (R2_1) -- (R3_4);
        
        \begin{scope}[shift={(13.4, -3.45)}]
            \node[vtx] (R4_1) at (0,0) {};
            \node[vtx] (R4_2) at (1.2,0) {1};
            \node[vtx] (R4_3) at (2.4,0) {2};
            \node[vtx] (R4_4) at (3.6,0) {1};
            \draw (R4_1) -- (R4_2);
            \draw[double-arrow] (R4_2) -- (R4_3);
            \draw (R4_3) -- (R4_4);
        \end{scope}
        \draw[connect] (R3_4) -- (R4_1);
        
        \begin{scope}[shift={(9.2, -4.6)}]
            \node[vtx] (R5_1) at (0,0) {1};
            \node[vtx] (R5_2) at (1.2,0) {1};
            \node[vtx] (R5_3) at (2.4,0) {2};
            \node[vtx] (R5_4) at (3.6,0) {1};
            \draw (R5_1) -- (R5_2);
            \draw[double-arrow] (R5_2) -- (R5_3);
            \draw (R5_3) -- (R5_4);
        \end{scope}
        \draw[connect] (R4_1) -- (R5_4);
        
        \begin{scope}[shift={(13.4, -5.75)}]
            \node[vtx] (R6_1) at (0,0) {1};
            \node[vtx] (R6_2) at (1.2,0) {2};
            \node[vtx] (R6_3) at (2.4,0) {2};
            \node[vtx] (R6_4) at (3.6,0) {1};
            \draw (R6_1) -- (R6_2);
            \draw[double-arrow] (R6_2) -- (R6_3);
            \draw (R6_3) -- (R6_4);
        \end{scope}
        \draw[connect] (R5_4) -- (R6_1);
        
        \begin{scope}[shift={(9.2, -6.9)}]
            \node[vtx] (R7_1) at (0,0) {1};
            \node[vtx] (R7_2) at (1.2,0) {2};
            \node[vtx] (R7_3) at (2.4,0) {3};
            \node[vtx] (R7_4) at (3.6,0) {1};
            \draw (R7_1) -- (R7_2);
            \draw[double-arrow] (R7_2) -- (R7_3);
            \draw (R7_3) -- (R7_4);
        \end{scope}
        \draw[connect] (R6_1) -- (R7_4);
        
        \begin{scope}[shift={(13.4, -8.05)}]
            \node[vtx] (R8_1) at (0,0) {1};
            \node[vtx] (R8_2) at (1.2,0) {2};
            \node[vtx] (R8_3) at (2.4,0) {3};
            \node[vtx] (R8_4) at (3.6,0) {2};
            \draw (R8_1) -- (R8_2);
            \draw[double-arrow] (R8_2) -- (R8_3);
            \draw (R8_3) -- (R8_4);
        \end{scope}
        \draw[connect] (R7_4) -- (R8_1);
        
    \end{tikzpicture}
    \caption{Two paths in the Kostant Game on the $F_4$ diagram ending in different configurations.}
    \label{fig:kostant_f4_final}
\end{figure}
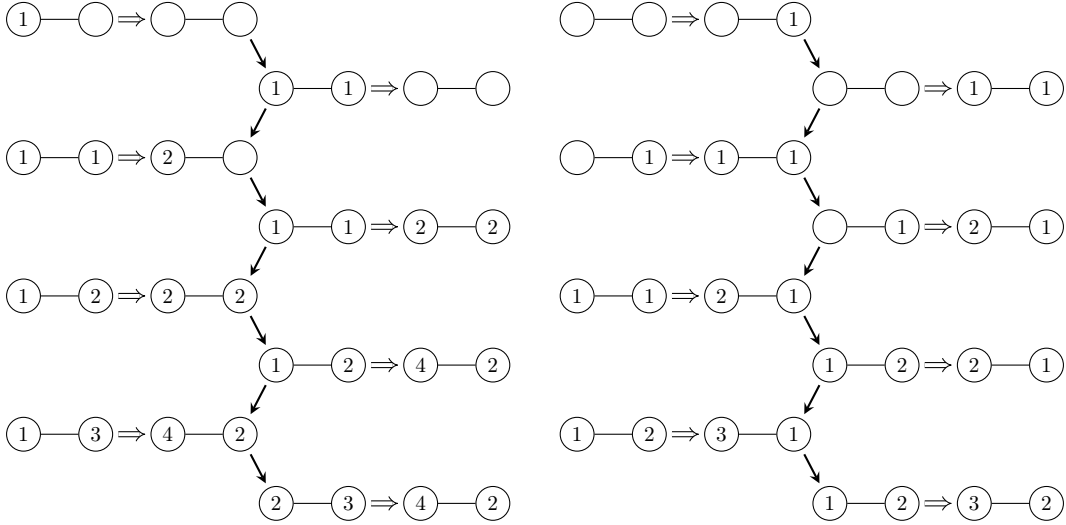

\end{example}

\begin{theorem}
    Let $\Gamma$ be a multiply-laced Dynkin diagram, i.e., $B_n, C_n, F_4$, or $G_2$. When playing the multiply-laced Kostant game on $\Gamma$, only two different final configurations can be obtained.
\end{theorem}

\begin{proof}
	To prove this theorem, it suffices to verify all possible cases extensively (for example, using the Java application described in Section \ref{java}), as was done in the previous example on $F_4$. It can be noted that the final configuration depends on whether the game starts before the short root or after the long root, giving rise to only two possible final configurations. The bifurcation into two unique terminal paths arises from the asymmetry introduced by the multiple edge. The choice of starting the game on one side or the other of this asymmetric ``barrier'' fundamentally determines the evolution of the configurations, restricting the state space to one of two possible sinks.
\end{proof}

\subsection{Interpretation of the game in root systems}

We now link the combinatorial dynamics of the Kostant game with the fundamental algebraic structure of root systems: the action of the Weyl group. We will demonstrate that the ``firing'' rule of the game is a direct analogue of a Weyl reflection, and that the ``sadness'' condition guarantees a key growth property.

Let $\Phi$ be an irreducible root system with Dynkin diagram $\Gamma$ and simple root basis $\Delta=\{\alpha_1,\dots,\alpha_n\}$. To each game configuration $\mathbf{c}=(c_1,\dots,c_n)$, we associate the vector $\beta \in E$ as the linear combination:
$$
  \beta = \sum_{i=1}^n c_i\,\alpha_i.
$$
We denote the \emph{height} of $\beta$ by the sum of its coefficients: $\mathrm{ht}(\beta)=\sum_i c_i$.

Our goal is to show that the update rule (Definition \ref{def:kostantgamemulti}) corresponds to the action of a Weyl reflection on the vector $\beta$. The general formula for a Weyl reflection $s_k$ (associated with the simple root $\alpha_k$) on a vector $\beta$ is:
$$
s_k(\beta) = \beta - \langle\beta, \alpha_k^\vee\rangle\,\alpha_k
$$
where $\alpha_k^\vee = \dfrac{2\alpha_k}{(\alpha_k, \alpha_k)}$ is the corresponding simple co-root.

Let us compute the inner product $\langle\beta, \alpha_k^\vee\rangle$:
$$
\langle\beta, \alpha_k^\vee\rangle = \left\langle \sum_{i=1}^n c_i\alpha_i, \alpha_k^\vee \right\rangle = \sum_{i=1}^n c_i \langle\alpha_i, \alpha_k^\vee\rangle = \sum_{i=1}^n c_i A_{ik},
$$
where $A_{ik}$ are the coefficients of the Cartan matrix. Substituting this into the reflection formula:
\begin{align*}
s_k(\beta) &= \sum_{i=1}^n c_i\alpha_i - \left( \sum_{i=1}^n c_i A_{ik} \right) \alpha_k \\
&= \sum_{i\neq k} c_i\alpha_i + c_k\alpha_k - \left( c_k A_{kk} + \sum_{j \in N(k)} c_j A_{jk} \right)\alpha_k.
\end{align*}
Using the properties of the Cartan matrix, $A_{kk}=2$, and the relationship with the Dynkin diagram coefficients, $n_{k,j} = -A_{jk}$ for $j \neq k$, we obtain:
\begin{align*}
s_k(\beta) &= \sum_{i\neq k} c_i\alpha_i + c_k\alpha_k - \left( 2c_k - \sum_{j \in N(k)} c_j n_{k,j} \right)\alpha_k \\
&= \sum_{i\neq k} c_i\alpha_i + \left( -c_k + \sum_{j \in N(k)} c_j n_{k,j} \right)\alpha_k.
\end{align*}
If we denote the new vector as $\beta' = s_k(\beta) = \sum_{i=1}^n c'_i \alpha_i$, its coefficients are:
\begin{itemize}
    \item $c'_i = c_i$ for $i \neq k$.
    \item $c'_k = -c_k + \sum_{j \in N(k)} n_{k,j} c_j$.
\end{itemize}
This demonstrates that the update rule of the generalized Kostant game is \textbf{exactly} the coordinate expression of the action of a Weyl reflection.

Now, let us analyze the sadness condition. A vertex $k$ is sad if $c_k < \dfrac{1}{2} \sum_{j \in N(k)} n_{k,j} c_j$, which is equivalent to $2c_k < \sum_{j \in N(k)} n_{k,j} c_j$.
Let us see how the height of vector $\beta$ changes after a reflection:
\begin{align*}
\mathrm{ht}(s_k(\beta)) - \mathrm{ht}(\beta) &= \left( \sum_{i \neq k} c_i + c'_k \right) - \left( \sum_{i \neq k} c_i + c_k \right) \\
&= c'_k - c_k \\
&= \left( -c_k + \sum_{j \in N(k)} n_{k,j} c_j \right) - c_k \\
&= -2c_k + \sum_{j \in N(k)} n_{k,j} c_j.
\end{align*}
Therefore, the height of the vector increases, $\mathrm{ht}(s_k(\beta)) > \mathrm{ht}(\beta)$, if and only if $-2c_k + \sum_{j \in N(k)} n_{k,j} c_j > 0$, which is precisely the sadness condition $2c_k < \sum_{j \in N(k)} n_{k,j} c_j$.

In conclusion, the Kostant game is a combinatorial realization of a fundamental process in root system theory. Each move corresponds to a Weyl reflection, and the rule of playing only on sad vertices ensures that the height of the vector associated with the configuration strictly increases at each step. Since the number of positive roots in a finite type system is finite, the game cannot continue indefinitely and must terminate. This justifies why the game will terminate on all finite type Dynkin diagrams (A, B, C, D, E, F, G).

\subsection{Classification of Dynkin diagrams via the game}
\label{sec2.2}

From the previous discussion, we know that the Kostant game must terminate on the Dynkin diagrams classified in Section \ref{sec1.3}. However, we would like to know if these are the only graphs where it terminates. This would allow for a characterization of the diagrams without having to resort to root system theory. Paradoxically, a very good way to do this is to find some graphs that are comparatively simple for which the game does not terminate, and then use these and Proposition \ref{proposition:subgrafo-no-finitud} to eliminate possible graphs of Kostant finite type. This classification will culminate in this section, providing an alternative proof of Theorem \ref{teo:clasifidynkin}.

\begin{definition}[Affine Dynkin Diagrams]
Let $ G $ be a Dynkin diagram, and let us play the Kostant game until reaching its final configuration. The \textbf{associated affine Dynkin diagram} is defined as the graph obtained by adding a new vertex $ v_0 $ and connecting it to the vertices that remained excited in $ G $.

We denote the resulting diagrams as $ \widetilde{A}_n $, $ \widetilde{B}_n $, $ \widetilde{C}_n $, $ \widetilde{D}_n $, $ \widetilde{E}_6 $, $ \widetilde{E}_7 $, $ \widetilde{E}_8 $, $ \widetilde{F}_4 $, and $ \widetilde{G}_2 $. These new diagrams are known as \emph{affine Dynkin diagrams} (see Figure \ref{fig:dynkinafin}).
\end{definition}

\begin{example}[Construction of the affine diagram $\widetilde A_{n-1}$]
Let us illustrate the procedure described in Theorem 5.9 for the case of the diagram of type $A_{n-1}$, which is a linear graph with $n-1$ vertices.
\begin{enumerate}
    \item \textbf{Play the game on the Kostant finite graph $A_{n-1}$:}
    
    The graph $A_{n-1}$ is a line of $n-1$ vertices, labeled $v_1, v_2, \dots, v_{n-1}$. 
    
    We start the game by assigning the value 1 to a vertex (e.g., $v_1$) and 0 to the others. Vertex $v_2$ becomes sad, and upon ``firing'' it, its value becomes 1. This effect propagates along the chain.
    
    The final configuration $F$ of the game, where no unhappy vertices remain, is the one in which all vertices have the value 1.
    
    $$
    F(v_i) = 1 \quad \text{for all } i=1, \dots, n-1. 
    $$

    \item \textbf{Identify excited vertices:}
    
    Now, we must determine which vertices are excited in this final configuration $F$. A vertex $v$ is excited if its value is greater than half the sum of the values of its neighbors.
    \begin{itemize}
        \item For an \textbf{interior} vertex $v_i$ (with $1 < i < n-1$), its neighbors are $v_{i-1}$ and $v_{i+1}$. The condition is:
        
        $$
        F(v_i) > \dfrac{1}{2} \left( F(v_{i-1}) + F(v_{i+1}) \right)
        \quad \implies \quad
        1 > \dfrac{1}{2} (1+1) = 1.      
        $$
        This is false. In fact, since the equality holds, the interior vertices are \textbf{happy}.
        
        \item For the \textbf{end} vertices, $v_1$ and $v_{n-1}$, each has only one neighbor.
        For $v_1$, its only neighbor is $v_2$:
        
        $$ F(v_1) > \dfrac{1}{2} F(v_2)
        \quad \implies \quad
        1 > \dfrac{1}{2} (1) = 0.5.
        $$
        This condition is true. Therefore, $v_1$ and $v_{n-1}$ are the \textbf{excited} vertices.
    \end{itemize}

    \item \textbf{Construct the affine diagram $\widetilde A_{n-1}$:}
    
    We introduce a new vertex, which we call $v_0$. We connect it to all excited vertices of the configuration $F$. In this case, we connect $v_0$ to $v_1$ and $v_{n-1}$.
    
    The resulting graph is a \textbf{simple cycle} of $n$ vertices: $v_0 \to v_1 \to v_2 \to \dots \to v_{n-1} \to v_0$. This is, by definition, the affine Dynkin diagram of type $\widetilde A_{n-1}$.
\end{enumerate}
\end{example}

\begin{figure}[H]
    \centering
    \begin{tikzpicture}[scale=0.8, transform shape]
        
        \tikzset{
            w/.style={circle, draw, fill=white, inner sep=0pt, minimum size=6pt},
            b/.style={circle, draw, fill=black, inner sep=0pt, minimum size=6pt},
            dashededge/.style={dashed, thin},
            doublearr/.style={double, double distance=1.5pt, -implies, shorten >=1pt, shorten <=1pt}
        }

        \begin{scope}[shift={(0,0)}]
            \node[w] (a1) at (0,0) {};
            \node[w] (a2) at (1,0) {};
            \node[w] (a3) at (2,0) {};
            \node[w] (an) at (4,0) {};
            \node[b] (root) at (2, 1) {};
            
            \draw (a1) -- (a2) -- (a3);
            \draw[dashededge] (a3) -- (an);
            \draw (a1) -- (root) -- (an);
            \node at (5, 0.5) {$\tilde{A}_n$};
        \end{scope}
        
        \begin{scope}[shift={(0,-2.5)}]
            \node[b] (b0) at (-0.5, 0.5) {}; 
            \node[w] (b1) at (-0.5, -0.5) {};
            \node[w] (b2) at (0.5, 0) {};
            \node[w] (b3) at (1.5, 0) {};
            \node[w] (bn) at (3.5, 0) {};
            \node[w] (bn_end) at (4.5, 0) {};
            
            \draw (b0) -- (b2);
            \draw (b1) -- (b2) -- (b3);
            \draw[dashededge] (b3) -- (bn);
            \draw[doublearr] (bn) -- (bn_end); 
            \node at (5.5, 0) {$\tilde{B}_n$};
        \end{scope}
        
        \begin{scope}[shift={(0,-5)}]
            \node[b] (c0) at (0,0) {}; 
            \node[w] (c1) at (1,0) {};
            \node[w] (c2) at (2,0) {};
            \node[w] (cn) at (4,0) {};
            \node[w] (cn_end) at (5,0) {};
            
            \draw[doublearr] (c0) -- (c1); 
            
            \draw (c1) -- (c2);
            \draw[dashededge] (c2) -- (cn);
            
            \draw[doublearr] (cn_end) -- (cn); 
            
            \node at (6, 0) {$\tilde{C}_n$};
        \end{scope}
        
        \begin{scope}[shift={(0,-7.5)}]
            \node[b] (d0) at (-0.5, 0.5) {}; 
            \node[w] (d1) at (-0.5, -0.5) {};
            \node[w] (d2) at (0.5, 0) {};
            \node[w] (d3) at (1.5, 0) {};
            \node[w] (dn) at (3.5, 0) {};
            \node[w] (d_end1) at (4.5, 0.5) {};
            \node[w] (d_end2) at (4.5, -0.5) {};
            
            \draw (d0) -- (d2);
            \draw (d1) -- (d2) -- (d3);
            \draw[dashededge] (d3) -- (dn);
            \draw (dn) -- (d_end1);
            \draw (dn) -- (d_end2);
            \node at (5.5, 0) {$\tilde{D}_n$};
        \end{scope}
        
        \def\xsep{8.5}
        
        \begin{scope}[shift={(\xsep, 0)}]
            \node[b] (f0) at (0,0) {};
            \node[w] (f1) at (1,0) {};
            \node[w] (f2) at (2,0) {};
            \node[w] (f3) at (3,0) {};
            \node[w] (f4) at (4,0) {};
            
            \draw (f0) -- (f1) -- (f2);
            \draw[doublearr] (f2) -- (f3);
            \draw (f3) -- (f4);
            \node at (2, -0.5) {$\tilde{F}_4$};
        \end{scope}
        
        \begin{scope}[shift={(\xsep+5, 0)}]
            \node[b] (g0) at (0,0) {};
            \node[w] (g1) at (1,0) {};
            \node[w] (g2) at (2,0) {};
            
            \draw (g0) -- (g1);
            
            \draw[doublearr] (g1) -- (g2); 
            \draw (g1) -- (g2);
            
            \node at (1, -0.5) {$\tilde{G}_2$};
        \end{scope}
        
        \begin{scope}[shift={(\xsep+8, 0)}]
            \node[b] (a1_0) at (0,0) {};
            \node[w] (a1_1) at (1,0) {};
            
            \draw[double, double distance=1.5pt, <->, >=implies, shorten >=1pt, shorten <=1pt] (a1_0) -- (a1_1);
            
            \node at (0.5, -0.5) {$\tilde{A}_1$};
        \end{scope}
        
        \begin{scope}[shift={(\xsep, -3.5)}]
            \node[w] (e1) at (0,0) {};
            \node[w] (e2) at (1,0) {};
            \node[w] (e3) at (2,0) {};
            \node[w] (e4) at (3,0) {};
            \node[w] (e5) at (4,0) {};
            \node[w] (e_up) at (2,1) {};
            \node[b] (e_top) at (2,2) {};
            
            \draw (e1) -- (e2) -- (e3) -- (e4) -- (e5);
            \draw (e3) -- (e_up) -- (e_top);
            \node at (2, -0.5) {$\tilde{E}_6$};
        \end{scope}
        
        \begin{scope}[shift={(\xsep, -6)}]
            \node[b] (e0) at (0,0) {}; 
            \node[w] (e1) at (1,0) {};
            \node[w] (e2) at (2,0) {};
            \node[w] (e3) at (3,0) {};
            \node[w] (e4) at (4,0) {};
            \node[w] (e5) at (5,0) {};
            \node[w] (e6) at (6,0) {};
            \node[w] (e_up) at (3,1) {}; 
            
            \draw (e0) -- (e1) -- (e2) -- (e3) -- (e4) -- (e5) -- (e6);
            \draw (e3) -- (e_up);
            \node at (3, -0.5) {$\tilde{E}_7$};
        \end{scope}
        
        \begin{scope}[shift={(\xsep, -8)}]
            \node[b] (e0) at (0,0) {}; 
            \node[w] (e1) at (1,0) {};
            \node[w] (e2) at (2,0) {};
            \node[w] (e3) at (3,0) {};
            \node[w] (e4) at (4,0) {};
            \node[w] (e5) at (5,0) {};
            \node[w] (e6) at (6,0) {};
            \node[w] (e7) at (7,0) {};
            \node[w] (e_up) at (5,1) {}; 
            
            \draw (e0) -- (e1) -- (e2) -- (e3) -- (e4) -- (e5) -- (e6) -- (e7);
            \draw (e5) -- (e_up);
            \node at (5, -0.5) {$\tilde{E}_8$};
        \end{scope}
        
    \end{tikzpicture}
    \caption{Affine Dynkin diagrams.}
    \label{fig:dynkinafin}
\end{figure}
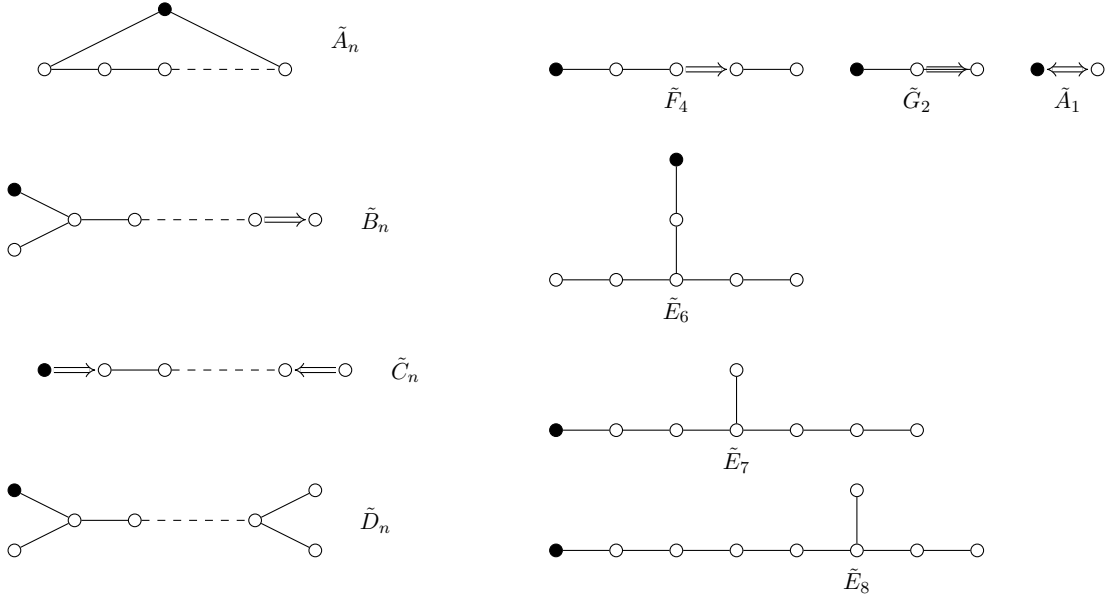

\begin{theorem}[{\cite[Theorem 5.9]{elek2016reflection}}]
The affine Dynkin diagrams $\tilde{A}_n$ ($n \ge 1$), $\tilde{D}_n$ ($n \ge 4$), $\tilde{E}_6$, $\tilde{E}_7$, and $\tilde{E}_8$ are of Kostant infinite type.
\end{theorem}

\begin{proof}
The proof follows a constructive strategy, revealing a mechanism by which the game can continue indefinitely. Let $\tilde{\Gamma}$ be an affine diagram, obtained from a Dynkin diagram $\Gamma$ (of type A, D, or E) by adding a vertex $v_0$.

We start the game on $\tilde{\Gamma}$ by restricting moves solely to the subgraph $\Gamma$. Since $\Gamma$ is of Kostant finite type, this phase of the game necessarily terminates, leading to a stable state on $\Gamma$ where all its vertices are happy. Let us denote the configuration of values on the vertices of $\Gamma$ in this state as $\mathbf{R}$. A fundamental property of this configuration $\mathbf{R}$ (which corresponds to the coefficients of the highest root of the root system) is that all its components are positive integers.

At this point, we consider the state of the affine vertex $v_0$, whose value has remained 0. The neighbors of $v_0$ are all in $\Gamma$ and their values are the positive components of $\mathbf{R}$. The sum of the values of its neighbors, $S(v_0)$, is therefore strictly positive. Since the value of $v_0$ is 0, it holds that $\mathbf{v}(v_0) < S(v_0)$, which means that $v_0$ is sad.

Upon performing the move on $v_0$, its value updates to $S(v_0) - 0 = S(v_0)$, a positive integer. This change alters the state of the neighbors of $v_0$, which cease to be happy and allow the game to continue. Crucially, we have established a cycle: the system on $\Gamma$ stabilizes, which in turn ``activates'' the affine vertex $v_0$, whose move destabilizes $\Gamma$ and allows the process to restart.

In each iteration of this cycle, the total sum of values in the graph increases. A game where vertex values can grow without bound cannot be of Kostant finite type. Therefore, all affine diagrams are of Kostant infinite type.
\end{proof}

\begin{theorem}[{\cite[Theorem 5.10]{elek2016reflection}}]
The only connected simply-laced graphs of Kostant finite type are the Dynkin diagrams of type $A_n$ ($n \ge 1$), $D_n$ ($n \ge 4$), $E_6$, $E_7$, and $E_8$.
\end{theorem}

\begin{proof}
The proof proceeds by systematically eliminating all graph structures that do not belong to the A, D, E classification. The fundamental premise is that if a graph is of Kostant finite type, it cannot contain a subgraph of Kostant infinite type. By Theorem \ref{proposition:subgrafo-no-finitud}, this means that a graph of Kostant finite type cannot contain any affine Dynkin diagram as a subgraph.

Let $\Gamma$ be a simple connected graph. First, $\Gamma$ cannot contain cycles. If it had a cycle of length $m$, this cycle would be isomorphic to the affine diagram $\tilde{A}_{m-1}$, which would imply that $\Gamma$ is of Kostant infinite type. Thus, $\Gamma$ must be a tree.

Next, we examine the degree of the vertices of $\Gamma$. If $\Gamma$ had a vertex of degree 4 or higher, this vertex together with four of its neighbors would form a subgraph isomorphic to $\tilde{D}_4$. The impossibility of containing $\tilde{D}_4$ forces the maximum degree of any vertex in $\Gamma$ to be 3.

Now consider branching points (vertices of degree 3). If $\Gamma$ had two or more branching points, the path connecting them together with the other incident edges would form a subgraph isomorphic to a diagram $\tilde{D}_n$ for some $n \ge 5$. Consequently, $\Gamma$ must have at most one branching point.

This series of restrictions leaves us with only two possibilities for the topology of $\Gamma$: either it has no branching points, or it has exactly one.
If $\Gamma$ has no branching points, it is a path graph, which corresponds to the series $A_n$.

If $\Gamma$ has a unique branching point, it is a star with three arms. Let $p, q, r$ be the lengths of these arms, ordered as $p \le q \le r$. For $\Gamma$ to be of Kostant finite type, it cannot contain $\tilde{E}_6, \tilde{E}_7$, or $\tilde{E}_8$.
The non-containment of $\tilde{E}_6$ (arms of length 2,2,2) implies that the shortest arm must have length $p=1$.
The non-containment of $\tilde{E}_7$ (arms 1,3,3) implies that, given $p=1$, the second shortest arm cannot exceed length 2, i.e., $q \le 2$.
Finally, the non-containment of $\tilde{E}_8$ (arms 1,2,5) implies that if $p=1$ and $q=2$, the length of the third arm cannot exceed 4, i.e., $r \le 4$.

Synthesizing these restrictions, the only permitted three-arm star graphs are:
\begin{itemize}
    \item $p=1, q=1$: arms $(1,1,r)$, which define the series $D_{r+3}$.
    \item $p=1, q=2$: arms $(1,2,r)$ with $r \in \{2,3,4\}$. These correspond to $E_6$ (1,2,2), $E_7$ (1,2,3), and $E_8$ (1,2,4).
\end{itemize}
Having exhausted all other possibilities, we conclude that the only graphs of Kostant finite type are, indeed, those of type A, D, and E.
\end{proof}

\begin{remark}[Extension to all Dynkin diagrams]
One can modify the rules of the game so that it can be played on connected multiply-laced Dynkin diagrams as done in \ref{def:kostantgamemulti}. The complete classification on multiply-laced diagrams (including $B_n$, $C_n$, $F_4$, $G_2$) can be obtained analogously, identifying the corresponding multiply-laced affine diagrams as forbidden subgraphs where the game (on multiply-laced edges) does not terminate.
\end{remark}

\begin{corollary}
The termination of the Kostant game provides a combinatorial criterion that completely characterizes the finite type Dynkin diagrams. 
\end{corollary}

This perspective aligns with the classification discussed in Section \ref{sec1.3} and enriches the understanding of the underlying structure of semisimple Lie algebras and compact Lie groups as worked on in \cite{Humphreys1972}, \cite{Kirillov2008}, and partially in \cite{Kirillov2005CompactGroups}.

\section{Extension of the game and Weyl groups}
\label{Chapter3}

In the previous section, it was established that the Kostant game, in its classical form, is a powerful combinatorial tool that operates on a root system, allowing us to understand the geometry of its simple roots, construct the highest root, and finally fully characterize the Dynkin diagrams.

A natural and deeper question emerges: can the dynamics of the game reveal not only the total structure but also the internal substructures of the root system? Specifically, how does the game relate to the associated Weyl group, not in its entirety, but in its fundamental components, such as quotients by parabolic subgroups?

This section answers this question affirmatively. To do so, a modified version of the Kostant game is introduced and analyzed. We will demonstrate that this extension is not merely a curiosity, but a precise algorithmic tool whose dynamics encode the structure of Weyl group quotients. We will establish that this modified version allows us, in a constructive manner, to represent and analyze the elements of these quotients, thus opening a new perspective on the interaction between the combinatorics of the game and the algebra of Weyl groups.


\subsection{The modified Kostant game}
\label{sec3.1}

Recalling what was mentioned in the introduction, studying the Mukai conjecture in the article \cite{CaviedesCastro2022} led to a tool given by a modified version of the game on Dynkin diagrams.

This modification consists of augmenting a Dynkin graph $\Gamma$ by adding a \textbf{source vertex}, denoted by $\bar{j}$, which is adjacent only to a vertex $j$ of the original graph. A directed edge $\bar{j} \to j$ is added, and the resulting graph is denoted by $\Gamma_j$.

\begin{definition}
The \textbf{modified Kostant game} at vertex $j$ is defined on this new graph $\Gamma_j$ following the rules of Definition \ref{def:kostantgamemulti}, with two special conditions:
\begin{enumerate}
    \item The initial configuration consists of a single chip on the source vertex $\bar{j}$ (and zero on all others).
    \item The source vertex $\bar{j}$ is considered always happy, so it is never a candidate for a reflection.
\end{enumerate}
Configurations and their heights are defined as in the standard game, and the goal remains to reach a state where all vertices of the original graph $\Gamma$ are happy or excited.
\end{definition}

\begin{remark}
    Note that by adding this new always-happy vertex, the start of the modified Kostant game will be determined by the modified vertex, as it will be the only one initially sad. This gives us an additional facility not present in the original Kostant game, where Theorem \ref{theorem:inicio-independiente} had to be used to make sense of the definition of Kostant finite diagrams.
\end{remark}

\begin{example}
    Below we can see two examples of the game on $A_4$ modifying the first vertex and the second vertex of the graph respectively. Here it can be noted that, again, being simply-laced diagrams, the set of configurations reaches a unique possible final configuration, just as in the original Kostant game.
    
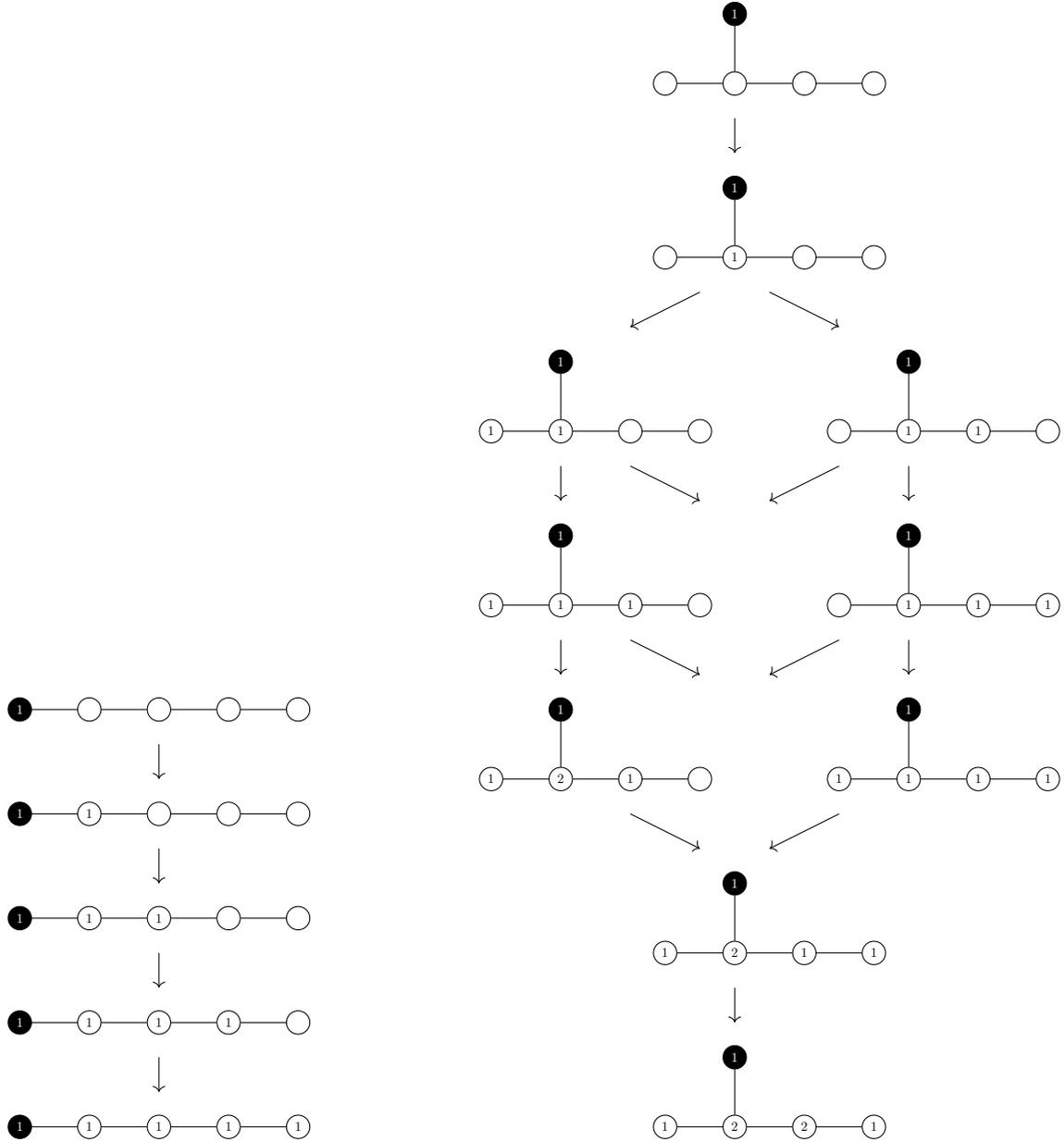
\begin{figure}[H]
\centering

\begin{tikzpicture}[scale=1]
  \node[draw,fill=black,shape=circle,scale=0.5] (a) at (0,0) {$\color{white}{1}$};
  \node[draw,fill=white,shape=circle,scale=0.5] (b) at (1,0) {\phantom{$0$}};
  \node[draw,fill=white,shape=circle,scale=0.5] (c) at (2,0) {\phantom{$0$}};
  \node[draw,fill=white,shape=circle,scale=0.5] (d) at (3,0) {\phantom{$0$}};
  \node[draw,fill=white,shape=circle,scale=0.5] (e) at (4,0) {\phantom{$0$}};
  \draw (a)--(b)--(c)--(d)--(e);
  \draw [->] (2,-0.5) -- (2,-1);
  \node[draw,fill=black,shape=circle,scale=0.5] (a) at (0,-1.5) {$\color{white}{1}$};
  \node[draw,fill=white,shape=circle,scale=0.5] (b) at (1,-1.5) {$1$};
  \node[draw,fill=white,shape=circle,scale=0.5] (c) at (2,-1.5) {\phantom{$0$}};
  \node[draw,fill=white,shape=circle,scale=0.5] (d) at (3,-1.5) {\phantom{$0$}};
  \node[draw,fill=white,shape=circle,scale=0.5] (e) at (4,-1.5) {\phantom{$0$}};
  \draw (a)--(b)--(c)--(d)--(e);
  \draw [->] (2,-2) -- (2,-2.5);
  \node[draw,fill=black,shape=circle,scale=0.5] (a) at (0,-3)
  {$\color{white}{1}$};
  \node[draw,fill=white,shape=circle,scale=0.5] (b) at (1,-3) {$1$};
  \node[draw,fill=white,shape=circle,scale=0.5] (c) at (2,-3) {$1$};
  \node[draw,fill=white,shape=circle,scale=0.5] (d) at (3,-3) {\phantom{$0$}};
  \node[draw,fill=white,shape=circle,scale=0.5] (e) at (4,-3) {\phantom{$0$}};
  \draw (a)--(b)--(c)--(d)--(e);
  \draw [->] (2,-3.5) -- (2,-4);
  \node[draw,fill=black,shape=circle,scale=0.5] (a) at (0,-4.5)
  {$\color{white}{1}$};
  \node[draw,fill=white,shape=circle,scale=0.5] (b) at (1,-4.5) {$1$};
  \node[draw,fill=white,shape=circle,scale=0.5] (c) at (2,-4.5) {$1$};
  \node[draw,fill=white,shape=circle,scale=0.5] (d) at (3,-4.5) {$1$};
  \node[draw,fill=white,shape=circle,scale=0.5] (e) at (4,-4.5) {\phantom{$0$}};
  \draw (a)--(b)--(c)--(d)--(e);
  \draw [->] (2,-5) -- (2,-5.5);
  \node[draw,fill=black,shape=circle,scale=0.5] (a) at (0,-6)
  {$\color{white}{1}$};
  \node[draw,fill=white,shape=circle,scale=0.5] (b) at (1,-6) {$1$};
  \node[draw,fill=white,shape=circle,scale=0.5] (c) at (2,-6) {$1$};
  \node[draw,fill=white,shape=circle,scale=0.5] (d) at (3,-6) {$1$};
  \node[draw,fill=white,shape=circle,scale=0.5] (e) at (4,-6) {$1$};
  \draw (a)--(b)--(c)--(d)--(e);
\end{tikzpicture}
\qquad \qquad \qquad
\begin{tikzpicture}[scale=1]
  \node[draw,shape=circle,scale=0.5] (a) at (-0.5,0) {\phantom{$0$}};
  \node[draw,fill=white,shape=circle,scale=0.5] (b) at (0.5,0) {\phantom{$0$}};
  \node[draw,fill=white,shape=circle,scale=0.5] (c) at (1.5,0) {\phantom{$0$}};
  \node[draw,fill=black,shape=circle,scale=0.5] (d) at (-0.5,1) {$\color{white}{1}$};
  \node[draw,fill=white,shape=circle,scale=0.5] (e) at (-1.5,0) {\phantom{$0$}};
  \draw (e) -- (a) -- (b)--(c);
  \draw (a) -- (d);
  \draw [->] (-0.5,-0.5) -- (-0.5,-1);
\node[draw,shape=circle,scale=0.5] (a) at (-0.5,-2.5) {$1$};
  \node[draw,fill=white,shape=circle,scale=0.5] (b) at (0.5,-2.5) {\phantom{$0$}};
  \node[draw,fill=white,shape=circle,scale=0.5] (c) at (1.5,-2.5) {\phantom{$0$}};
  \node[draw,fill=black,shape=circle,scale=0.5] (d) at (-0.5,-1.5) {$\color{white}{1}$};
  \node[draw,fill=white,shape=circle,scale=0.5] (e) at (-1.5,-2.5) {\phantom{$0$}};
  \draw (e) -- (a) -- (b)--(c);
  \draw (a) -- (d);
  \draw [->] (-1,-3) -- (-2,-3.5);
  \draw [->] (0,-3) -- (1,-3.5);
\node[draw,shape=circle,scale=0.5] (a) at (-3,-5) {$1$};
  \node[draw,fill=white,shape=circle,scale=0.5] (b) at (-2,-5) {\phantom{$0$}};
  \node[draw,fill=white,shape=circle,scale=0.5] (c) at (-1,-5) {\phantom{$0$}};
  \node[draw,fill=black,shape=circle,scale=0.5] (d) at (-3,-4) {$\color{white}{1}$};
  \node[draw,fill=white,shape=circle,scale=0.5] (e) at (-4,-5) {$1$};
  \draw (e) -- (a) -- (b)--(c);
  \draw (a) -- (d);
\node[draw,shape=circle,scale=0.5] (a) at (2,-5) {$1$};
  \node[draw,fill=white,shape=circle,scale=0.5] (b) at (3,-5) {$1$};
  \node[draw,fill=white,shape=circle,scale=0.5] (c) at (4,-5) {\phantom{$0$}};
  \node[draw,fill=black,shape=circle,scale=0.5] (d) at (2,-4) {$\color{white}{1}$};
  \node[draw,fill=white,shape=circle,scale=0.5] (e) at (1,-5) {\phantom{$0$}};
  \draw (e) -- (a) -- (b)--(c);
  \draw (a) -- (d);
  \draw [->] (-2,-5.5) -- (-1,-6);
  \draw [->] (1,-5.5) -- (0,-6);
  \draw [->] (-3,-5.5) -- (-3,-6);
  \draw [->] (2,-5.5) -- (2,-6);
  \node[draw,shape=circle,scale=0.5] (a) at (-3,-7.5) {$1$};
  \node[draw,fill=white,shape=circle,scale=0.5] (b) at (-2,-7.5) {$1$};
  \node[draw,fill=white,shape=circle,scale=0.5] (c) at (-1,-7.5) {\phantom{$0$}};
  \node[draw,fill=black,shape=circle,scale=0.5] (d) at (-3,-6.5) {$\color{white}{1}$};
  \node[draw,fill=white,shape=circle,scale=0.5] (e) at (-4,-7.5){$1$};
  \draw (e) -- (a) -- (b)--(c);
  \draw (a) -- (d);
  \node[draw,shape=circle,scale=0.5] (a) at (2,-7.5) {$1$};
  \node[draw,fill=white,shape=circle,scale=0.5] (b) at (3,-7.5) {$1$};
  \node[draw,fill=white,shape=circle,scale=0.5] (c) at (4,-7.5) {$1$};
  \node[draw,fill=black,shape=circle,scale=0.5] (d) at (2,-6.5) {$\color{white}{1}$};
  \node[draw,fill=white,shape=circle,scale=0.5] (e) at (1,-7.5){\phantom{$0$}};
  \draw (e) -- (a) -- (b)--(c);
  \draw (a) -- (d);
  \draw [->] (-2,-8) -- (-1,-8.5);
  \draw [->] (1,-8) -- (0,-8.5);
  \draw [->] (-3,-8) -- (-3,-8.5);
  \draw [->] (2,-8) -- (2,-8.5);
  \node[draw,shape=circle,scale=0.5] (a) at (-3,-10) {$2$};
  \node[draw,fill=white,shape=circle,scale=0.5] (b) at (-2,-10) {$1$};
  \node[draw,fill=white,shape=circle,scale=0.5] (c) at (-1,-10) {\phantom{$0$}};
  \node[draw,fill=black,shape=circle,scale=0.5] (d) at (-3,-9) {$\color{white}{1}$};
  \node[draw,fill=white,shape=circle,scale=0.5] (e) at (-4,-10){$1$};
  \draw (e) -- (a) -- (b)--(c);
  \draw (a) -- (d);
  \node[draw,shape=circle,scale=0.5] (a) at (2,-10) {$1$};
  \node[draw,fill=white,shape=circle,scale=0.5] (b) at (3,-10) {$1$};
  \node[draw,fill=white,shape=circle,scale=0.5] (c) at (4,-10) {$1$};
  \node[draw,fill=black,shape=circle,scale=0.5] (d) at (2,-9) {$\color{white}{1}$};
  \node[draw,fill=white,shape=circle,scale=0.5] (e) at (1,-10){$1$};
  \draw (e) -- (a) -- (b)--(c);
  \draw (a) -- (d);
   \draw [->] (-2,-10.5) -- (-1,-11);
  \draw [->] (1,-10.5) -- (0,-11);
  \node[draw,shape=circle,scale=0.5] (a) at (-0.5,-12.5) {$2$};
  \node[draw,fill=white,shape=circle,scale=0.5] (b) at (0.5,-12.5) {$1$};
  \node[draw,fill=white,shape=circle,scale=0.5] (c) at (1.5,-12.5) {$1$};
  \node[draw,fill=black,shape=circle,scale=0.5] (d) at (-0.5,-11.5) {$\color{white}{1}$};
  \node[draw,fill=white,shape=circle,scale=0.5] (e) at (-1.5,-12.5) {$1$};
  \draw (e) -- (a) -- (b)--(c);
  \draw (a) -- (d);
  \draw [->] (-0.5,-13) -- (-0.5,-13.5);
  \node[draw,shape=circle,scale=0.5] (a) at (-0.5,-15) {$2$};
  \node[draw,fill=white,shape=circle,scale=0.5] (b) at (0.5,-15) {$2$};
  \node[draw,fill=white,shape=circle,scale=0.5] (c) at (1.5,-15) {$1$};
  \node[draw,fill=black,shape=circle,scale=0.5] (d) at (-0.5,-14) {$\color{white}{1}$};
  \node[draw,fill=white,shape=circle,scale=0.5] (e) at (-1.5,-15) {$1$};
  \draw (e) -- (a) -- (b)--(c);
  \draw (a) -- (d);
\end{tikzpicture}
\caption{The set of configurations of the Kostant game modifying two different vertices on $A_4$.}
\label{fig:a4_mod2vértices}
\end{figure}
\end{example}

Although the main objective of the article \cite{CaviedesCastro2022} was to solve a problem in algebraic geometry (where counting roots via the Kostant game was crucial), fundamental results on the dynamics of the game itself were established in the process. In particular, adapting the notation of the article to that established in this work, we have the following characterization for the case of a single modified vertex:

\begin{theorem}[{\cite[Theorem 3.19]{CaviedesCastro2022}}]
    \label{theorem:3.19caviedes}
    A sequence of moves of the modified Kostant game on the co-root Dynkin diagram of $\Gamma$ at vertex $j$ encodes the reduced expression of some element in $W^{S \setminus \{j\}}$. Conversely, any reduced expression of an element in $W^{S \setminus \{j\}}$ can be obtained in this way.
\end{theorem}

This theorem and those to follow will be proven in a more general version in the next section. It is also the central tool underpinning the proofs of the theorems in the article, which in our case concerning the modified Kostant game, allow us to translate between the language of valid game moves and that of the algebra of Weyl groups.

A theorem that follows from the previous result and differentiates the original game in the multiply-laced case from our new modified Kostant game is the following.

\begin{theorem}[{\cite[Theorem 3.12]{CaviedesCastro2022}}]
\label{theorem:3.12caviedes}
The modified Kostant game on a vertex of a Dynkin diagram leads to a unique terminal configuration.
\end{theorem}

Note that unlike the case of the standard game on multiply-laced diagrams, here we do have a unique final configuration; that is, we recover the property we already had in the simply-laced case thanks to Theorem \ref{teo:unicofin}.

Let us now look at an example to visualize the previous theorems by modifying a multiply-laced Dynkin diagram. The procedure in checking the game steps is the central idea used in \cite{CaviedesCastro2022} to prove Theorem \ref{theorem:3.19caviedes}.


\begin{example}[The Game on $B_2$ with one modified vertex]
Let us explore a multiply-laced case. Consider the Dynkin diagram of type $B_2$. This system has one short root (associated with vertex 1, $\alpha_1$) and one long root (associated with vertex 2, $\alpha_2$). The Weyl group $W(B_2)$ is isomorphic to the dihedral group $D_8$ and has 8 elements.

\textbf{Game Configuration.}
We modify only the vertex of the short root, so the set of modified vertices is $I = \{1\}$, and the set of unmodified vertices is $J = \{2\}$. The associated parabolic subgroup is $W_J = W_{\{2\}} = \langle s_2 \rangle = \{\mathrm{id}, s_2\}$. The set $W^J = W/W_J$ contains $8/2 = 4$ elements.

The Dynkin diagram of type $B_2$ has a double arrow pointing from the long root vertex (2) to the short one (1). The neighbors within the diagram are $N(1) = \{2\}$ and $N(2) = \{1\}$, with asymmetric arrow numbers:
$$
n_{1,2} = 2 \quad (\text{arrows from 2 to 1}) \quad \text{and} \quad n_{2,1} = 1 \quad (\text{arrows from 1 to 2}).
$$
The unhappiness condition for a vertex $v$ in a configuration $c = c_1\alpha_1 + c_2\alpha_2$ is taken from our general definition:
$$
c_v < \dfrac{1}{2} \left( \sum_{u \in N(v)} n_{v,u}c_u + \sum_{p \in I} \delta_{vp} \right).
$$
Applying this to our two vertices (remembering that $I=\{1\}$):
\begin{itemize}
    \item For Vertex 1 ($v=1$):
    $ c_1 < \dfrac{1}{2} \left( n_{1,2}c_2 + \delta_{1,1} \right) = \dfrac{1}{2}(2c_2 + 1) $.
    \item For Vertex 2 ($v=2$):
    $ c_2 < \dfrac{1}{2} \left( n_{2,1}c_1 + \delta_{2,1} \right) = \dfrac{1}{2}(c_1 + 0) = \dfrac{c_1}{2} $.
\end{itemize}

\textbf{A Game Playthrough.}
We start with the initial configuration $c_0 = (0,0)$.
\begin{itemize}
    \item \textbf{Step 0:} $c_0=(0,0)$.
    Vertex 1 is sad: $0 < 1/2(0+1)$. Vertex 2 is not: $0 \not< 0/2$.
    The only possible move is on vertex 1. Move sequence: $(1)$.
    
    \item \textbf{Step 1:} Play on $v=1$.
    $c_1 \to -c_1 + n_{1,2}c_2 + 1 = -0 + 2(0) + 1 = 1$.
    The new configuration is $c_1=(1,0)$.
    Vertex 1: $1 \not< 1/2(0+1)$, is happy.
    Vertex 2: $0 < 1/2(1)$, is sad.
    The only possible move is on 2. Sequence: $(1,2)$.

    \item \textbf{Step 2:} Play on $v=2$.
    $c_2 \to -c_2 + n_{2,1}c_1 + 0 = -0 + 1(1) + 0 = 1$.
    The new configuration is $c_2=(1,1)$.
    Vertex 1: $1 < 1/2(2(1)+1) = 1.5$. Is sad.
    Vertex 2: $1 \not< 1/2(1)$. Is happy.
    The only possible move is on 1. Sequence: $(1,2,1)$.

    \item \textbf{Step 3:} Play on $v=1$.
    $c_1 \to -c_1 + n_{1,2}c_2 + 1 = -1 + 2(1) + 1 = 2$.
    The new configuration is $c_3=(2,1)$.
    Vertex 1: $2 \not< 1/2(2(1)+1) = 1.5$. Happy.
    Vertex 2: $1 \not< 2/2$. Happy.
    Both vertices are happy. The game ends.
\end{itemize}

The sequence of moves was $(i_1, i_2, i_3) = (1, 2, 1)$, producing configurations $(0,0) \to (1,0) \to (1,1) \to (2,1)$.

\begin{figure}[H]
\centering
\begin{tikzpicture}[scale=1.5]

  \tikzstyle{double-arrow}=[double, double distance=1.5pt, -implies, shorten >=4pt, shorten <=4pt]

  \node[draw,fill=black,shape=circle,scale=0.5] (s1_0) at (0,7) {$\color{white}{1}$};
  \node[draw,fill=white,shape=circle,scale=0.5] (v1_0) at (0,6) {\phantom{$0$}};
  \node[draw,fill=white,shape=circle,scale=0.5] (v2_0) at (1.2,6) {\phantom{$0$}};
  \draw (s1_0) -- (v1_0);
  \draw[double-arrow] (v2_0) -- (v1_0); 
  \draw[->] (0.6, 5.8) -- (0.6, 5.2);

  \node[draw,fill=black,shape=circle,scale=0.5] (s1_1) at (0,5) {$\color{white}{1}$};
  \node[draw,fill=white,shape=circle,scale=0.5] (v1_1) at (0,4) {$1$};
  \node[draw,fill=white,shape=circle,scale=0.5] (v2_1) at (1.2,4) {\phantom{$0$}};
  \draw (s1_1) -- (v1_1);
  \draw[double-arrow] (v2_1) -- (v1_1);
  \draw[->] (0.6, 3.8) -- (0.6, 3.2);

  \node[draw,fill=black,shape=circle,scale=0.5] (s1_2) at (0,3) {$\color{white}{1}$};
  \node[draw,fill=white,shape=circle,scale=0.5] (v1_2) at (0,2) {$1$};
  \node[draw,fill=white,shape=circle,scale=0.5] (v2_2) at (1.2,2) {$1$};
  \draw (s1_2) -- (v1_2);
  \draw[double-arrow] (v2_2) -- (v1_2);
  \draw[->] (0.6, 1.8) -- (0.6, 1.2);

  \node[draw,fill=black,shape=circle,scale=0.5] (s1_3) at (0,1) {$\color{white}{1}$};
  \node[draw,fill=white,shape=circle,scale=0.5] (v1_3) at (0,0) {$2$};
  \node[draw,fill=white,shape=circle,scale=0.5] (v2_3) at (1.2,0) {$1$};
  \draw (s1_3) -- (v1_3);
  \draw[double-arrow] (v2_3) -- (v1_3);

\end{tikzpicture}
\caption{The modified Kostant game on $B_2$ altering the short root vertex.}
\label{fig:kostantb2mod1}
\end{figure}
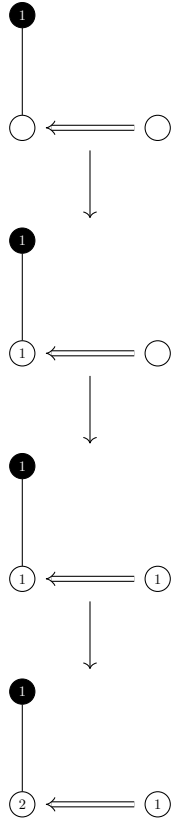

\textbf{Parallel Algebraic Simulation}
The theorem establishes a bijection between the reduced expressions of $W^J$ (the \textbf{algebra}) and the game plays on the \textbf{co-root} diagram (the \textbf{game}).
\begin{itemize}
    \item The \textbf{algebra} of a root system $X$ is determined by its Cartan matrix $A(X)$, which defines the action of reflections.
    \item The \textbf{game} on a diagram is defined by the integers $n_{i,j}$. The rules of the root diagram $\Gamma_X$ are $n_{i,j} = -A_{ij}(X)$, while those of the co-root diagram $\tilde{\Gamma}_X$ are $\tilde{n}_{i,j} = -A_{ji}(X)$.
\end{itemize}
Our goal is to simulate the game on the board $\Gamma_{B_2}$. According to the theorem, we must find a root system $X$ whose algebraic engine simulates a game on its co-root board $\tilde{\Gamma}_X$ that is identical to our target board $\Gamma_{B_2}$.
$$ \text{We seek } X \text{ such that } \tilde{\Gamma}_X = \Gamma_{B_2} $$
This equality of boards implies an equality in their connection rules:
$$ \tilde{n}_{i,j}(X) = n_{i,j}(B_2) \quad \implies \quad -A_{ji}(X) = -A_{ij}(B_2) $$
This gives us the condition $A_{ji}(X) = A_{ij}(B_2)$, which is the definition of the transpose matrix: $A(X) = A(B_2)^T$.
From root system theory, we know that $A(B_2)^T = A(C_2)$.

To simulate the game on the \textbf{root} diagram of $B_2$, we are forced by the theorem to use the \textbf{algebraic} engine of the dual system, $C_2$.

We will use the algebraic engine of $C_2$, where $\alpha_1$ is long and $\alpha_2$ is short. The reflection formulas are:
$$ s_1(\alpha_2) = \alpha_2 + 2\alpha_1 \quad \text{and} \quad s_2(\alpha_1) = \alpha_1 + \alpha_2. $$
The state vector is $\beta = \beta_1$. The action on $\beta$ is $s_1(\beta) = \beta+\alpha_1$ and $s_2(\beta) = \beta$.

\begin{itemize}
    \item \textbf{Step 0:} $w_0 = \mathrm{id}$. $v_0 = \beta$. Configuration $c_0 = (0,0)$. Matches.
    
    \item \textbf{Step 1:} $w_1 = s_1$. $v_1 = s_1(v_0) = s_1(\beta) = \beta + \alpha_1$. Configuration $c_1 = (1,0)$. Matches.
    
    \item \textbf{Step 2:} $w_2 = s_2 s_1$. $v_2 = s_2(v_1) = s_2(\beta + \alpha_1) = s_2(\beta) + s_2(\alpha_1) = \beta + (\alpha_1 + \alpha_2) = \alpha_1+\alpha_2+\beta$. Configuration $c_2=(1,1)$. Matches.
    
    \item \textbf{Step 3:} $w_3 = s_1 s_2 s_1$. $v_3 = s_1(v_2) = s_1(\alpha_1+\alpha_2+\beta) = s_1(\alpha_1) + s_1(\alpha_2) + s_1(\beta) = (-\alpha_1) + (\alpha_2+2\alpha_1) + (\beta+\alpha_1) = 2\alpha_1+\alpha_2+\beta$. Configuration $c_3=(2,1)$. Matches.
\end{itemize}
The algebraic simulation, using the correct engine of $C_2$, perfectly reproduces the manual play on $B_2$.

\textbf{Connection with the Weyl Group.}
The sequence of moves $(1, 2, 1)$ corresponds to the element $w=s_1s_2s_1 \in W(B_2)$. The inversion set (calculated in $B_2$) is $\mathcal{I}(w) = \{\alpha_1, s_1(\alpha_2), s_1s_2(\alpha_1)\} = \{\alpha_1, \alpha_2+2\alpha_1, \alpha_1+\alpha_2\}$.

The parabolic subgroup is $W_J = W_{\{2\}}$, and its positive roots are $R_J^+ = \{\alpha_2\}$. Since none of the roots in $\mathcal{I}(w)$ belong to $R_J^+$, it is confirmed that $w \in W^J$.

The set of minimal length representatives for the quotient $W(B_2)/W_{\{2\}}$ is
$$ W^J = \{\mathrm{id}, s_1, s_2s_1, s_1s_2s_1\}. $$
Our game has constructed the longest element, $s_1s_2s_1$, of this set, just as the theorem predicts.
\end{example}

Furthermore, \cite{CaviedesCastro2022} proves the following theorem which allows us to perform a positive root count. In the next section, a generalized version of this theorem will be demonstrated.

\begin{theorem}[{\cite[Theorem 3.14]{CaviedesCastro2022}}]
\label{theorem:3.14caviedes}
Given a Dynkin diagram $\Gamma$, if we denote by $h_j+1$ the height of the unique terminal configuration of the modified Kostant game on the co-root diagram $\check{\Gamma}$ at vertex $j$, then $\sum_{\alpha \in \Phi^+} \alpha = \sum_{\alpha_j \in S} h_j \alpha_j$.
\end{theorem}

\subsection{The game modifying multiple vertices}
\label{sec3.2}

We will now see a generalization of the modified Kostant game allowing for the simultaneous modification of multiple vertices of the Dynkin diagram. This will later allow us to take the idea from the previous section to characterize the quotients of Weyl groups with their parabolic subgroups in a more general way.

\begin{definition}[Generalized Kostant game with $n$ modified vertices]
\label{theorem:kostantgenn}

Let $\Phi$ be a root system with simple basis $\Delta=\{\alpha_i\}_{i=1}^r$ and let $S=\{1,\dots,r\}$. Fix a subset $I=\{p_1,\dots,p_n\}\subset S$ (the modified vertices) and write $J=S\setminus I$. 
To the Dynkin diagram $\Gamma$ of $\Phi$, we add, for each $p\in I$, an auxiliary vertex $\bar p$ with an arrow $\bar p\to p$; the chip located at $\bar p$ has value $1$ and never moves.

Let $N(v)$ denote the set of neighboring vertices of $v$ in the Dynkin diagram. A vertex $v$ is sad in a configuration $c = \sum_{j=1}^r c_j \alpha_j$ if
$$
c_v < \dfrac{1}{2} \left( \sum_{u\in N(v)} n_{v,u}\,c_u + \sum_{p\in I}\delta_{vp} \right),
$$
where $n_{v,u}$ is the number of arrows from $u$ to $v$ and the term $\sum_{p\in I}\delta_{vp}$ refers to the value 1 found when taking a vertex that was modified; if it is not modified, this term is 0. The update rule at a sad vertex $v$ is
$$
c_v \longrightarrow -c_v + \sum_{u\in N(v)} n_{v,u}\,c_u + \sum_{p\in I}\delta_{vp}.
$$
Intuitively, this last term, $\sum_{p\in I}\delta_{vp}$, simply adds 1 to the sum of the neighbors if the vertex $v$ itself was modified (i.e., if $v \in I$), representing the ``influence'' of its external source.

The goal of the game is the same as before: to make all vertices of $\Gamma$ happy or excited.
\end{definition}

\begin{example}
    To visualize this generalized modified version of the game on the $A_2$ diagram, we modify both the first and the second vertex. The set of configurations is presented as follows.

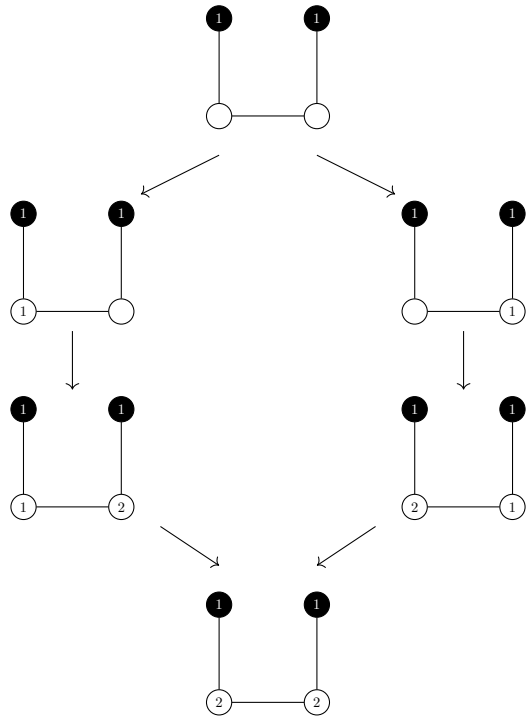
\begin{figure}[H]
\centering
\begin{tikzpicture}[scale=1.3]
  \node[draw,fill=black,shape=circle,scale=0.5] (a) at (0,7) {$\color{white}{1}$};
  \node[draw,fill=white,shape=circle,scale=0.5] (b) at (0,6) {\phantom{$0$}};
  \node[draw,fill=white,shape=circle,scale=0.5] (c) at (1,6) {\phantom{$0$}};
  \node[draw,fill=black,shape=circle,scale=0.5] (d) at (1,7) {$\color{white}{1}$};
  \draw (a) -- (b) -- (c) -- (d);
  \draw [->] (0,5.6) -- (-0.8,5.2);
  \draw [->] (1,5.6) -- (1.8,5.2);
  \node[draw,fill=black,shape=circle,scale=0.5] (a) at (-2,5) {$\color{white}{1}$};
  \node[draw,fill=white,shape=circle,scale=0.5] (b) at (-2,4) {$1$};
  \node[draw,fill=white,shape=circle,scale=0.5] (c) at (-1,4) {\phantom{$0$}};
  \node[draw,fill=black,shape=circle,scale=0.5] (d) at (-1,5) {$\color{white}{1}$};
  \draw (a) -- (b) -- (c) -- (d);
  \draw [->] (-1.5,3.8) -- (-1.5,3.2);

  \node[draw,fill=black,shape=circle,scale=0.5] (a) at (2,5) {$\color{white}{1}$};
  \node[draw,fill=white,shape=circle,scale=0.5] (b) at (2,4) {\phantom{$0$}};
  \node[draw,fill=white,shape=circle,scale=0.5] (c) at (3,4) {$1$};
  \node[draw,fill=black,shape=circle,scale=0.5] (d) at (3,5) {$\color{white}{1}$};
  \draw (a) -- (b) -- (c) -- (d);
  \draw [->] (2.5,3.8) -- (2.5,3.2);
  \node[draw,fill=black,shape=circle,scale=0.5] (a) at (-2,3) {$\color{white}{1}$};
  \node[draw,fill=white,shape=circle,scale=0.5] (b) at (-2,2) {$1$};
  \node[draw,fill=white,shape=circle,scale=0.5] (c) at (-1,2) {$2$};
  \node[draw,fill=black,shape=circle,scale=0.5] (d) at (-1,3) {$\color{white}{1}$};
  \draw (a) -- (b) -- (c) -- (d);
  \draw [->] (-0.6,1.8) -- (0,1.4);

  \node[draw,fill=black,shape=circle,scale=0.5] (a) at (2,3) {$\color{white}{1}$};
  \node[draw,fill=white,shape=circle,scale=0.5] (b) at (2,2) {$2$};
  \node[draw,fill=white,shape=circle,scale=0.5] (c) at (3,2) {$1$};
  \node[draw,fill=black,shape=circle,scale=0.5] (d) at (3,3) {$\color{white}{1}$};
  \draw (a) -- (b) -- (c) -- (d);
  \draw [->] (1.6,1.8) -- (1,1.4);
  \node[draw,fill=black,shape=circle,scale=0.5] (a) at (0,1) {$\color{white}{1}$};
  \node[draw,fill=white,shape=circle,scale=0.5] (b) at (0,0) {$2$};
  \node[draw,fill=white,shape=circle,scale=0.5] (c) at (1,0) {$2$};
  \node[draw,fill=black,shape=circle,scale=0.5] (d) at (1,1) {$\color{white}{1}$};
  \draw (a) -- (b) -- (c) -- (d);

\end{tikzpicture}
\caption{The modified Kostant game on $A_2$ altering two vertices.}
\label{fig:A2mod2ver}
\end{figure}
\end{example}

Working through this example, it becomes evident that the reflections $s_1 s_2 s_1$ and $s_2 s_1 s_2$ are equal. Therefore, the group $W/W_P$ is characterized as being generated by $s_1$ and $s_2$ with the relations $ s_1^2 = s_2^2 = 1$ and $s_1 s_2 s_1 = s_2 s_1 s_2$, which is none other than the permutation group on three elements $S_3$. This relationship between the configurations of the modified Kostant game and the quotient of Weyl groups with parabolic subgroups will be explored in the following subsections.

\subsection{Algebraic model of the game}

To prove our generalization of Theorem \ref{theorem:3.19caviedes}, we need to construct an algebraic framework that faithfully models our game. This framework must include not only the roots of the system but also the external sources (new, always-happy vertices) that initiate the game dynamics.

Let $V$ be the real vector space generated by the simple roots $\{\alpha_i\}_{i\in S}$ of a root system, with its bilinear form $(\cdot,\cdot)$ invariant under the Weyl group $W$. For a set of modified vertices $I\subseteq S$, we introduce a formal vector $\beta_p$ for each $p\in I$, which will serve as the algebraic representation of the chip source at that vertex. The extended state space is then defined as the direct sum:
$$ V' = V \oplus \bigoplus_{p\in I}\mathbb{R}\,\beta_p. $$

Next, we define a symmetric bilinear form $\langle\cdot,\cdot\rangle$ on this new space $V'$. This form must encode both the internal interactions of the root system and the interactions of the new sources with the board. We define it on the basis $\{\alpha_i\}_{i\in S}\cup\{\beta_p\}_{p\in I}$ as follows:

\begin{enumerate}
    \item $\langle\alpha_i,\alpha_j\rangle = (\alpha_i,\alpha_j)$. The form restricts to the usual one on the original space $V$.
    \item $\langle\beta_p,\beta_q\rangle = \delta_{pq}$. 
    \item $\langle\beta_p,\alpha_i^\vee\rangle = -\delta_{pi}$. This is the key condition modeling our game. It translates to $\langle\beta_p,\alpha_i\rangle = -\delta_{pi} \dfrac{(\alpha_i,\alpha_i)}{2}$. By definition, each source $\beta_p$ is orthogonal to every simple co-root $\alpha_i^\vee$ except its own ($i=p$), where the interaction has multiplicity 1. This corresponds exactly to the game rule where each external source connects only to its corresponding modified vertex.
\end{enumerate}

With this structure, the action of a simple reflection $s_i \in W$ on $V'$ is naturally determined. For a vector $v \in V'$, the action is $s_i(v)=v-\langle v,\alpha_i^\vee\rangle\alpha_i$. Applying it to our basis, we obtain the explicit expressions:
$$ s_i(\alpha_j)=\alpha_j - A_{ji}\alpha_i, $$
$$ s_i(\beta_p)=\beta_p - \langle\beta_p, \alpha_i^\vee\rangle \alpha_i = \beta_p - (-\delta_{pi})\alpha_i = \beta_p + \delta_{pi}\alpha_i. $$
Note that $s_i$ only affects the source $\beta_i$, leaving the others ($\beta_p$ with $p\neq i$) invariant, as expected from a reflection localized at vertex $i$.

To ensure this construction is mathematically sound, we verify two fundamental properties in the following lemmas.

\begin{lemma}[Action of $W$ on $V'$]
The above assignment defines a well-defined action of the Weyl group $W$ on $V'$.
\end{lemma}

\begin{proof}
    To demonstrate that the assignment defines a group action, we must verify that the defining relations of the Weyl group $W$ are preserved. Since $W$ is a Coxeter group generated by $S=\{s_i\}$, it suffices to verify the relations $s_i^2 = \mathrm{id}$ for all $i$, and the braid relations $(s_i s_j)^{m_{ij}} = \mathrm{id}$ for all $i \neq j$.
    
    The action is linear by construction. We verify the relations on the basis of $V'$.
    
    \begin{enumerate}
        \item \textbf{Involution relation ($s_i^2 = \mathrm{id}$):}
        The action on the basis $\{\alpha_j\}$ is the standard one, where this relation holds. For the source basis vectors $\{\beta_p\}$, we calculate:
        \begin{align*}
            s_i(s_i(\beta_p)) &= s_i(\beta_p + \delta_{pi}\alpha_i) \\
            &= s_i(\beta_p) + \delta_{pi}s_i(\alpha_i) \\
            &= (\beta_p + \delta_{pi}\alpha_i) + \delta_{pi}(-\alpha_i) = \beta_p.
        \end{align*}
        Therefore, $s_i^2$ acts as the identity on the entire basis of $V'$.
        
        \item \textbf{Braid relations ($(s_i s_j)^{m_{ij}} = \mathrm{id}$):}
        These relations already hold in the subspace $V$. We only need to verify their action on the source vectors $\beta_p$. The argument divides into cases according to the index $p$:
        
        \begin{itemize}
            \item \textbf{Case 1: $p \notin \{i, j\}$.} 
            In this case, both $\delta_{pi}=0$ and $\delta_{pj}=0$. Therefore, reflections $s_i$ and $s_j$ act as the identity on $\beta_p$:
            $$ s_i(\beta_p) = \beta_p \quad \text{and} \quad s_j(\beta_p) = \beta_p. $$
            Consequently, any product of $s_i$ and $s_j$, such as $(s_i s_j)^{m_{ij}}$, will also act as the identity on $\beta_p$, and the relation holds trivially.
            
            \item \textbf{Case 2: $p \in \{i, j\}$.} 
            Without loss of generality, assume $p=i$. We must verify that $(s_i s_j)^{m_{ij}}(\beta_i) = \beta_i$.
            
            Observe that the action of any $s_k$ on a vector of the form $\beta_i + v$ (with $v \in V$) decomposes as:
            $$ s_k(\beta_i + v) = s_k(\beta_i) + s_k(v) = (\beta_i + \delta_{ki}\alpha_k) + s_k(v). $$
            This means that the action on the source part ($\beta_i$) is independent of the action on the root part ($v$). Let $\pi_V: V' \to V$ be the projection onto the root space and $\pi_\beta: V' \to \mathbb{R}\beta_i$ be the projection onto the source space. The action of $w \in W$ respects this decomposition:
            $$ w(\beta_i) = \pi_\beta(w(\beta_i)) + \pi_V(w(\beta_i)). $$
            It can be shown by induction that for any $w \in W$, $\pi_\beta(w(\beta_i)) = \beta_i$. The source component is never destroyed. The real action occurs in the root component. Therefore, we can write:
            $$ w(\beta_i) = \beta_i + v_w \quad \text{for some } v_w \in V. $$
            Now we apply the braid relation operator, $U = (s_i s_j)^{m_{ij}}$:
            $$ U(\beta_i) = \beta_i + v_U $$
            We know that the action of $U$ on $V$ is the identity, i.e., $U(v)=v$ for all $v \in V$. Our goal is to show that $v_U = \mathbf{0}$.
            
            By the linearity of the action and the decomposition we established:
            $$ U(\beta_i) = U(\beta_i + \mathbf{0}) = \beta_i + U(\mathbf{0}) = \beta_i + \mathbf{0} = \beta_i. $$
            The action of $U$ on the root part of $w(\beta_i)$ is simply the standard action of $W$ on $V$. Since $(s_i s_j)^{m_{ij}} = \mathrm{id}$ as a transformation in $V$, any root term that has accumulated by applying reflections to $\beta_i$ will cancel out in the end. We conclude that $(s_i s_j)^{m_{ij}}(\beta_i) = \beta_i$.
        \end{itemize}
    \end{enumerate}
    Having verified that both the involution and braid relations hold on a basis of $V'$, we conclude that the assignment is a well-defined group action.
\end{proof}

\begin{lemma}[Invariance of the bilinear form]
The bilinear form $\langle \cdot,\cdot \rangle$ defined on $V'$ is W-invariant.
\end{lemma}

\begin{proof}
We verify the invariance for simple reflections $s_i$ acting on the basis elements of $V'$.

\medskip

(i) For $\alpha_j, \alpha_k$:
$$
\begin{aligned}
\langle s_i\alpha_j, s_i\alpha_k \rangle 
&= \langle \alpha_j - A_{ji}\alpha_i, \alpha_k - A_{ki}\alpha_i \rangle \\
&= \langle\alpha_j,\alpha_k\rangle - A_{ki}\langle\alpha_j,\alpha_i\rangle - A_{ji}\langle\alpha_i,\alpha_k\rangle + A_{ji}A_{ki}\langle\alpha_i,\alpha_i\rangle \\
&= \langle\alpha_j,\alpha_k\rangle.
\end{aligned}
$$
The cancellation in the last step is due to $A_{ji} = 2\dfrac{(\alpha_j,\alpha_i)}{(\alpha_i,\alpha_i)}$.

\medskip

(ii) For $\alpha_j, \beta_p$:
$$
\begin{aligned}
\langle s_i\alpha_j, s_i\beta_p \rangle
&= \langle \alpha_j - A_{ji}\alpha_i, \beta_p + \delta_{pi}\alpha_i \rangle \\
&= \langle\alpha_j,\beta_p\rangle + \delta_{pi}\langle\alpha_j,\alpha_i\rangle - A_{ji}\langle\alpha_i,\beta_p\rangle - A_{ji}\delta_{pi}\langle\alpha_i,\alpha_i\rangle \\
&= \langle\alpha_j,\beta_p\rangle + \delta_{pi}\left(A_{ji}\dfrac{(\alpha_i,\alpha_i)}{2}\right) - A_{ji}\left(-\delta_{pi}\dfrac{(\alpha_i,\alpha_i)}{2}\right) - A_{ji}\delta_{pi}(\alpha_i,\alpha_i) \\
&= \langle\alpha_j,\beta_p\rangle.
\end{aligned}
$$
The last three terms cancel out, since $(\dfrac{1}{2} + \dfrac{1}{2} - 1)A_{ji}\delta_{pi}(\alpha_i,\alpha_i) = 0$.

\medskip

(iii) For $\beta_p, \beta_q$:
$$
\begin{aligned}
\langle s_i\beta_p, s_i\beta_q \rangle
&= \langle \beta_p + \delta_{pi}\alpha_i, \beta_q + \delta_{qi}\alpha_i \rangle \\
&= \langle\beta_p,\beta_q\rangle + \delta_{qi}\langle\beta_p,\alpha_i\rangle + \delta_{pi}\langle\alpha_i,\beta_q\rangle + \delta_{pi}\delta_{qi}\langle\alpha_i,\alpha_i\rangle \\
&= \langle\beta_p,\beta_q\rangle + \delta_{qi}\left(-\delta_{pi}\dfrac{(\alpha_i,\alpha_i)}{2}\right) + \delta_{pi}\left(-\delta_{qi}\dfrac{(\alpha_i,\alpha_i)}{2}\right) + \delta_{pi}\delta_{qi}(\alpha_i,\alpha_i) \\
&= \langle\beta_p,\beta_q\rangle.
\end{aligned}
$$
Again, the last three terms cancel out.

\medskip

In each case, $\langle s_i x, s_i y \rangle = \langle x, y \rangle$ is recovered, so the extended bilinear form is invariant under all simple reflections and, therefore, under the action of the entire $W$.
\end{proof}

With this algebraic machinery solidly established, we are ready to formalize the intuition we have developed. The following theorem rigorously establishes that our algebraic model is not just an analogy, but a perfect simulation of the game dynamics.

For what follows, we will work with the basis of simple roots $\Delta=\{\alpha_i\}$. Recall the fundamental connection between the geometry and the algebra of the inner product. For each simple root $\alpha_i$, we define its associated co-root:
$$
\alpha_i^\vee=\dfrac{2\alpha_i}{(\alpha_i,\alpha_i)}.
$$
The Cartan matrix of the root system $\Phi$, denoted by $A=(A_{ij})$, is defined with the convention
$$
A_{ij} \;:=\; \langle\alpha_j,\alpha_i^\vee\rangle \;=\; \dfrac{2(\alpha_j,\alpha_i)}{(\alpha_i,\alpha_i)}.
$$
Our game takes place on the co-root graph, whose geometry is dictated by the Cartan matrix of the dual system, $A^\vee=A^T$ (see \ref{prop:cartandual}). Therefore, the number of arrows $n_{i,j}$ from a vertex $j$ to a vertex $i$ in this graph relates to the transposed entry of the original matrix. For $i\neq j$, the relationship is:
$$
n_{i,j} \;=\; -A_{ji}.
$$
This convention, which connects the geometry of the game board (the co-root graph) with the algebra of the original root system (matrix $A$), will be used implicitly in the calculations that follow.

We denote
$$
W_J=\langle s_j\mid j\in J \rangle, \qquad
\Phi^+_J = \mathrm{span}_{\mathbb{Z}_{\ge0}}\{\alpha_j\mid j\in J\},
$$
and define
$$
W^J = \{w\in W \mid \mathcal{I}(w)\subset \Phi^+\setminus \Phi^+_J\}, 
\qquad
\mathcal{I}(w)=\{\alpha\in \Phi^+ \mid w(\alpha)\in -\Phi^+\}.
$$
Here $\mathcal{I}(w)$ is the inversion set of $w$ and $W^J$ is the set of minimal length representatives of the cosets of $W/W_J$: recall that by Theorem \ref{thm:equivalencia_WJ} we have
$$
w\in W^J \quad\Longleftrightarrow\quad \ell(ws_j)>\ell(w)\ \text{ for all } j\in J.
$$

\begin{theorem}[Equivalence between game and algebra]
    \label{prop:simulacion_algebraica}
    The update rule from a configuration $c_\lambda$ to $c_{\lambda+1}$ under a move at vertex $v$ in the generalized Kostant game is identical to the transformation of the coefficients of the root part of the algebraic state vector $v_\lambda$ under the action of the reflection $s_v$.
\end{theorem}

\begin{proof}
    As established in the construction of the algebraic state space, the action of the external sources is encoded by the relation $\langle\beta_p, \alpha_i^\vee\rangle = -\delta_{pi}$. For the aggregate state vector $\beta := \sum_{p\in I} \beta_p$, this implies:
    $$
    \langle\beta, \alpha_i^\vee\rangle = \sum_{p\in I} \langle\beta_p, \alpha_i^\vee\rangle = -\sum_{p\in I} \delta_{pi}.
    $$
    This term is $-1$ if $i \in I$ (the vertex is modified) and $0$ if $i \notin I$. This will be key to reproducing the game rule.
    
    Let $(i_1,\dots,i_t)$ be a sequence of moves. We define by recursion:
    $$
    w_0 = \mathrm{id},\qquad
    w_\lambda = s_{i_\lambda}w_{\lambda-1}\quad (1\le \lambda\le t).
    $$
    The evolution of the system is described by the state vectors:
    $$
    v_0 = \beta, \qquad v_\lambda := w_\lambda(\beta) = s_{i_\lambda}(v_{\lambda-1}).
    $$
    The chip configuration at step $\lambda$, denoted by $c_\lambda = \sum_{j=1}^r c_{\lambda,j}\alpha_j$, is defined via the decomposition:
    $$
    v_\lambda = c_\lambda + \beta = \sum_{j=1}^r c_{\lambda,j}\alpha_j + \beta,
    $$
    with the initial configuration $c_{0,j}=0$ for all $j$.
    
    Now, we explicitly simulate the action of a move at vertex $v=i_{\lambda+1}$. The new state $v_{\lambda+1}$ is obtained by applying the reflection $s_v$ to the previous state $v_\lambda$. Instead of applying the reflection to the entire vector $v_\lambda$ at once, we will use linearity and the action of $s_v$ on the basis vectors, which has already been established:
    \begin{align*}
        v_{\lambda+1} = s_v(v_\lambda) &= s_v\left( \sum_{j=1}^r c_{\lambda,j}\alpha_j + \beta \right) \\
        &= \sum_{j=1}^r c_{\lambda,j} s_v(\alpha_j) + s_v(\beta).
    \end{align*}
    Recall the formulas for the action of $s_v$:
    \begin{enumerate}
        \item $s_v(\alpha_j) = \alpha_j - A_{jv}\alpha_v$, where $A_{jv} = \langle \alpha_j, \alpha_v^\vee \rangle$.
        \item $s_v(\beta) = \beta + (\sum_{p\in I} \delta_{pv})\alpha_v = \beta - \langle\beta, \alpha_v^\vee\rangle \alpha_v$.
    \end{enumerate}
    Substituting these expressions into the equation for $v_{\lambda+1}$:
    \begin{align*}
        v_{\lambda+1} &= \sum_{j=1}^r c_{\lambda,j}(\alpha_j - A_{jv}\alpha_v) + \beta + \left(-\langle\beta, \alpha_v^\vee\rangle\right)\alpha_v \\
        &= \sum_{j=1}^r c_{\lambda,j}\alpha_j - \left(\sum_{j=1}^r c_{\lambda,j}A_{jv}\right)\alpha_v + \beta - \langle\beta, \alpha_v^\vee\rangle\alpha_v \\
        &= \left( \sum_{j=1}^r c_{\lambda,j}\alpha_j + \beta \right) - \left( \sum_{j=1}^r c_{\lambda,j}A_{jv} + \langle\beta, \alpha_v^\vee\rangle \right)\alpha_v.
    \end{align*}
    This last line is precisely $v_\lambda - \langle v_\lambda, \alpha_v^\vee \rangle \alpha_v$, showing the consistency of both approaches.
    
    Now, to find the new chip configuration $c_{\lambda+1}$, we reorganize the terms grouping the coefficients of each simple root $\alpha_k$:
    $$
    v_{\lambda+1} = \sum_{k \neq v} c_{\lambda,k}\alpha_k + c_{\lambda,v}\alpha_v - \left( \sum_{j=1}^r c_{\lambda,j}A_{jv} \right)\alpha_v + \beta - \langle\beta, \alpha_v^\vee\rangle\alpha_v.
    $$
    The new configuration vector $c_{\lambda+1} = v_{\lambda+1} - \beta$ is, therefore:
    $$
    \sum_{k=1}^r c_{\lambda+1,k}\alpha_k = \sum_{k \neq v} c_{\lambda,k}\alpha_k + \left( c_{\lambda,v} - \sum_{j=1}^r c_{\lambda,j}A_{jv} - \langle\beta, \alpha_v^\vee\rangle \right)\alpha_v.
    $$
    From here, the new coefficients $c_{\lambda+1,k}$ are deduced:
    \begin{itemize}
        \item For $k \neq v$, the coefficient does not change: $\mathbf{c_{\lambda+1,k} = c_{\lambda,k}}$.
        \item For $k=v$, the new coefficient is:
        $$
        c_{\lambda+1,v} = c_{\lambda,v} - \sum_{j=1}^r c_{\lambda,j}A_{jv} - \langle\beta, \alpha_v^\vee\rangle.
        $$
    \end{itemize}
    Let us expand the sum to reveal the game rule. We separate the term $j=v$ from the sum:
    \begin{align*}
        c_{\lambda+1,v} &= c_{\lambda,v} - \left( c_{\lambda,v}A_{vv} + \sum_{j \neq v} c_{\lambda,j}A_{jv} \right) - \langle\beta, \alpha_v^\vee\rangle \\
        &= c_{\lambda,v} - 2c_{\lambda,v} - \sum_{u \in N(v)} c_{\lambda,u}A_{uv} - \left(-\sum_{p\in I}\delta_{pv}\right) && \text{(Using } A_{vv}=2 \text{ and the definition of } \beta) \\
        &= -c_{\lambda,v} - \sum_{u \in N(v)} c_{\lambda,u}(-n_{v,u}) + \sum_{p\in I}\delta_{pv} && \text{(Using } n_{v,u} = -A_{uv} \text{ for } u\neq v) \\
        &= -c_{\lambda,v} + \sum_{u \in N(v)} n_{v,u}c_{\lambda,u} + \sum_{p\in I}\delta_{pv}.
    \end{align*}
    The update rule emerging from the algebraic model is:
    $$
    c_{\lambda+1,v} = -c_{\lambda,v} + \sum_{u \in N(v)} n_{v,u}c_{\lambda,u} + \sum_{p\in I}\delta_{pv}.
    $$
    This coincides exactly with the update rule for a sad vertex $v$ in Definition \ref{theorem:kostantgenn}, where the term $\sum_{p\in I} m_{p,v}$ corresponds to $\sum_{p\in I}\delta_{pv}$ for our choice of game parameters.
    
    Therefore, we have demonstrated that the action of the reflection $s_v$ in the algebraic state space faithfully reproduces the dynamics of the modified Kostant game on the co-root graph.
\end{proof}

\subsection{Reduced expressions in \texorpdfstring{$W/W_J$}{W/W_J}}

Having established the equivalence between game moves and the action of reflections in the algebraic state space $V'$ thanks to Theorem \ref{prop:simulacion_algebraica}, we are now in a position to demonstrate the central result of this work. Specifically, a canonical bijection is established between the valid game plays and the reduced expressions of the minimal length representatives in $W/W_J$.

\begin{theorem}[Correspondence between game plays and reduced words]
    \label{theorem:3.19general}
    Let $W$ be the Weyl group of a root system and $J \subseteq S$ a subset of simple reflections. There exists a canonical bijection between the sequences of moves of the Kostant game modified on $I=S\setminus J$ and the reduced expressions of the elements in the quotient $W^J$.
    
    Specifically, if we associate the word $w = s_{i_t}\cdots s_{i_1}$ to the sequence of moves $(i_1, \dots, i_t)$, then the sequence is a valid game play if and only if $w$ is a reduced expression for an element in $W^J$.
\end{theorem}

\begin{proof}

\textbf{($\Rightarrow$) A sequence of valid moves induces an element of $W^J$.}
Suppose that $(i_1, \dots, i_t)$ is a sequence of valid moves. Validity at step $\lambda$ (for $1 \le \lambda \le t$) is equivalent to the algebraic condition:
$$
\langle v_{\lambda-1}, \alpha_{i_\lambda}^\vee \rangle < 0.
$$
We define the integer quantity $K_\lambda$ from this condition:
$$
K_\lambda := -\langle v_{\lambda-1}, \alpha_{i_\lambda}^\vee \rangle.
$$
The validity of the move means that $K_\lambda$ must be a strictly positive integer. Using the $W$-invariance of the bilinear form, we can express $K_\lambda$ in terms of the initial state $\beta$:
$$
K_\lambda = -\langle w_{\lambda-1}(\beta), \alpha_{i_\lambda}^\vee \rangle = -\langle \beta, w_{\lambda-1}^{-1}(\alpha_{i_\lambda}^\vee) \rangle.
$$

The key object in this expression is the transformed co-root $\tilde{\gamma}^\vee := w_{\lambda-1}^{-1}(\alpha_{i_\lambda}^\vee)$. We will show that this co-root must be positive.

The fundamental reason is that, by construction, the aggregate vector $\beta$ has a non-positive inner product with any positive co-root. By definition, $\langle \beta, \alpha_j^\vee \rangle = -\sum_{p\in I}\delta_{pj} \le 0$ for every simple root $\alpha_j$. Since every positive co-root $\gamma^\vee \in (\Phi^\vee)^+$ is a linear combination with non-negative integer coefficients of simple co-roots, it follows by linearity that $\langle \beta, \gamma^\vee \rangle \le 0$.

Now, suppose for the sake of contradiction that $\tilde{\gamma}^\vee$ is a negative co-root. Then, its opposite $-\tilde{\gamma}^\vee$ would be a positive co-root. Calculating $K_\lambda$, we would have:
$$ K_\lambda = -\langle \beta, \tilde{\gamma}^\vee \rangle = \langle \beta, -\tilde{\gamma}^\vee \rangle \le 0. $$
Since $-\tilde{\gamma}^\vee$ is a positive co-root, we know that this inner product must be $\le 0$. This implies that $K_\lambda \le 0$, which directly contradicts the condition that the move is valid (which requires $K_\lambda > 0$).

Therefore, the transformed co-root $w_{\lambda-1}^{-1}(\alpha_{i_\lambda}^\vee)$ must be positive.

Now, to find the explicit value of $K_\lambda$, we expand the positive co-root $w_{\lambda-1}^{-1}(\alpha_{i_\lambda}^\vee)$ in the basis of simple co-roots:
$$
w_{\lambda-1}^{-1}(\alpha_{i_\lambda}^\vee) = \sum_{j=1}^r k_j \alpha_j^\vee, \quad \text{where } k_j \in \mathbb{Z}_{\ge 0}.
$$
Substituting this expansion into our formula for $K_\lambda$:
\begin{align*}
	K_\lambda &= -\left\langle \sum_{p \in I} \beta_p, \sum_{j=1}^r k_j \alpha_j^\vee \right\rangle = -\sum_{j=1}^r k_j \left\langle \sum_{p \in I} \beta_p, \alpha_j^\vee \right\rangle \\
	&= -\sum_{j=1}^r k_j \left( -\sum_{p \in I} \delta_{pj} \right) = \sum_{j=1}^r k_j \left( \delta_{j \in I} \right) \\
	&= \sum_{j \in I} k_j.
\end{align*}
The condition that the move is valid ($K_\lambda > 0$) translates to $\sum_{j \in I} k_j > 0$. Since all coefficients $k_j$ are non-negative integers, this implies that there must exist at least one index $j_0 \in I$ such that $k_{j_0} > 0$. This demonstrates that the co-root $w_{\lambda-1}^{-1}(\alpha_{i_\lambda}^\vee)$ cannot be written as a linear combination (with non-negative coefficients) of simple co-roots from $J=S\setminus I$. In other words:
$$
w_{\lambda-1}^{-1}(\alpha_{i_\lambda}^\vee) \in (\Phi^\vee)^+ \setminus (\Phi_J^\vee)^+.
$$
The map sending a root to its co-root ($\alpha \mapsto \alpha^\vee$) is a bijection from the set of roots $\Phi$ to $\Phi^\vee$ that is $W$-equivariant and preserves positivity. Therefore, the condition on the co-root is equivalent to a condition on the root $\tilde{\alpha}_\lambda := w_{\lambda-1}^{-1}(\alpha_{i_\lambda})$:
$$
\tilde{\alpha}_\lambda \in \Phi^+ \setminus \Phi_J^+.
$$
The first part, $\tilde{\alpha}_\lambda \in \Phi^+$, is the classical condition for the length of the Weyl element to increase: $\ell(s_{i_\lambda}w_{\lambda-1}) = \ell(w_{\lambda-1}) + 1$. Since this holds for all $\lambda$, the word $w_t = s_{i_t}\cdots s_{i_1}$ is a reduced expression.

The second part, $\tilde{\alpha}_\lambda \notin \Phi_J^+$, relates to the definition of $W^J$. Observe that, given the definition $w_t = s_{i_t}\cdots s_{i_1}$, the inversion set of $w_t$ corresponds precisely to the set of generated roots:
$$ \mathcal{I}(w_t) = \{\tilde{\alpha}_1, \dots, \tilde{\alpha}_t\}. $$
We have shown that all these roots are in $\Phi^+ \setminus \Phi_J^+$. Therefore, $\mathcal{I}(w_t) \subset \Phi^+ \setminus \Phi_J^+$, which by definition implies that $w_t \in W^J$.

\medskip

\textbf{($\Leftarrow$) An element of $W^J$ induces a sequence of valid moves.}
Conversely, let $w \in W^J$ and let $w = s_{i_t}\cdots s_{i_1}$ be one of its reduced expressions. For each $\lambda \in \{1, \dots, t\}$, we define $w_{\lambda-1} = s_{i_{\lambda-1}}\cdots s_{i_1}$.

The fact that the expression is reduced implies that the root $\tilde{\alpha}_\lambda := w_{\lambda-1}^{-1}(\alpha_{i_\lambda})$ is positive for all $\lambda$. Furthermore, since $w \in W^J$, its inverse inversion set, $\mathcal{I}(w^{-1}) = \{\tilde{\alpha}_1, \dots, \tilde{\alpha}_t\}$, is contained in $\Phi^+ \setminus \Phi_J^+$. Therefore, for each $\lambda$, we have:
$$
\tilde{\alpha}_\lambda \in \Phi^+ \setminus \Phi_J^+.
$$
This implies that the corresponding co-root, $w_{\lambda-1}^{-1}(\alpha_{i_\lambda}^\vee)$, is in $(\Phi^\vee)^+ \setminus (\Phi_J^\vee)^+$. Expanding this co-root in the basis of simple co-roots:
$$
w_{\lambda-1}^{-1}(\alpha_{i_\lambda}^\vee) = \sum_{j=1}^r k_j \alpha_j^\vee, \quad \text{with } k_j \in \mathbb{Z}_{\ge 0},
$$
the condition of non-membership in $(\Phi_J^\vee)^+$ assures us that at least one coefficient $k_{j_0}$ with $j_0 \in I$ must be strictly positive.

Now, we verify the validity of the move by calculating the integer $K_\lambda$:
$$
K_\lambda := -\langle v_{\lambda-1}, \alpha_{i_\lambda}^\vee \rangle = -\langle \beta, w_{\lambda-1}^{-1}(\alpha_{i_\lambda}^\vee) \rangle.
$$
Using the expansion of the co-root we just established, we find the same expression for $K_\lambda$ as in the previous implication:
$$
K_\lambda = \sum_{j \in I} k_j.
$$
Since we know that all $k_j$ are non-negative and at least one of them with index in $I$ is strictly positive, the sum $K_\lambda = \sum_{j \in I} k_j$ is a strictly positive integer.

Given that $K_\lambda > 0$, the validity condition $\langle v_{\lambda-1}, \alpha_{i_\lambda}^\vee \rangle = -K_\lambda < 0$ holds. This inequality confirms that the move at vertex $i_\lambda$ is valid. Since this argument applies for all $\lambda$ from $1$ to $t$, the sequence of reflections $(i_1, \dots, i_t)$ corresponds to a valid sequence of moves in the modified game.
\end{proof}


\begin{example}[The Game on $A_2$ with both vertices modified]
    
    To illustrate the theorem, consider the case of the Dynkin diagram of type $A_2$. The root system has two simple roots, $\alpha_1$ and $\alpha_2$. The Weyl group $W(A_2)$ is isomorphic to the dihedral group $D_6$ and has 6 elements.
    
    \textbf{Game Configuration.}
    We modify both vertices, so the set of modified vertices is $I = \{1, 2\}$, and the set of unmodified vertices is $J = \emptyset$. The associated parabolic subgroup is $W_J = W_\emptyset = \{\mathrm{id}\}$. The set of minimal length representatives is the Weyl group itself, $W^J = W(A_2)$.
    The Dynkin diagram of type $A_2$ has two vertices, 1 and 2, connected by an edge. The set of neighbors within the Dynkin diagram for each vertex is:
    $$
    N(1) = \{2\} \quad \text{and} \quad N(2) = \{1\}.
    $$
    In this diagram, $n_{1,2}=1$ and $n_{2,1}=1$.
    
    The unhappiness condition for a vertex $v \in \{1,2\}$ in a configuration $c = c_1\alpha_1 + c_2\alpha_2$ is taken from our general definition:
    $$
    c_v < \dfrac{1}{2} \left( \sum_{u \in N(v)} n_{v,u}c_u + \sum_{p \in I} \delta_{vp} \right).
    $$
    Applying this to our two vertices:
    \begin{itemize}
        \item For Vertex 1 ($v=1$):
        $$ c_1 < \dfrac{1}{2} \left( n_{1,2}c_2 + (\delta_{1,1} + \delta_{1,2}) \right) = \dfrac{1}{2}(c_2 + 1). $$
        \item For Vertex 2 ($v=2$):
        $$ c_2 < \dfrac{1}{2} \left( n_{2,1}c_1 + (\delta_{2,1} + \delta_{2,2}) \right) = \dfrac{1}{2}(c_1 + 1). $$
    \end{itemize}
    These are the rules that will govern our game.
    
    \textbf{A Game Playthrough.}
    We start with the initial configuration $c_0 = (0,0)$.
    \begin{itemize}
        \item \textbf{Step 0:} $c_0=(0,0)$.
        Both vertices are unhappy: $0 < 1/2$. We choose to play on vertex 1. Move sequence: $(1)$.
        
        \item \textbf{Step 1:} Play on $v=1$.
        $c_1 \to -c_1 + n_{1,2}c_2 + 1 = -0 + 1(0) + 1 = 1$.
        The new configuration is $c_1=(1,0)$.
        Now, vertex 1 is happy ($1 \not< 1/2(0+1)$), but vertex 2 is sad ($0 < 1/2(1+1)$). The only possible move is on 2. Sequence: $(1,2)$.
        
        \item \textbf{Step 2:} Play on $v=2$.
        $c_2 \to -c_2 + n_{2,1}c_1 + 1 = -0 + 1(1) + 1 = 2$.
        The new configuration is $c_2=(1,2)$.
        Vertex 1: $1 < 1/2(2+1)$ is true. Vertex 1 has become sad!
        Vertex 2: $2 \not< 1/2(1+1)$.
        The only possible move is on 1. Sequence: $(1,2,1)$.
        
        \item \textbf{Step 3:} Play on $v=1$.
        $c_1 \to -c_1 + n_{1,2}c_2 + 1 = -1 + 1(2) + 1 = 2$.
        The new configuration is $c_3=(2,2)$.
        Vertex 1: $2 \not< 1/2(2+1)$.
        Vertex 2: $2 \not< 1/2(2+1)$.
        Both vertices are happy. The game ends.
    \end{itemize}
    The sequence of moves was $(i_1, i_2, i_3) = (1, 2, 1)$. The diagram representing the possible configurations of the game we just had is the same as depicted in Figure \ref{fig:A2mod2ver}.
    
    \textbf{Parallel Algebraic Simulation.}
    The aggregate state vector is $\beta = \beta_1 + \beta_2$. The Cartan matrix of $A_2$ is $A = \begin{pmatrix} 2 & -1 \\ -1 & 2 \end{pmatrix}$.
    \begin{itemize}
        \item \textbf{Step 0:} $w_0 = \mathrm{id}$.
        $v_0 = \beta$. Configuration $c_0 = (0,0)$, matches.
        We check the move at $i_1=1$:
        $\langle v_0, \alpha_1^\vee \rangle = \langle \beta_1 + \beta_2, \alpha_1^\vee \rangle = \langle \beta_1, \alpha_1^\vee \rangle + \langle \beta_2, \alpha_1^\vee \rangle = -1 + 0 = -1 < 0$. The move is valid.
        
        \item \textbf{Step 1:} $w_1 = s_1$. Move $i_2=2$.
        $v_1 = s_1(v_0) = s_1(\beta) = \beta + \alpha_1$.
        Configuration $c_1 = (1,0)$, matches.
        We check the move at $i_2=2$:
        $\langle v_1, \alpha_2^\vee \rangle = \langle \alpha_1 + \beta, \alpha_2^\vee \rangle = \langle \alpha_1, \alpha_2^\vee \rangle + \langle \beta, \alpha_2^\vee \rangle = A_{12} - 1 = -1 - 1 = -2 < 0$. The move is valid.
        
        \item \textbf{Step 2:} $w_2 = s_2 s_1$. Move $i_3=1$.
        $v_2 = s_2(v_1) = v_1 - \langle v_1, \alpha_2^\vee \rangle \alpha_2 = (\alpha_1+\beta) - (-2)\alpha_2 = \alpha_1 + 2\alpha_2 + \beta$.
        Configuration $c_2 = (1,2)$, matches.
        We check the move at $i_3=1$:
        $\langle v_2, \alpha_1^\vee \rangle = \langle \alpha_1 + 2\alpha_2 + \beta, \alpha_1^\vee \rangle = \langle \alpha_1, \alpha_1^\vee \rangle + 2\langle \alpha_2, \alpha_1^\vee \rangle + \langle \beta, \alpha_1^\vee \rangle = A_{11} + 2A_{21} - 1 = 2 + 2(-1) - 1 = -1 < 0$. The move is valid.
        
        \item \textbf{Step 3:} $w_3 = s_1 s_2 s_1$.
        $v_3 = s_1(v_2) = v_2 - \langle v_2, \alpha_1^\vee \rangle \alpha_1 = (\alpha_1 + 2\alpha_2 + \beta) - (-1)\alpha_1 = 2\alpha_1+2\alpha_2+\beta$.
        The final configuration is $c_3=(2,2)$, matches. The game ends.
    \end{itemize}
    
    \textbf{Connection with the Weyl Group.}
    The sequence of moves $(1,2,1)$ gave us the element $w = s_1s_2s_1 \in W(A_2)$. This is a reduced expression for the longest element of the Weyl group, $w_0$. The theorem predicts that this sequence of moves must correspond to an element in $W(A_2)$, and indeed it has. The inversion set is $\mathcal{I}(w) = \{\alpha_1, s_1(\alpha_2), s_1s_2(\alpha_1)\} = \{\alpha_1, \alpha_1+\alpha_2, \alpha_2\}$, which are precisely the 3 positive roots of $A_2$.
    
\end{example}


\begin{example}[The Game on $B_2$ with both vertices modified]

We now consider the case of the Dynkin diagram of type $B_2$ with total modification. In this system, $\alpha_1$ is the short simple root and $\alpha_2$ is the long one. The Weyl group $W(B_2)$ is isomorphic to the dihedral group $D_8$ of 8 elements.

\textbf{Game Configuration.}
We modify both vertices, so the set of modified vertices is $I = \{1, 2\}$, and the set of unmodified vertices is $J = \emptyset$. The associated parabolic subgroup is trivial, $W_J = W_\emptyset = \{\mathrm{id}\}$. Therefore, the set of minimal length representatives is the Weyl group itself, $W^J = W(B_2)$.

The game rules are defined on the root diagram of $B_2$, where the double arrow points from vertex 2 (long) to 1 (short):
$$
n_{1,2} = 2 \quad (\text{arrows from 2 to 1}) \quad \text{and} \quad n_{2,1} = 1 \quad (\text{arrows from 1 to 2}).
$$
The unhappiness condition for a vertex $v$ in a configuration $c = c_1\alpha_1 + c_2\alpha_2$ is adapted for $I=\{1,2\}$:
$$
c_v < \dfrac{1}{2} \left( \sum_{u \in N(v)} n_{v,u}c_u + \sum_{p \in I} \delta_{vp} \right).
$$
Applying this to our two vertices:
\begin{itemize}
    \item For Vertex 1 ($v=1$):
    $ c_1 < \dfrac{1}{2} \left( n_{1,2}c_2 + (\delta_{1,1} + \delta_{1,2}) \right) = \dfrac{1}{2}(2c_2 + 1) = c_2 + \dfrac{1}{2} $.
    \item For Vertex 2 ($v=2$):
    $ c_2 < \dfrac{1}{2} \left( n_{2,1}c_1 + (\delta_{2,1} + \delta_{2,2}) \right) = \dfrac{1}{2}(c_1 + 1) $.
\end{itemize}

\textbf{A Game Playthrough.}
We start with the initial configuration $c_0 = (0,0)$.
\begin{itemize}
    \item \textbf{Step 0:} $c_0=(0,0)$. Both vertices are unhappy. We choose to play on 1. Sequence: $(1)$.
    \item \textbf{Step 1:} Play on $v=1$. $c_1 \to -0 + 2(0) + 1 = 1$. New config: $c_1=(1,0)$. Only 2 is unhappy. Sequence: $(1,2)$.
    \item \textbf{Step 2:} Play on $v=2$. $c_2 \to -0 + 1(1) + 1 = 2$. New config: $c_2=(1,2)$. Only 1 is unhappy. Sequence: $(1,2,1)$.
    \item \textbf{Step 3:} Play on $v=1$. $c_1 \to -1 + 2(2) + 1 = 4$. New config: $c_3=(4,2)$. Only 2 is unhappy. Sequence: $(1,2,1,2)$.
    \item \textbf{Step 4:} Play on $v=2$. $c_2 \to -2 + 1(4) + 1 = 3$. New config: $c_4=(4,3)$. Both vertices are happy. The game ends.
\end{itemize}
The sequence of moves $(1,2,1,2)$ has produced configurations $(0,0) \to (1,0) \to (1,2) \to (4,2) \to (4,3)$.

The diagram representing the possible configurations of the game we just had is represented in Figure \ref{fig:kostantb2mod2}.

\begin{figure}[H]
\centering
\begin{tikzpicture}[scale=1.3]

  \tikzstyle{double-arrow}=[double, double distance=1.5pt, -implies, shorten >=4pt, shorten <=4pt]

  \def\xshift{1.4}

  \node[draw,fill=black,shape=circle,scale=0.5] (s1_0) at (0,9) {$\color{white}{1}$};
  \node[draw,fill=black,shape=circle,scale=0.5] (s2_0) at (\xshift,9) {$\color{white}{1}$};
  \node[draw,fill=white,shape=circle,scale=0.5] (v1_0) at (0,8) {\phantom{$0$}};
  \node[draw,fill=white,shape=circle,scale=0.5] (v2_0) at (\xshift,8) {\phantom{$0$}};
  \draw (s1_0) -- (v1_0);
  \draw (s2_0) -- (v2_0);
  \draw[double-arrow] (v2_0) -- (v1_0);
  \draw[->] (\xshift/2, 7.8) -- (\xshift/2, 7.2);

  \node[draw,fill=black,shape=circle,scale=0.5] (s1_1) at (0,7) {$\color{white}{1}$};
  \node[draw,fill=black,shape=circle,scale=0.5] (s2_1) at (\xshift,7) {$\color{white}{1}$};
  \node[draw,fill=white,shape=circle,scale=0.5] (v1_1) at (0,6) {$1$};
  \node[draw,fill=white,shape=circle,scale=0.5] (v2_1) at (\xshift,6) {\phantom{$0$}};
  \draw (s1_1) -- (v1_1);
  \draw (s2_1) -- (v2_1);
  \draw[double-arrow] (v2_1) -- (v1_1);
  \draw[->] (\xshift/2, 5.8) -- (\xshift/2, 5.2);

  \node[draw,fill=black,shape=circle,scale=0.5] (s1_2) at (0,5) {$\color{white}{1}$};
  \node[draw,fill=black,shape=circle,scale=0.5] (s2_2) at (\xshift,5) {$\color{white}{1}$};
  \node[draw,fill=white,shape=circle,scale=0.5] (v1_2) at (0,4) {$1$};
  \node[draw,fill=white,shape=circle,scale=0.5] (v2_2) at (\xshift,4) {$2$};
  \draw (s1_2) -- (v1_2);
  \draw (s2_2) -- (v2_2);
  \draw[double-arrow] (v2_2) -- (v1_2);
  \draw[->] (\xshift/2, 3.8) -- (\xshift/2, 3.2);

  \node[draw,fill=black,shape=circle,scale=0.5] (s1_3) at (0,3) {$\color{white}{1}$};
  \node[draw,fill=black,shape=circle,scale=0.5] (s2_3) at (\xshift,3) {$\color{white}{1}$};
  \node[draw,fill=white,shape=circle,scale=0.5] (v1_3) at (0,2) {$4$};
  \node[draw,fill=white,shape=circle,scale=0.5] (v2_3) at (\xshift,2) {$2$};
  \draw (s1_3) -- (v1_3);
  \draw (s2_3) -- (v2_3);
  \draw[double-arrow] (v2_3) -- (v1_3);
  \draw[->] (\xshift/2, 1.8) -- (\xshift/2, 1.2);

  \node[draw,fill=black,shape=circle,scale=0.5] (s1_4) at (0,1) {$\color{white}{1}$};
  \node[draw,fill=black,shape=circle,scale=0.5] (s2_4) at (\xshift,1) {$\color{white}{1}$};
  \node[draw,fill=white,shape=circle,scale=0.5] (v1_4) at (0,0) {$4$};
  \node[draw,fill=white,shape=circle,scale=0.5] (v2_4) at (\xshift,0) {$3$};
  \draw (s1_4) -- (v1_4);
  \draw (s2_4) -- (v2_4);
  \draw[double-arrow] (v2_4) -- (v1_4);

\end{tikzpicture}
\caption{The modified Kostant game on $B_2$ altering both vertices.}
\label{fig:kostantb2mod2}
\end{figure}
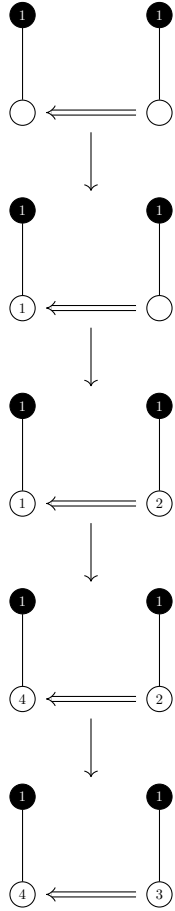

\textbf{Parallel Algebraic Simulation.}
To simulate the game on the \textbf{root} diagram of $B_2$, we must use the \textbf{algebraic} engine of the dual system, $C_2$. In $C_2$, $\alpha_1$ is long and $\alpha_2$ is short. The Cartan matrix of $C_2$ is $A(C_2) = \begin{pmatrix} 2 & -1 \\ -2 & 2 \end{pmatrix}$, so the inner products are $\langle\alpha_1,\alpha_2\rangle = -1$ and $\langle\alpha_2,\alpha_1\rangle = -2$.

The aggregate state vector is $\beta = \beta_1 + \beta_2$. The action of the simple reflections on $\beta$ is:
$$ s_1(\beta) = s_1(\beta_1+\beta_2) = (\beta_1+\alpha_1) + \beta_2 = \beta+\alpha_1 $$
$$ s_2(\beta) = s_2(\beta_1+\beta_2) = \beta_1 + (\beta_2+\alpha_2) = \beta+\alpha_2 $$

\begin{itemize}
    \item \textbf{Step 0:} $w_0 = \mathrm{id}$. $v_0 = \beta$. Configuration $c_0=(0,0)$. Matches.
    
    \item \textbf{Step 1:} Move $i_1=1$. $w_1 = s_1$.
    $v_1 = s_1(v_0) = s_1(\beta) = \beta + \alpha_1$. Configuration $c_1=(1,0)$. Matches.

    \item \textbf{Step 2:} Move $i_2=2$. $w_2 = s_2s_1$.
    $\langle v_1, \alpha_2 \rangle = \langle \alpha_1+\beta, \alpha_2 \rangle = \langle\alpha_1,\alpha_2\rangle+\langle\beta_1,\alpha_2\rangle+\langle\beta_2,\alpha_2\rangle = -1+0-1 = -2$.
    $v_2 = v_1 - (-2)\alpha_2 = (\alpha_1+\beta) + 2\alpha_2 = \alpha_1+2\alpha_2+\beta$. Configuration $c_2=(1,2)$. Matches.
    
    \item \textbf{Step 3:} Move $i_3=1$. $w_3 = s_1s_2s_1$.
    $\langle v_2, \alpha_1 \rangle = \langle \alpha_1+2\alpha_2+\beta, \alpha_1 \rangle = \langle\alpha_1,\alpha_1\rangle+2\langle\alpha_2,\alpha_1\rangle+\langle\beta,\alpha_1\rangle = 2+2(-2)+(-1) = -3$.
    $v_3 = v_2 - (-3)\alpha_1 = (\alpha_1+2\alpha_2+\beta) + 3\alpha_1 = 4\alpha_1+2\alpha_2+\beta$. Configuration $c_3=(4,2)$. Matches.

    \item \textbf{Step 4:} Move $i_4=2$. $w_4 = s_2s_1s_2s_1$.
    $\langle v_3, \alpha_2 \rangle = \langle 4\alpha_1+2\alpha_2+\beta, \alpha_2 \rangle = 4\langle\alpha_1,\alpha_2\rangle+2\langle\alpha_2,\alpha_2\rangle+\langle\beta,\alpha_2\rangle = 4(-1)+2(2)+(-1) = -1$.
    $v_4 = v_3 - (-1)\alpha_2 = (4\alpha_1+2\alpha_2+\beta) + \alpha_2 = 4\alpha_1+3\alpha_2+\beta$. Configuration $c_4=(4,3)$. Matches.
\end{itemize}
The algebraic simulation, using the correct engine of $C_2$, perfectly reproduces the game.

\textbf{Connection with the Weyl Group.}
The sequence of moves $(i_1, i_2, i_3, i_4) = (1, 2, 1, 2)$ corresponds, by definition of the theorem, to the Weyl group element $w_4 = s_{i_4} s_{i_3} s_{i_2} s_{i_1}$, which is
$$ w = s_2 s_1 s_2 s_1. $$
This is a reduced expression for the longest element of the Weyl group $W(B_2)$, denoted $w_0$. Since $I=\{1,2\}$, we have $J=\emptyset$ and $W^J = W(B_2)$, so the theorem predicts that a complete game will construct a reduced expression of the element of maximum length of the entire group, which has been verified. The length of the word is 4, which coincides with the number of positive roots in $R^+(B_2)$:
$$ R^+(B_2) = \{\alpha_1, \alpha_2, \alpha_1+\alpha_2, 2\alpha_1+\alpha_2 \}. $$
The game has generated one of the two possible reduced expressions for $w_0$, confirming the theory precisely.
\end{example}

This last example reveals something additional: in the case where we play on Dynkin graphs modifying all vertices, an interesting connection is also revealed, as follows.

\begin{corollary}[Case of all vertices modified]\label{cor:full-modification}
If $J=\emptyset$, i.e., if all vertices of the Dynkin diagram are modified, then
$$
W^J=W.
$$
Therefore, valid sequences of moves correspond bijectively to the reduced expressions of all elements of $W$. 
\end{corollary}

\begin{proof}
If $J=\emptyset$, then $W_J=\{e\}$. Thus $W/W_J=W$, and the minimal length representatives are all elements of $W$. 
The correspondence established in Theorem \ref{theorem:3.19general} then guarantees that every element of $W$ is described by valid move sequences of the modified game.
\end{proof}

Corollary \ref{cor:full-modification} shows that the Kostant game modified at all vertices recovers the entire combinatorics of reduced expressions of the full Weyl group. 
In particular, this allows visualizing the graph of reduced words: the vertices of the graph correspond to reduced expressions of a given element $w\in W$, and the edges correspond to the relations allowing passage from one reduced expression to another. 
From the game perspective, these edges are interpreted as distinct sequences of moves leading to the same final state.

Let us now examine the consequences and phenomena of this corollary, i.e., those arising from taking
$$
I=S,\qquad J=\varnothing,
$$
by placing sources at all simple vertices of the Dynkin diagram (with the modified game conventions already introduced, particularly assuming that for every vertex $j$ we have $\sum_{p\in S} m_{p j}>0$). 

With $I=S$ and the positivity of the columns $\sum_{p\in S}m_{p j}$, for any positive root $\alpha\in\Phi^+$ the quantity
$$
K(\alpha) \;=\; -\langle \beta,\alpha^\vee\rangle
\;=\; \sum_{j} k_j\Big(\sum_{p\in S} m_{p j}\Big)
\qquad(\alpha=\sum_j k_j\alpha_j,\ k_j\ge 0)
$$
is strictly positive. Therefore, any inversion (any $\tilde\alpha_\lambda$ appearing in the process) can serve to fire, and the game can realize sequences corresponding to all reduced expressions in $W$. In practical terms:

\begin{itemize}
  \item The game constitutes a uniform device for generating all reduced words of $W$.
  \item The local relations of the game (firings at adjacent vertices, effects on neighbor coefficients) exactly reproduce the relations defining the group.
\end{itemize}

\textbf{Family $A_n$}
In type $A_n$, the Weyl group is
$$
W(A_n)\cong S_{n+1},
$$
generated by adjacent transpositions $s_1,\dots,s_n$ with the braid relation $s_i s_{i+1} s_i = s_{i+1} s_i s_{i+1}$ and commutations $s_i s_j = s_j s_i$ for $|i-j|>1$. By placing sources at all vertices:

\begin{itemize}
  \item The theorem implies an exact correspondence between valid sequences and reduced decompositions of permutations into adjacent transpositions.
  \item The game offers a constructive procedure to enumerate and traverse the reduced decompositions of a $w\in S_{n+1}$, and in particular for the longest element $w_0$.
  \item Braid transformations are seen in the game as local changes in the firing sequence; distant commutations correspond to firings affecting non-adjacent and independent vertices.
\end{itemize}

\textbf{Families $B_n$ and $C_n$ (hyperoctahedral)}
In types $B_n$ and $C_n$, the Weyl group is the hyperoctahedral group (symmetries of a cube / signed permutations). Braid relations include edges of multiplicity 2 between roots of different lengths; in the game, this is reflected in coefficients $n_{ij}$ greater than 1 and in update rules with weighted sums. With $I=S$ and positive columns, the game generates all reduced expressions of $W(B_n)$ / $W(C_n)$. However, the structure of commutation classes and the number of reduced decompositions for a given element tends to be more complex than in type $A$.

\textbf{Family $D_n$}
For $D_n$ (Weyl group of type $\mathrm{SO}_{2n}$), two branched vertices appear in the final part of the diagram. The game with $I=S$ still produces all reduced words. The presence of the diagram bifurcation translates into symmetries and specific \emph{commutation class} structures (e.g., elements changing parity at certain positions).

\textbf{Exceptional types ($E_6,E_7,E_8,F_4,G_2$)}
For exceptional types, the formal statement is the same: with $I=S$ and positive columns, the game generates all reduced expressions. However, the concrete combinatorics (number of reduced words of an element, commutation class structure, etc.) is generally more complicated and requires specialized treatment or computer-aided calculation.

\subsection{Applications of the modified game}
\label{sec3.3}

Below we will see some repercussions of the characterization of the minimal length representatives of Weyl group quotients with parabolic subgroups described previously via the game. Some consequences have already been referenced before, but others will be seen to be unexpected and novel.

\subsection{Counting roots from the game}

One of the most elegant applications of the modified game is its ability to calculate sums of specific subsets of positive roots. We now adapt the key results of the original article \cite{CaviedesCastro2022} to our generalized framework.

First, we establish that our generalized game is always finite and deterministic, a crucial property inherited from the structure of Weyl groups.

\begin{corollary}[Generalization of Theorem 3.12 of \cite{CaviedesCastro2022}]
    \label{cor:3.12gen}
    The modified Kostant game on a set of vertices $I \subset S$ of a Dynkin diagram $\Gamma$ always terminates, and it does so in a unique final configuration, which we denote by $c_I$.
\end{corollary}
\begin{proof}
    The proof relies directly on the bijection of Theorem \ref{theorem:3.19general}. Any game constructs a reduced expression of an element in $W^J$ ($J=S\setminus I$). Since $W^J$ is finite, it possesses a unique element of maximum length, $w_I$. By Theorem \ref{thm:completacion_w0j}, every game can be extended until reaching the state associated with $w_I$, at which point the game ends. The uniqueness of $w_I$ guarantees the uniqueness of the final configuration $c_I = w_I(\beta) - \beta$.
\end{proof}

With uniqueness guaranteed, we can explore the relationship between these final configurations and the global structure of the root system. The original article \cite{CaviedesCastro2022} presents a remarkable identity connecting the simplest game plays with the sum of all positive roots. Our generalized framework allows us not only to state but also to prove this result.

First, we need a key structural result that partitions the set of positive roots.

\begin{corollary}[Partition of the set of positive roots]
    \label{cor:decomp-phi-plus}
    Let $\Phi^+$ be the set of all positive roots of the system. For each $j \in S$, let $w_{\{j\}}$ be the unique element of maximum length in $W^{S\setminus\{j\}}$. Then, $\Phi^+$ is the \textbf{disjoint union} of the inversion sets of each of these elements:
    $$\Phi^+ = \bigsqcup_{j \in S} \mathcal{I}(w_{\{j\}})$$
\end{corollary}

\begin{proof}
    The proof relies directly on the results established in Section \ref{Chapter1} regarding roots and parabolic subgroups.
    
    First, by Lemma \ref{lemm:raiconj}, we know that for every positive root $\alpha \in \Phi^+$, there exists at least one pair $(w, \alpha_j)$ with $w \in W$ and $\alpha_j \in \Delta$ such that $\alpha = w(\alpha_j)$. This guarantees that the union of the orbits of all simple roots covers $\Phi^+$.
    
    The crucial point is to demonstrate that each root $\alpha$ belongs to the orbit of a unique simple root under the action of a maximal parabolic subgroup. That is, for each $\alpha \in \Phi^+$, there exists a unique $j \in S$ such that $\alpha$ can be written as $v(\alpha_j)$ for some $v \in W_{S\setminus\{j\}}$.
    
    The existence and uniqueness of this representation is a direct consequence of the parabolic decomposition theorem (\ref{thm:descomposicion_parabolica}). 

    Once established that each $\alpha \in \Phi^+$ belongs to the orbit of one and only one simple root $\alpha_j$ under the action of $W_{S\setminus\{j\}}$, the rest of the argument follows. This unique membership means that $\alpha$ cannot be expressed as a linear combination of simple roots from $S \setminus \{j\}$; in other words:
    $$ \alpha \in \Phi^+ \setminus \Phi^+_{S\setminus\{j\}} $$
    for a unique $j \in S$.
    
    As we have already established in the theory, this set is precisely the inversion set of the maximum length element of the corresponding quotient:
    $$ \mathcal{I}(w_{\{j\}}) = \Phi^+ \setminus \Phi^+_{S\setminus\{j\}} $$
    Since each $\alpha \in \Phi^+$ belongs to exactly one of these sets $\mathcal{I}(w_{\{j\}})$, it is concluded that the union not only covers $\Phi^+$, but is also disjoint.
\end{proof}

This result is the key to proving the counting theorem. It tells us that we can reconstruct the sum of all positive roots simply by summing the sets of roots generated by each simple game.

\begin{theorem}[Root Counting Identity]
    \label{theorem:3.14general}
    Let $c_{\{j\}} = \sum_{k \in S} h_k(\{j\}) \alpha_k$ be the final configuration of the game modified only at vertex $j$. Then, the sum of all positive roots of the system can be expressed as:
    $$ \sum_{\alpha \in \Phi^+} \alpha \;=\; \sum_{j \in S} c_{\{j\}} = \sum_{k \in S} \left( \sum_{j \in S} h_k(\{j\}) \right) \alpha_k. $$
\end{theorem}
\begin{proof}
    The proof is a consequence of combining our theoretical framework with the previous corollary.
    \begin{enumerate}
        \item By Corollary \ref{cor:decomp-phi-plus}, we can decompose the sum over all positive roots into sums over the disjoint inversion sets:
        $$ \sum_{\alpha \in \Phi^+} \alpha = \sum_{j \in S} \left( \sum_{\alpha \in \mathcal{I}(w_{\{j\}})} \alpha \right). $$
        \item Now, we need an identity for the sum of roots in a single inversion set. It can be shown that for simply-laced systems, this sum is equal to the final configuration of the game:
        $$ \sum_{\alpha \in \mathcal{I}(w_{\{j\}})} \alpha = c_{\{j\}}. $$
        \item Substituting this result into the first equation, we obtain the theorem for the simply-laced case:
        $$ \sum_{\alpha \in \Phi^+} \alpha = \sum_{j \in S} c_{\{j\}}. $$
    \end{enumerate}
    The proof for the general multiply-laced case is more technical but relies on the same partition principle.
\end{proof}

\begin{example}[The case of $A_4$]
    We illustrate Theorem \ref{theorem:3.14general} with an example. Let $\Gamma = A_4$. We number the vertices from left to right in increasing order. Let $h_i+1$, for $i=1,\dots,4$, be the height of the unique final configuration of the modified Kostant game at vertex $i$.
    
    By the symmetry of the diagram, we have $h_1 = h_4$ and $h_2 = h_3$. We play the games for vertices 1 and 2, whose configurations were seen before and are illustrated in Figure \ref{fig:a4_mod2vértices}. From the final configurations obtained in said games, it is concluded that:
    $$ h_1 = h_4 = 4 \quad \text{and} \quad h_2 = h_3 = 6. $$
    Theorem \ref{theorem:3.14general} then states that the sum of the positive roots of $A_4$ is equal to:
    $$ \sum_{\alpha \in \Phi^+} \alpha = 4\alpha_1 + 6\alpha_2 + 6\alpha_3 + 4\alpha_4. $$
\end{example}

Finally, we can connect these results with the case of the game modified at all vertices ($I=S$), which, as seen in Corollary \ref{cor:full-modification}, corresponds to the construction of the full Weyl group, $W$. 

For simply-laced systems, the identity of Theorem \ref{theorem:3.14general} allows us to establish a remarkable synthesis. On one hand, we know that the final configuration of the game with all vertices modified, $c_S$, corresponds to the sum of all positive roots:
$$ c_S = \sum_{\alpha \in \mathcal{I}(w_0)} \alpha = \sum_{\alpha \in \Phi^+} \alpha. $$
On the other hand, the same theorem tells us that this total sum of roots can be obtained by summing the final configurations of each individual game:
$$ \sum_{\alpha \in \Phi^+} \alpha = \sum_{j \in S} c_{\{j\}}. $$
Combining both identities, we arrive at the conclusion:
$$ c_S = \sum_{j \in S} c_{\{j\}}. $$
This reveals a profound and computationally powerful property of our game: the final configuration of the most complex game (modifying all vertices) can be recovered simply by vectorially summing the final configurations of the simplest games, where only one vertex is modified at a time.

\subsection{Resolution of the Mukai conjecture in the symplectic case}
\label{sec3.3.3}

In algebraic geometry, a Fano variety $X$ is a complex projective variety whose anticanonical class $-K_X$ is `ample'. Two fundamental invariants are defined: the Picard number $\rho_X = \mathrm{rk} (\mathrm{Pic}(X))$ and the pseudo-index $\iota_X$, which is the minimum integer $-K_X \cdot C$ where $C$ runs over all rational curves in $X$. The Mukai conjecture, posed in the context of Fano varieties in \cite{Mukai1988}, asserts that for any Fano variety of dimension $n$, the inequality
$$ \rho_X (\iota_X - 1) \leq n $$
holds, with equality only in very special cases (products of projective spaces). This conjecture summarizes the idea that the geometry of $X$ is strongly constrained by the interaction between $\rho_X$ and $\iota_X$.

The article \cite{CaviedesCastro2022} seeks to formulate and verify an analogue of this inequality in the symplectic world. Given a compact and monotone symplectic space $(M,\omega)$, the counting of holomorphic sections (typical of algebraic geometry) is replaced by indices of elliptic operators associated with formal line bundles. Thus, a ``generalized Hilbert polynomial'' $H$ is constructed whose coefficients are topological indices. The analysis of the zeros of $H$ is the key to establishing numerical inequalities analogous to Mukai's.

The main problem consists of identifying these zeros. For generalized flag varieties, the authors introduce a variant of the Kostant game, precisely the one seen in Section \ref{sec3.1}. As we know, the original game, defined on a Dynkin diagram, allows producing the highest root of a system from chip redistribution rules. In the modified version we worked on before, an extra vertex $\bar j$ connected to a node $j$ of the diagram is added, and new rules are imposed which, by Corollary \ref{cor:3.12gen}, guarantee the existence and uniqueness of the final configuration. This version serves to index and describe the linear factors appearing in the factorization of the polynomial $H$ in the case of flags.

Among the main results are factorization theorems ensuring that $H$ decomposes into linear factors determined precisely by the configurations of the modified Kostant game. With this explicit control over the zeros, the authors can apply a general criterion leading to establishing Mukai-type inequalities. In particular, it is proven that in symplectic flag varieties, the number of independent parameters in $H^2$ and the dimension satisfy the same relationship predicted by the Mukai conjecture in the algebraic case.

In summary, the work of \cite{CaviedesCastro2022} connects three ideas: (i) the Mukai conjecture as a constraint on Fanos \cite{Mukai1988}, (ii) the substitution of sections for indices in symplectic geometry, and (iii) the use of the modified Kostant game to control the zeros of a generalized Hilbert polynomial. This approach demonstrates how a combinatorial tool apparently foreign to symplectic geometry turns out to be decisive for verifying partial versions of the conjecture in a particular context.

\subsection{Regularity of the language of reduced words}

In computation theory, it is a well-known fact that regular languages—that is, languages that can be formed simply from basic languages by unions, concatenations, and Kleene stars—are determined algorithmically by means of deterministic finite automata. This is the so-called Kleene's theorem, which states that a language is regular if and only if it is accepted by a deterministic finite automaton. On the other hand, within the theory of Coxeter groups, it is already known that the language of reduced words in Weyl groups is a regular language (see \cite[\S 4.8]{BjornerBrenti2005}). In this section, we will give an alternative proof provided by Theorem \ref{theorem:3.19general} as follows.

\begin{theorem}[Regularity of the Coset Representative Language]
	Let $W$ be a finite Weyl group with a set of generators $S = \{s_1, \dots, s_n\}$, and let $W_J$ be a standard parabolic subgroup for some $J \subseteq S$. Let $W^J$ denote the set of minimal length representatives for the right cosets in $W/W_J$. The language $\mathcal{L}(W^J)$, consisting of all reduced words over the alphabet $S$ representing the elements of $W^J$, is a regular language.
\end{theorem}

\begin{proof}
	The proof is constructive. We will prove the theorem by constructing a Deterministic Finite Automaton (DFA), denoted as $M$, which recognizes precisely the language $\mathcal{L}(W^J)$. The construction relies on the bijection between the valid move sequences of the modified Kostant game and the words in $\mathcal{L}(W^J)$, established in Theorem \ref{theorem:3.19general}.
	
	We define the DFA $M = (Q, \Sigma, \delta, q_0, F)$ as follows:
	
	\begin{itemize}
		\item \textbf{States ($Q$):} The set of states is the set of all \textbf{configurations} $c$ reachable in the modified Kostant game (associated with $W/W_J$), plus a non-accepting ``trap'' (or sink) state, $q_{\text{trap}}$.
		$$ Q = \{ c_w \mid w \text{ is a valid sequence of moves} \} \cup \{q_{\text{trap}}\} $$
		Where $c_w$ is the configuration reached after the sequence $w$. Since $W^J$ is finite, the number of valid sequences is finite (the game always terminates by Corollary \ref{cor:3.12gen}), and thus $Q$ is a finite set.
		
		\item \textbf{Alphabet ($\Sigma$):} The alphabet of the automaton is the set of simple generators of $W$.
		$$ \Sigma = S = \{s_1, \dots, s_n\} $$
		
		\item \textbf{Initial State ($q_0$):} The initial state is the game configuration before any move, $c_{\varepsilon}$, corresponding to the empty word $\varepsilon$ (where only the source vertices have chips).
		
		\item \textbf{Transition Function ($\delta$):} The transition function $\delta: Q \times \Sigma \to Q$ is defined directly from the game rules. For a valid configuration $c \in Q \setminus \{q_{\text{trap}}\}$ and a generator $s_i \in \Sigma$:
		$$ \delta(c, s_i) =
		\begin{cases}
			c' & \text{if vertex } i \text{ is sad in } c \text{ (valid move)}, \\
			q_{\text{trap}} & \text{if vertex } i \text{ is not sad in } c \text{ (invalid move)}.
		\end{cases}
		$$
		Here, $c'$ denotes the new configuration obtained after firing vertex $i$. For the trap state, the transition is absorbing: $\delta(q_{\text{trap}}, s_i) = q_{\text{trap}}$ for all $s_i \in \Sigma$. The trap state captures any attempt to make a move forbidden by the game rules.
		
		\item \textbf{Accepting States ($F$):} This is the crucial point. A configuration state $c_w$ is an accepting state if the word $w$ leading to it is a reduced word for an element in $W^J$.
		$$ F = \{ c_w \in Q \mid \text{the word } w \text{ represents an element of } W^J \} $$
		By Theorem \ref{theorem:3.19general}, \emph{every} sequence of valid moves $w$ produces an element of $W^J$. Therefore, any reachable valid configuration is an accepting state. Thus, the set of accepting states is simply the entire set of states except the trap: 
		$$F = Q \setminus \{q_{\text{trap}}\}.$$
	\end{itemize}
	
	Now, we must demonstrate that the language accepted by this automaton, $L(M)$, is exactly $\mathcal{L}(W^J)$.
	
	($\supseteq$) Let $w \in \mathcal{L}(W^J)$. By definition, $w$ is a reduced word for an element in $W^J$. By Theorem \ref{theorem:3.19general}, $w$ must correspond to a sequence of valid moves in the modified Kostant game. By the definition of our transition function $\delta$, when processing the word $w$ from the initial state $q_0 = c_{\varepsilon}$, each transition will correspond to a valid move, so the automaton will never fall into the state $q_{\text{trap}}$ and will end in the configuration state $c_w$. Since $w$ represents an element of $W^J$, we have $c_w \in F$. Therefore, $M$ accepts the word $w$, and we conclude that $\mathcal{L}(W^J) \subseteq L(M)$.
	
	($\subseteq$) Let $w$ be a word accepted by $M$. By the definition of acceptance in a DFA, this means that the sequence of transitions for $w$ starting from $q_0$ does not enter $q_{\text{trap}}$ and ends in a state $c_w \in F$.
	The fact that the sequence does not enter the trap means that at each step a valid transition was made according to the game rules (always playing on sad vertices). By Theorem \ref{theorem:3.19general}, this implies that $w$ is a reduced word for an element in $W^J$. Therefore, $w \in \mathcal{L}(W^J)$, and we conclude that $L(M) \subseteq \mathcal{L}(W^J)$.
	
	Having demonstrated both inclusions, $L(M) = \mathcal{L}(W^J)$. Since we have explicitly constructed a Deterministic Finite Automaton recognizing $\mathcal{L}(W^J)$ based on the game dynamics, it is concluded by Kleene's theorem that the language is regular.
\end{proof}

As an immediate consequence of this general construction, we can recover the classical result for the full group by considering the particular case where the parabolic subgroup is trivial.

\begin{corollary}[Recovery of the classical result]
	In the particular case where all vertices of the Dynkin diagram are modified (i.e., $I=S$ and $J=\emptyset$), the set of representatives $W^J$ coincides with the full group $W$. Consequently, the automaton constructed via the modified Kostant game recognizes the language $\mathcal{L}(W)$ of all reduced words of the group, thus recovering the classical result of Björner and Brenti \cite[\S 4.8]{BjornerBrenti2005} via a purely combinatorial construction.
\end{corollary}

The previous proof constructs the automaton $M$ using the game configurations ($c_w$) as states. By Theorem \ref{theorem:3.19general}, there exists a bijection between valid sequences of moves and the elements of $W^J$. This induces an isomorphism between the configuration automaton $M$ and an equivalent automaton $M'$ whose states are the elements of the Weyl group themselves, $w \in W$. In this automaton $M'$, a transition $w \xrightarrow{s_i} ws_i$ is valid if and only if $\ell(ws_i) > \ell(w)$, which is the condition that the word remains reduced.

For illustrative purposes in the following example, we will construct this isomorphic automaton $M'$, as its visualization with group elements as state labels is more direct and informative.

\begin{example}[Construction of the DFA for $\mathcal{L}(W^J)$ in type $A_2$ with $J=\{s_1\}$]
	Consider the Weyl group $W = A_2$, with generators $S = \{s_1, s_2\}$ and the braid relation $(s_1s_2)^3 = e$, or equivalently, $s_1s_2s_1 = s_2s_1s_2$. The group has 6 elements:
	
	$$
	W = \{e, s_1, s_2, s_1s_2, s_2s_1, s_1s_2s_1 \} 
	$$
	Let us choose the subset $J = \{s_1\}$, which defines the parabolic subgroup $W_J = \{e, s_1\}$.
	
	Our goal is to construct the DFA that recognizes the language $\mathcal{L}(W^J)$, that is, the language of reduced words for the minimal length representatives of the quotient $W/W_J$.
	
	\textbf{1. Identify Representatives ($W^J$).}
	An element $w \in W$ belongs to $W^J$ if and only if its length increases when multiplied on the right by any generator in $J$. In this case, $w \in W^{\{s_1\}}$ if and only if $\ell(ws_1) > \ell(w)$. We check each element:
	\begin{itemize}
		\item $e$: $\ell(e \cdot s_1) = 1 > \ell(e) = 0$. \textbf{Yes.}
		\item $s_1$: $\ell(s_1 \cdot s_1) = 0 < \ell(s_1) = 1$. No.
		\item $s_2$: $\ell(s_2 \cdot s_1) = 2 > \ell(s_2) = 1$. \textbf{Yes.}
		\item $s_1s_2$: $\ell(s_1s_2 \cdot s_1) = 3 > \ell(s_1s_2) = 2$. \textbf{Yes.}
		\item $s_2s_1$: $\ell(s_2s_1 \cdot s_1) = 1 < \ell(s_2s_1) = 2$. No.
		\item $s_1s_2s_1$: $\ell(s_1s_2s_1 \cdot s_1) = 2 < \ell(s_1s_2s_1) = 3$. No.
	\end{itemize}
	The set of minimal length representatives is $W^J = \{e, s_2, s_1s_2\}$. The language we seek is $\mathcal{L}(W^J) = \{\varepsilon, s_2, s_1s_2\}$.
	
	\textbf{2. Construct the Automaton (according to the proof).}
	We construct the DFA $M = (Q, \Sigma, \delta, q_0, F)$ where the states are the elements of $W$.
	
	\begin{itemize}
		\item \textbf{States ($Q$):} $\{e, s_1, s_2, s_1s_2, s_2s_1, s_1s_2s_1, q_{\text{trap}}\}$.
		\item \textbf{Alphabet ($\Sigma$):} $\{s_1, s_2\}$.
		\item \textbf{Initial State ($q_0$):} $e$.
		\item \textbf{Accepting States ($F$):} The elements of $W^J$.
		
		$$
		F = \{e, s_2, s_1s_2\} 
		$$
		\item \textbf{Transition Function ($\delta$):} A transition $w \xrightarrow{s_i} ws_i$ is valid only if $\ell(ws_i) > \ell(w)$. Otherwise, it goes to the trap state.
		\begin{itemize}
			\item $\delta(e, s_1) = s_1$
			\item $\delta(e, s_2) = s_2$
			\item $\delta(s_1, s_2) = s_1s_2$
			\item $\delta(s_2, s_1) = s_2s_1$
			\item $\delta(s_1s_2, s_1) = s_1s_2s_1$
			\item $\delta(s_2s_1, s_2) = s_2s_1s_2 = s_1s_2s_1$
			\item All other transitions (e.g. $\delta(s_1, s_1)$ or $\delta(s_1s_2, s_2)$) go to $q_{\text{trap}}$.
		\end{itemize}
	\end{itemize}
	Observe that the transition $\delta(e, s_2) = s_2$ corresponds in the game to firing vertex 2 (which is sad in the initial configuration), while $\delta(s_1, s_1) = q_{\text{trap}}$ corresponds to attempting to fire vertex 1 when it is already happy.
	
	\textbf{3. Drawing the Automaton.}
	The following diagram represents the constructed automaton. Accepting states are marked with a double circle. For clarity, only transitions not leading to the trap state are shown, plus some examples of transitions that do.
	
	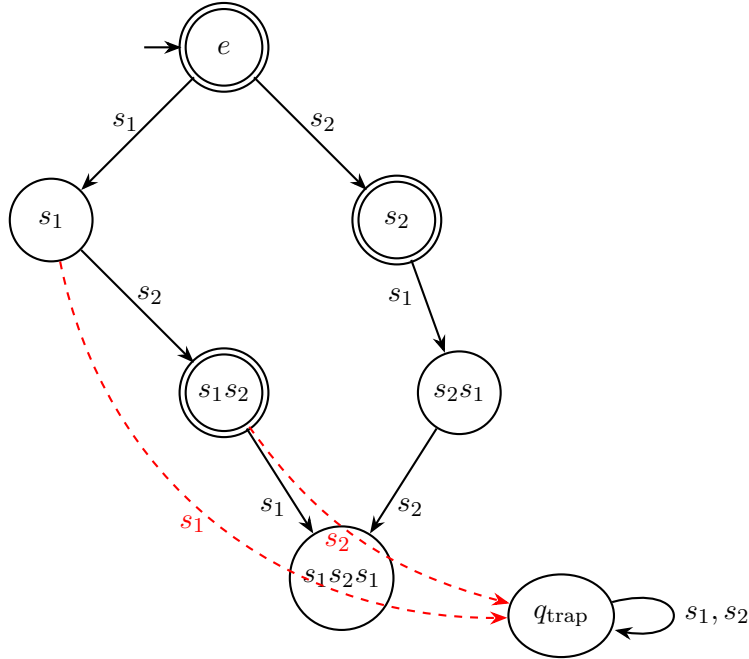
\begin{figure}[H]
		\centering
		\begin{tikzpicture}[
			>=Stealth,
			node distance=1.5cm and 1.5cm,
			auto,
			state/.style={circle, draw, thick, minimum size=1.1cm},
			accepting state/.style={state, double, double distance=1.5pt},
			initial text={},
			every edge/.style={draw, ->, thick}
			]
			\node[accepting state, initial] (e) {$e$};
			\node[state] (s1) [below left=of e] {$s_1$};
			\node[accepting state] (s2) [below right=of e] {$s_2$};
			
			\node[accepting state] (s1s2) [below right=of s1] {$s_1s_2$};
			
			\node[state] (s2s1) [right=2cm of s1s2] {$s_2s_1$};
			
			\path (s1s2.south) -- (s2s1.south) node[midway] (mid) {};
			\node[state] (s1s2s1) [below=1.2cm of mid] {$s_1s_2s_1$};
			
			\node[state, ellipse] (trampa) [right=1.5cm of s1s2s1, yshift=-0.5cm] {$q_{\text{trap}}$};
			
			\path
			(e) edge node[left, pos=0.4] {$s_1$} (s1)
			(e) edge node[right, pos=0.4] {$s_2$} (s2)
			(s1) edge node[right, pos=0.4] {$s_2$} (s1s2)
			(s2) edge node[left, pos=0.4] {$s_1$} (s2s1) 
			(s1s2) edge node[left, near end] {$s_1$} (s1s2s1)
			(s2s1) edge node[right, near end] {$s_2$} (s1s2s1);
			
			\path[->, dashed, red, thick]
			(s1) edge[bend right=40] node[left] {$s_1$} (trampa)
			(s1s2) edge[bend right=20] node[below, pos=0.4] {$s_2$} (trampa);
			
			\path (trampa) edge [loop right] node {$s_1, s_2$} (trampa);
			
		\end{tikzpicture}
		\caption{DFA recognizing the language $\mathcal{L}(A_2^{\{s_1\}}) = \{\varepsilon, s_2, s_1s_2\}$.}
		\label{fig:dfa_a2}
	\end{figure}
	
	The automaton accepts the empty word (since the initial state $e$ is accepting), the word $s_2$ (path $e \to s_2$), and the word $s_1s_2$ (path $e \to s_1 \to s_1s_2$). Any other word not in $\mathcal{L}(W^J)$, such as $s_1$ or $s_1s_2s_2$, will be rejected because it either ends in a non-accepting state ($s_1$) or falls into the trap state.
\end{example}

\begin{remark}
	Note that, implicitly, the Diamond/Hexagon Lemma \ref{lemma_diamantehexagono} guarantees the consistency in the functioning of the deterministic finite automaton that recognizes the language of valid sequences of moves. The local confluence ensured by said lemma justifies that the game dynamics can be described by purely local transition rules, thus allowing the acceptance of words that form part of the language.
\end{remark}


\subsection{Construction of Young Tableaux}

The structure of Weyl groups and their quotients, which we have explored through the Kostant game, admits a parallel and remarkably rich description in a different combinatorial language: that of partitions and \textbf{Young tableaux}. Although an exhaustive treatment is beyond the scope of this work, it is instructive to outline this connection, particularly for root systems of type $A_{n-1}$, where the Weyl group $W$ is isomorphic to the symmetric group $S_n$ (see \cite{Fulton1997}).

Consider the symmetric group $W = S_n$, generated by simple transpositions $s_i = (i, i+1)$. A maximal parabolic subgroup $W_J$ is obtained by omitting a single generator, $s_k$. The resulting quotient, $W/W_{S\setminus\{s_k\}}$, has as its minimal length representatives the \textbf{Grassmannian permutations}: those with a single descent at position $k$.

To go a step further, from sets to elements, Rudolf Winkel established in \cite{Winkel1996} a canonical bijection between the reduced expressions of a Grassmannian permutation $w$ and the \textbf{Standard Young Tableaux (SYT)} of the corresponding shape $\lambda(w)$. An SYT is a filling of the Young diagram with numbers that increase strictly in rows and columns. Winkel's bijection provides an explicit map:
$$
\text{Reduced Expressions of } w \quad \longleftrightarrow \quad \text{SYT of shape } \lambda(w)
$$
This correspondence is the bridge established between the algebra of Weyl words and the combinatorics of tableaux.

From Theorem \ref{theorem:3.19general}, which establishes a bijection between game plays and reduced expressions, an alternative perspective emerges. This connection is deeper: the game dynamics not only correspond to the words but essentially construct the associated Young Tableau visually.

The proposed mechanism is as follows: Given an element $w \in W^J$ (a Grassmannian permutation for the quotient $S_n/(S_k \times S_{n-k})$) and its associated Young shape $\lambda(w)$:

\begin{enumerate}
	\item \textbf{A board is set:} The board is the Young diagram of shape $\lambda(w)$, which fits into a $k \times (n-k)$ rectangle.
	
	\item \textbf{Each game builds a tableau:} Each sequence of moves $(i_1, i_2, \dots, i_m)$ corresponding to a reduced expression of $w$ constructs an SYT of shape $\lambda(w)$.
	
	\item \textbf{The filling mechanism:} The construction is step-by-step. At step $j$ of the game (where $1 \le j \le m$), a move is made on vertex $s_{i_j}$. The number $j$ is added to a new box of the tableau according to the following rule, based on the position of $i_j$ relative to the modified vertex $s_k$:
	\begin{itemize}
		\item If $i_j = k$ (playing on the modified vertex), number $j$ is placed on the main diagonal of the tableau.
		\item If $i_j < k$ (playing on a vertex to the left), number $j$ is placed \textbf{below} the main diagonal.
		\item If $i_j > k$ (playing on a vertex to the right), number $j$ is placed \textbf{above} the main diagonal.
	\end{itemize}
	At each step, the number is placed in the first available cell of its region that preserves the row and column growth of the SYT.
\end{enumerate}

In this way, the different valid games that construct the same element $w$ (and therefore end in the same final configuration) correspond to the different ways of filling the Young diagram, that is, to all SYTs of shape $\lambda(w)$.

\begin{example}[The case $S_4 / (S_2 \times S_2)$]
	Consider the quotient for $n=4$ and $k=2$, which corresponds to the game on $A_3$ with vertex $s_2$ modified. Take as an example the element $w$ represented by the reduced word $s_2s_1s_3s_2$. This element is a Grassmannian permutation whose associated Young shape, by the Schubert bijection, is $\lambda=(2,2)$.
	
	Combinatorial theory tells us that there are exactly two SYTs of this shape. On the other hand, the element $w$ admits exactly two reduced expressions, the second being $s_2s_3s_1s_2$ (obtained by the commutation relation $s_1s_3 = s_3s_1$).
	
	By our Theorem \ref{theorem:3.19general}, these two reduced expressions correspond to two valid and distinct sequences of moves in our game. Both games must culminate in the same final configuration, since they represent the same element $w$.
	
	The construction then postulates that these two distinct games build, step by step, the two SYTs of shape $\lambda=(2,2)$. Figure \ref{fig:conjetura_visual_1} explicitly illustrates this correspondence: each path of the game, although arriving at the same destination, traces a different history corresponding to a unique tableau filling.
\end{example}

\begin{figure}[H]
	\centering
	\begin{tikzpicture}[
		vtx/.style={circle, draw, fill=white, thick, minimum size=18pt, font=\small},
		src/.style={vtx, fill=black, text=white},
		arrow/.style={-Stealth, very thick},
		game/.style={scale=0.8, transform shape},
		tableau/.style={scale=0.9, transform shape} 
		]
		
		\node (title1) at (1.5,0) {\bfseries Game for $s_2s_1s_3s_2$};
		\node (title2) at (10.5,0) {\bfseries Game for $s_2s_3s_1s_2$};
		
		\node (g1_0) at (0, -1.5) {
			\begin{tikzpicture}[game]
				\node[vtx] (v1) at (0,0) {0}; \node[vtx] (v2) at (1.5,0) {0}; \node[vtx] (v3) at (3,0) {0};
				\node[src] (s2) at (1.5,1) {1}; \draw (v1)--(v2)--(v3); \draw (s2)--(v2);
			\end{tikzpicture}
		};
		\node[tableau, right=0.2cm of g1_0] (t1_0) {(Empty Tableau)};
		
		\node (g1_1) at (0, -4) {
			\begin{tikzpicture}[game]
				\node[vtx] (v1) at (0,0) {0}; \node[vtx] (v2) at (1.5,0) {1}; \node[vtx] (v3) at (3,0) {0};
				\node[src] (s2) at (1.5,1) {1}; \draw (v1)--(v2)--(v3); \draw (s2)--(v2);
			\end{tikzpicture}
		};
		\node[tableau, right=0.2cm of g1_1] (t1_1) {
			\begin{ytableau} 1 \end{ytableau}
		};
		
		\node (g1_2) at (0, -6.5) {
			\begin{tikzpicture}[game]
				\node[vtx] (v1) at (0,0) {1}; \node[vtx] (v2) at (1.5,0) {1}; \node[vtx] (v3) at (3,0) {0};
				\node[src] (s2) at (1.5,1) {1}; \draw (v1)--(v2)--(v3); \draw (s2)--(v2);
			\end{tikzpicture}
		};
		\node[tableau, right=0.2cm of g1_2] (t1_2) {
			\begin{ytableau} 1 \\ 2 \end{ytableau}
		};
		
		\node (g1_3) at (0, -9) {
			\begin{tikzpicture}[game]
				\node[vtx] (v1) at (0,0) {1}; \node[vtx] (v2) at (1.5,0) {1}; \node[vtx] (v3) at (3,0) {1};
				\node[src] (s2) at (1.5,1) {1}; \draw (v1)--(v2)--(v3); \draw (s2)--(v2);
			\end{tikzpicture}
		};
		\node[tableau, right=0.2cm of g1_3] (t1_3) {
			\begin{ytableau} 1 & 3 \\ 2 \end{ytableau}
		};
		
		\node (g1_4) at (0, -11.5) {
			\begin{tikzpicture}[game]
				\node[vtx] (v1) at (0,0) {1}; \node[vtx] (v2) at (1.5,0) {2}; \node[vtx] (v3) at (3,0) {1};
				\node[src] (s2) at (1.5,1) {1}; \draw (v1)--(v2)--(v3); \draw (s2)--(v2);
			\end{tikzpicture}
		};
		\node[tableau, right=0.2cm of g1_4] (t1_4) {
			\begin{ytableau} 1 & 3 \\ 2 & 4 \end{ytableau}
		};
		\node[below=0.1cm of g1_4, align=center] {\small Final Configuration \\ $c=(1,2,1)$};
		
		\node (g2_0) at (9, -1.5) {
			\begin{tikzpicture}[game]
				\node[vtx] (v1) at (0,0) {0}; \node[vtx] (v2) at (1.5,0) {0}; \node[vtx] (v3) at (3,0) {0};
				\node[src] (s2) at (1.5,1) {1}; \draw (v1)--(v2)--(v3); \draw (s2)--(v2);
			\end{tikzpicture}
		};
		\node[tableau, right=0.2cm of g2_0] (t2_0) {(Empty Tableau)};
		
		\node (g2_1) at (9, -4) {
			\begin{tikzpicture}[game]
				\node[vtx] (v1) at (0,0) {0}; \node[vtx] (v2) at (1.5,0) {1}; \node[vtx] (v3) at (3,0) {0};
				\node[src] (s2) at (1.5,1) {1}; \draw (v1)--(v2)--(v3); \draw (s2)--(v2);
			\end{tikzpicture}
		};
		\node[tableau, right=0.2cm of g2_1] (t2_1) {
			\begin{ytableau} 1 \end{ytableau}
		};
		
		\node (g2_2) at (9, -6.5) {
			\begin{tikzpicture}[game]
				\node[vtx] (v1) at (0,0) {0}; \node[vtx] (v2) at (1.5,0) {1}; \node[vtx] (v3) at (3,0) {1};
				\node[src] (s2) at (1.5,1) {1}; \draw (v1)--(v2)--(v3); \draw (s2)--(v2);
			\end{tikzpicture}
		};
		\node[tableau, right=0.2cm of g2_2] (t2_2) {
			\begin{ytableau} 1 & 2 \end{ytableau}
		};
		
		\node (g2_3) at (9, -9) {
			\begin{tikzpicture}[game]
				\node[vtx] (v1) at (0,0) {1}; \node[vtx] (v2) at (1.5,0) {1}; \node[vtx] (v3) at (3,0) {1};
				\node[src] (s2) at (1.5,1) {1}; \draw (v1)--(v2)--(v3); \draw (s2)--(v2);
			\end{tikzpicture}
		};
		\node[tableau, right=0.2cm of g2_3] (t2_3) {
			\begin{ytableau} 1 & 2 \\ 3 \end{ytableau}
		};
		
		\node (g2_4) at (9, -11.5) {
			\begin{tikzpicture}[game]
				\node[vtx] (v1) at (0,0) {1}; \node[vtx] (v2) at (1.5,0) {2}; \node[vtx] (v3) at (3,0) {1};
				\node[src] (s2) at (1.5,1) {1}; \draw (v1)--(v2)--(v3); \draw (s2)--(v2);
			\end{tikzpicture}
		};
		\node[tableau, right=0.2cm of g2_4] (t2_4) {
			\begin{ytableau} 1 & 2 \\ 3 & 4 \end{ytableau}
		};
		\node[below=0.1cm of g2_4, align=center] {\small Final Configuration \\ $c=(1,2,1)$};
		
		\draw[arrow] (g1_0) -- (g1_1); \draw[arrow] (g1_1) -- (g1_2); \draw[arrow] (g1_2) -- (g1_3); \draw[arrow] (g1_3) -- (g1_4);
		\draw[arrow] (g2_0) -- (g2_1); \draw[arrow] (g2_1) -- (g2_2); \draw[arrow] (g2_2) -- (g2_3); \draw[arrow] (g2_3) -- (g2_4);
		
	\end{tikzpicture}
	\caption{Two valid games on $A_3$ modified at $s_2$. Both correspond to different reduced expressions of the element $w=s_2s_1s_3s_2$ and construct the two SYTs of shape $\lambda=(2,2)$.}
	\label{fig:conjetura_visual_1}
\end{figure}
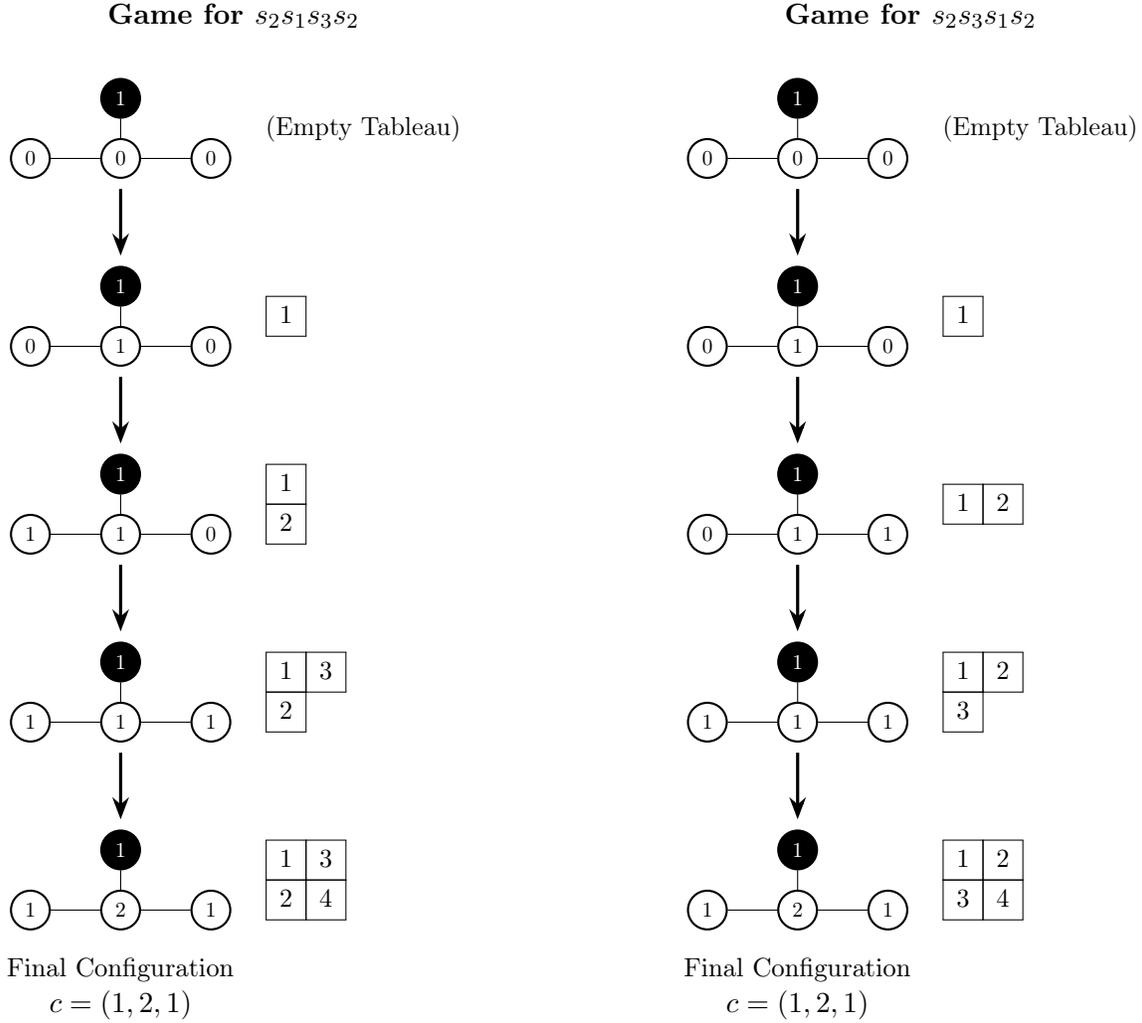

\begin{example}[Game on $A_4$ with the second vertex modified]
	Next, we consider the system of type $A_4$ ($n=5$) with the second vertex modified ($k=2$). The game in this case explores the quotient $S_5 / (S_2 \times S_3)$. The theory predicts that the resulting Standard Young Tableaux (SYT) must fit into a $k \times (n-k) = 2 \times 3$ rectangle.
	
	Figure \ref{fig:conjetura_a4_completa} illustrates a complete game. The game requires 6 moves to reach its terminal state, which is the final configuration $c_{final} = (1,2,2,1)$. The sequence of moves $(2,1,3,2,4,3)$ corresponds to a reduced word for the longest element in $W^J$, and as predicted by the construction, the game builds an SYT that completely fills the $2 \times 3$ rectangle.
\end{example}

\begin{figure}[H]
	\centering
	\begin{tikzpicture}[
		vtx/.style={circle, draw, fill=white, thick, minimum size=18pt, font=\small},
		src/.style={vtx, fill=black, text=white, font=\small},
		arrow/.style={-Stealth, very thick},
		node distance=0.6cm and 0.8cm
		]
		
		\node (game_title) at (1.5, 0) {\large\bfseries Game on $A_4$ ($k=2$)};
		\node (tableau_title) at (9, 0) {\large\bfseries Tableau Construction ($2 \times 3$)};
		\node[draw, dashed, rounded corners, right=0.5cm of tableau_title, yshift=-0.5cm] {
			\tiny\begin{tabular}{l}
				$v_1 \to \text{Diag -1}$ \\ $v_2 \to \text{Diag 0}$ \\
				$v_3 \to \text{Diag +1}$ \\ $v_4 \to \text{Diag +2}$
			\end{tabular}
		};
		
		\node (game0) at (1.5, -2) 
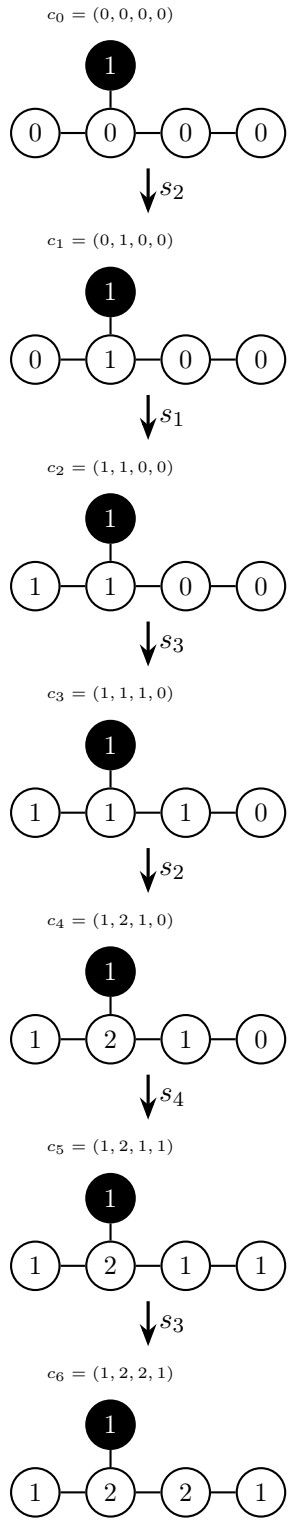
{
			\begin{tikzpicture}
				\node[vtx] (v1) at (0,0) {0}; \node[vtx] (v2) at (1,0) {0};
				\node[vtx] (v3) at (2,0) {0}; \node[vtx] (v4) at (3,0) {0};
				\node[src] (s2) at (1,0.9) {1}; 
				\draw[thick] (v1)--(v2)--(v3)--(v4); \draw[thick] (s2)--(v2);
				\node[above=0.1cm of s2] {\tiny $c_0=(0,0,0,0)$};
			\end{tikzpicture}
		};
		\node (tab0) at (9, -2) {(Empty Tableau)};
		
		\node (game1) [below=of game0] {
			\begin{tikzpicture}
				\node[vtx] (v1) at (0,0) {0}; \node[vtx] (v2) at (1,0) {1};
				\node[vtx] (v3) at (2,0) {0}; \node[vtx] (v4) at (3,0) {0};
				\node[src] (s2) at (1,0.9) {1}; 
				\draw[thick] (v1)--(v2)--(v3)--(v4); \draw[thick] (s2)--(v2);
				\node[above=0.1cm of s2] {\tiny $c_1=(0,1,0,0)$};
			\end{tikzpicture}
		};
		\node (tab1) at (tab0 |- game1) {\ytableausetup{boxsize=1.2em}\begin{ytableau} 1 \end{ytableau}};
		
		\node (game2) [below=of game1] {
			\begin{tikzpicture}
				\node[vtx] (v1) at (0,0) {1}; \node[vtx] (v2) at (1,0) {1};
				\node[vtx] (v3) at (2,0) {0}; \node[vtx] (v4) at (3,0) {0};
				\node[src] (s2) at (1,0.9) {1}; 
				\draw[thick] (v1)--(v2)--(v3)--(v4); \draw[thick] (s2)--(v2);
				\node[above=0.1cm of s2] {\tiny $c_2=(1,1,0,0)$};
			\end{tikzpicture}
		};
		\node (tab2) at (tab0 |- game2) {\ytableausetup{boxsize=1.2em}\begin{ytableau} 1 \\ 2 \end{ytableau}};
		
		\node (game3) [below=of game2] {
			\begin{tikzpicture}
				\node[vtx] (v1) at (0,0) {1}; \node[vtx] (v2) at (1,0) {1};
				\node[vtx] (v3) at (2,0) {1}; \node[vtx] (v4) at (3,0) {0};
				\node[src] (s2) at (1,0.9) {1}; 
				\draw[thick] (v1)--(v2)--(v3)--(v4); \draw[thick] (s2)--(v2);
				\node[above=0.1cm of s2] {\tiny $c_3=(1,1,1,0)$};
			\end{tikzpicture}
		};
		\node (tab3) at (tab0 |- game3) {\ytableausetup{boxsize=1.2em}\begin{ytableau} 1 & 3 \\ 2 \end{ytableau}};
		
		\node (game4) [below=of game3] {
			\begin{tikzpicture}
				\node[vtx] (v1) at (0,0) {1}; \node[vtx] (v2) at (1,0) {2};
				\node[vtx] (v3) at (2,0) {1}; \node[vtx] (v4) at (3,0) {0};
				\node[src] (s2) at (1,0.9) {1}; 
				\draw[thick] (v1)--(v2)--(v3)--(v4); \draw[thick] (s2)--(v2);
				\node[above=0.1cm of s2] {\tiny $c_4=(1,2,1,0)$};
			\end{tikzpicture}
		};
		\node (tab4) at (tab0 |- game4) {\ytableausetup{boxsize=1.2em}\begin{ytableau} 1 & 3 \\ 2 & 4 \end{ytableau}};
		
		\node (game5) [below=of game4] {
			\begin{tikzpicture}
				\node[vtx] (v1) at (0,0) {1}; \node[vtx] (v2) at (1,0) {2};
				\node[vtx] (v3) at (2,0) {1}; \node[vtx] (v4) at (3,0) {1};
				\node[src] (s2) at (1,0.9) {1}; 
				\draw[thick] (v1)--(v2)--(v3)--(v4); \draw[thick] (s2)--(v2);
				\node[above=0.1cm of s2] {\tiny $c_5=(1,2,1,1)$};
			\end{tikzpicture}
		};
		\node (tab5) at (tab0 |- game5) {\ytableausetup{boxsize=1.2em}\begin{ytableau} 1 & 3 & 5 \\ 2 & 4 \end{ytableau}};
		
		\node (game6) [below=of game5] {
			\begin{tikzpicture}
				\node[vtx] (v1) at (0,0) {1}; \node[vtx] (v2) at (1,0) {2};
				\node[vtx] (v3) at (2,0) {2}; \node[vtx] (v4) at (3,0) {1};
				\node[src] (s2) at (1,0.9) {1}; 
				\draw[thick] (v1)--(v2)--(v3)--(v4); \draw[thick] (s2)--(v2);
				\node[above=0.1cm of s2] {\tiny $c_6=(1,2,2,1)$};
			\end{tikzpicture}
		};
		\node (tab6) at (tab0 |- game6) {
			\ytableausetup{boxsize=1.2em}
			\begin{ytableau} 1 & 3 & 5 \\ 2 & 4 & 6 \end{ytableau}
		};
		
		\draw[arrow] (game0) -- node[right, pos=0.5] {$s_2$} (game1);
		\draw[arrow] (game1) -- node[right, pos=0.5] {$s_1$} (game2);
		\draw[arrow] (game2) -- node[right, pos=0.5] {$s_3$} (game3);
		\draw[arrow] (game3) -- node[right, pos=0.5] {$s_2$} (game4);
		\draw[arrow] (game4) -- node[right, pos=0.5] {$s_4$} (game5);
		\draw[arrow] (game5) -- node[right, pos=0.5] {$s_3$} (game6);
		
	\end{tikzpicture}
	\caption{Evolution of a complete game on $A_4$ ($k=2$) for the sequence $(2,1,3,2,4,3)$.}
	\label{fig:conjetura_a4_completa}
\end{figure}

\begin{remark}[The Complete Connection: Game, Shape, and Filling]
	The triple correspondence established between the game, reduced words, and SYTs is even deeper than it appears at first glance. The visual construction described above provides a concrete mechanism for the unified and powerful idea emerging from this research:
	\begin{itemize}
		\item The element of the Weyl quotient ($w \in W^J$) determines the \textbf{shape} ($\lambda$) of the Young Tableau.
		\item The path to construct that element in the game (the sequence of moves) determines the filling of that tableau.
	\end{itemize}
	Therefore, the generalized Kostant game is not just a word generator; it can be seen as a dynamic engine that, through its trajectory and endpoint, completely constructs a fundamental object of algebraic combinatorics, connecting the dynamics of a root system with the static structure of a Standard Young Tableau.
\end{remark}


\subsection{Software Availability}
\label{java}

To verify the results presented in this paper and explore the game dynamics on arbitrary graphs, we developed an open-source Java application. The software allows users to simulate the original Kostant game and its modified version, covering both simply-laced and multiply-laced diagrams (via directed weighted edges), and to visualize node states in real-time. The source code and executable are available at the following GitHub repository:
\begin{center}
	\href{https://github.com/sebascortes15/KostantGame}{\textcolor{red}{https://github.com/sebascortes15/KostantGame}}
\end{center}


\section{Conclusions and future outlook}
\label{sec:conclusions}

At the outset of this journey, we posed the question of whether it was possible to explore the intricate structure of root systems and their Weyl groups through a more constructive and algorithmic approach. This work provides an affirmative answer, not only by utilizing an existing combinatorial tool but by expanding it to reveal unsuspected depth. The research has culminated in the development of a generalized version of the Kostant game, transforming it from a method for enumerating roots into a sophisticated instrument for analyzing the fundamental substructures of Lie theory.

The cornerstone of this work is the proof of a \textbf{bijective correspondence} between the dynamics of the generalized game and the algebra of Weyl groups (Theorem \ref{theorem:3.19general}). We have proven that every sequence of valid moves on our board translates directly into a reduced expression for an element of the quotient $W/W_J$. This bridge is not merely a curiosity, but a functional dictionary that translates intuitive combinatorial actions into rigorous propositions about Weyl quotients, offering a novel method for their study.

The power of this formalism is manifested in its diverse applications. We have demonstrated how the game provides an alternative proof of the regularity of the language of reduced words, explicitly constructing a finite automaton from the game states. Furthermore, we have established a root counting theorem (Theorem \ref{theorem:3.14general}) that generalizes known results, showing that the sum of all positive roots can be recovered by summing the final configurations of the simplest games.

Perhaps the most intriguing connection is the one we have outlined with algebraic combinatorics. We postulated a visual construction linking game plays with the construction of \textbf{Standard Young Tableaux}. The examples developed in this work suggest that each game play not only generates a reduced word but also traces a path that constructs one of the associated SYTs, offering a dynamic realization of Winkel's classic bijection.

Finally, this work validates not only its theoretical relevance but also its practical origin, by clarifying the role of the modified game as a key tool in resolving particular cases of the Mukai conjecture in the symplectic context.

The end of this research is, in reality, the beginning of many others. The future perspectives that open up are vast and promising. A natural extension would be to take this game beyond the limits of finite Dynkin diagrams, to explore if it can shed new light on the infinite structures of Kac-Moody algebras. Likewise, the complete formalization of the Young Tableaux construction and its possible extension to other types of Weyl groups are immediate and fascinating directions for research.

In summary, this work has taken the elegant Kostant game and elevated it to a robust theoretical platform. By establishing its fundamental connection with Weyl group quotients, we have positioned this combinatorial game as a versatile and powerful tool, capable of bridging combinatorics, algebra, geometry, and formal language theory.


\section*{Acknowledgements}

The author would like to express his sincere gratitude to Ph.D. Alexander Caviedes Castro (Facultad de Ciencias, Universidad Nacional de Colombia) for his supervision and guidance throughout the development of this research. This paper is based on the author's master's thesis titled \textit{``Grupos de Weyl y el juego de Kostant''}, submitted to the Department of Mathematics at the Universidad Nacional de Colombia.


\end{document}